\documentclass[a4paper, reqno]{amsart}

\usepackage[utf8]{inputenc}	
\usepackage[T1]{fontenc}	
\usepackage{a4wide}
\usepackage{amsmath}		
\usepackage{amssymb}		
\usepackage{enumerate}	
\usepackage{nicefrac}		
\usepackage{array}			
\usepackage{calc}				
\usepackage{flafter}		
\usepackage{esdiff}
\usepackage{pgffor}   
\usepackage{pgf}
\usepackage{mathabx}
\usepackage{multicol}

\usepackage
[
	pdfauthor   = {Lisa Maria Kreusser},
  pdftitle    = {Title},
  pdfsubject  = {},
  pdfkeywords = {},
  bookmarks=true,
  bookmarksopen=true,
  pagebackref=false,
  hyperindex=true,
  colorlinks=true,
  linkcolor=blue,
  citecolor=blue,
  filecolor=blue,
  urlcolor=blue,
  ]
  {hyperref}						
\usepackage{textcomp}		
\usepackage{color}			
\usepackage{graphicx}		
\usepackage{amscd}			

\usepackage{mathrsfs}		
\usepackage{graphicx}
\usepackage[caption=false]{subfig}
\usepackage[font=footnotesize,labelfont=bf]{caption}

\newcommand\R{\mathbb{R}}

\newcommand\C{\mathbb{C}}

\newcommand\bl{\left(}
\newcommand\br{\right)}

\newcommand*\di{\mathop{}\!\mathrm{d}}

\renewcommand\epsilon{\varepsilon}
\renewcommand\theta{\vartheta}

\newtheoremstyle{mytheoremstyle} 
    {6pt}                    
    {6pt}                    
    {\itshape}                   
    {}						
    {\bf}                   
    {}                          
    {.5em}                       
    {}  
\theoremstyle{mytheoremstyle}
\newtheorem{satz}{Satz}[section]
\newtheorem{lemma}[satz]{Lemma}
\newtheorem{proposition}[satz]{Proposition}
\newtheorem{corollary}[satz]{Corollary}

\newtheorem{definition}[satz]{Definition}
\newtheorem{ass}[satz]{Assumption}
\newtheorem{rem}[satz]{Remark}

\numberwithin{equation}{section}
\usepackage{cleveref}

\title{Pattern formation of a nonlocal, anisotropic interaction model}
\author[M. Burger, B. D\"{u}ring, L. M. Kreusser, P. A. Markowich and C.-B. Sch\"{o}nlieb]{Martin Burger, Bertram D\"{u}ring, Lisa Maria Kreusser, Peter A. Markowich and Carola-Bibiane Sch\"{o}nlieb}

\begin{document}

\begin{abstract}
We consider a class of  interacting particle models with anisotropic, repulsive-attractive interaction forces whose orientations depend on an underlying tensor field. An example of this class of models is the so-called K\"{u}cken-Champod model describing the formation of fingerprint patterns. 
This class of models can be regarded as a generalization of a gradient flow of a nonlocal interaction potential which has a local repulsion and a long-range attraction structure. In contrast to  isotropic interaction models  the anisotropic forces in our class of models cannot be derived from a potential. The underlying tensor field introduces an anisotropy leading to complex patterns  which do not occur in  isotropic models. This anisotropy is characterized by one parameter in the model. We study the variation of this parameter, describing the transition between the isotropic and the anisotropic model, analytically and numerically. We analyze the equilibria of the corresponding mean-field partial differential equation  and investigate  pattern formation numerically in two dimensions by studying the dependence of the parameters in the model on the resulting patterns. 
\end{abstract}

\maketitle

\section{Introduction}

Nonlocal interaction models are mathematical models describing the  collective behavior of  large numbers of individuals where each individual can interact not only with its close neighbors but also with individuals far away. These models serve as basis for biological aggregation and have given us many tools to understand the fundamental behavior of collective motion and pattern formation in nature. For instance, these mathematical models are used to explain the complex phenomena observed in swarms of insects, flocks of birds, schools of fish or colonies of bacteria  \cite{Boi2000163,burger2007aggregation,selfpropelled_particles, selforganization, swarmequilibria, order_of_chaos, scienceanialaggregation, fishbehavior,swarmdynamics, PredictingPatternFormation, stabilityringpatterns, bacterialcolonies, Bertozzi2015, migratorylocusts, Birnir2007, nonlocal_swarm, Blanchet2006, Blanchet2008, cuckersmale, SwarmingPatterns}.

One of the key features of many of these models is the social communication between individuals at different scales which can be described by short- and long-range interactions  \cite{swarmequilibria,nonlocal_swarm,migratorylocusts}. Over large distances the individuals in these models can sense each other via sight, sound, smell, vibrations or other signals. While most models consider only isotropic interactions,  pattern formation in nature is often anisotropic \cite{patternsinnature} and  an anisotropy in the communication between individuals is essential  to describe these phenomena accurately. In this paper, we consider a class of interacting particle models with anisotropic interaction forces. It can be regarded as a generalization of isotropic interaction models. These anisotropic interaction models capture many important swarming behaviors which are neglected in the simplified isotropic interaction model.

Isotropic interaction models with radial interaction potentials can be regarded as the simplest form of interaction models \cite{nonlocalinteraction}. The resulting patterns are found as stationary points of the $N$ particle interaction energy
\begin{align}\label{eq:discreteinteractionenergy}
E(x_1,\ldots,x_N)=\frac{1}{2N^2}\sum_{\substack{j,k=1\\k\neq j}}^N W\bl x_j-x_k\br
\end{align}
where $W(d)=w(|d|)$ denotes the radially symmetric interaction potential and $x_j=x_j(t)\in \R^n$ for $j=1,\ldots,N$ denote the positions of the particles at time $t\geq 0$  \cite{Bertozzi2015, stabilityringpatterns}. The associated gradient flow is:
\begin{align}\label{eq:standardmodel}
\frac{\di x_j}{\di t}
=\frac{1}{N}\sum_{\substack{k=1\\k\neq j}}^N F(x_j-x_k)
\end{align}
where $F(d)=-\nabla W(d)$. Here, $F(x_j-x_k)$ is a conservative force, aligned along the distance vector $x_j-x_k$. Denoting the density of particles at location $x\in \R^n$ and at time $t>0$ by  $\rho=\rho(t,x)$ the interaction energy is given by
$$\mathcal{W}[\rho]=\frac{1}{2}\int_{\R^2} \bl W\ast \rho\br(x) \rho(\di x)$$ and the continuum equation corresponding to \eqref{eq:standardmodel}, also referred to as the aggregation equation  \cite{bertozzi2009, stabilityringpatterns, Bertozzi2015, aggregationeq}, reads
\begin{align}\label{eq:standardmodelmacroscopic}
\rho_t+\nabla\cdot\bl \rho u\br=0,\qquad u=-\nabla W\ast \rho
\end{align} 
where  $u=u(t,x)$ is the macroscopic velocity field. The aggregation equation \eqref{eq:standardmodelmacroscopic} whose well-posedness has been proved in \cite{Bertozzi2011} has extensively been studied recently, mainly in terms of its gradient flow structure \cite{Li2004, Carrillo:2003, opac-b1122739, gradientflows, Carrillo2006}, the blow-up dynamics for fully
attractive potentials \cite{bertozzi2009, carrillo2011, bertozzi2012, Carrillo2016304}, and the rich variety of steady states \cite{Fellner2010, Fellner20111436, swarmequilibria, nonlocalInteractionEquation, Carrillo2012306, Carrillo2012550, bertozzi2012, nonlocalinteraction, balague_preprint, vonBrecht2012, PredictingPatternFormation, Confinement, Canizo2015, Carrillo2016, Carrillo2012306}.

If the radially symmetric potential $W$ is purely attractive, e.g. $w$ is an increasing function with $w(0)=0$, the density of the particles converges to a Dirac Delta function located at the center of mass of the density \cite{Burger2008}. In this case, the Dirac Delta function is the unique stable steady state and a global attractor  \cite{carrillo2011}. Under certain conditions the collapse towards the Dirac Delta function can take place in finite time  \cite{carrillo2011, bertozzi2009, finiteTimeBlowup, Bertozzi2007}.

In biological applications, however, it is not sufficient to consider purely attractive potentials since the inherently nonlocal interactions between the individual entities  occur on different scales  \cite{swarmequilibria, migratorylocusts, nonlocal_swarm}.  These interactions are usually described by  short-range repulsion to prevent collisions between the individuals as well as  long-range attraction that keeps the swarm cohesive \cite{Mogilner2003, Okubo}. The associated radially symmetric potentials $w$, also referred to as repulsive-attractive potentials, first decrease  and then increase as a function of the radius. These potentials lead to possibly more complex steady states than the purely attractive potentials and can be considered as a minimal model for pattern formation in large systems of individuals \cite{nonlocalinteraction}.

The 1D nonlocal interaction equation with a repulsive-attractive potential  has been studied in \cite{Fellner20111436, nonlocalInteractionEquation, Fellner2010}. The authors show that the behavior of the solution strongly depends on the regularity of the interaction potential. More precisely, the solution converges to a sum of Dirac masses for regular interaction, while it remains uniformly bounded for singular repulsive potentials. 

Pattern formation for repulsive-attractive potentials in multiple dimensions is studied in \cite{Bertozzi2015, PredictingPatternFormation, stabilityringpatterns, vonBrecht2012}.  The authors perform a linear stability analysis of  ring equilibria and derive conditions on the potential to classify the different instabilities. This analysis can also be used to study the  stability of flock solutions and mill rings in the associated second-order model, see  \cite{secondordermodel} and \cite{Carrillo2014NonlinFlock} for the linear and  nonlinear stability of flocks, respectively. A numerical study of the $N$ particle interaction model for specific repulsion-attraction potentials \cite{stabilityringpatterns, Bertozzi2015} outlines a wide range of radially symmetric patterns such as rings, annuli and uniform circular patches, while exceedingly complex patterns  are also possible. In particular, minimizers of the interaction energy  \eqref{eq:discreteinteractionenergy}, i.e. stable stationary states of the microscopic model \eqref{eq:standardmodel}, can be radially symmetric even for radially symmetric potentials. This has been studied and discussed by instabilities of the sphere and ring solution in \cite{Bertozzi2015, PredictingPatternFormation, vonBrecht2012}.  The convergence of radially symmetric solutions towards spherical shell stationary states in multiple dimensions is discussed in \cite{nonlocalinteraction}. Another possibility to produce concentration in lower dimensional sets is to use potentials which are not radially symmetric. This has been explored recently in the area of dislocations in  \cite{Mora}. Moreover, the nonlocal interaction equation in heterogeneous environments (where domain boundaries are also allowed) is investigated in \cite{nonlocalinteraction_heterogeneities}. Besides, interaction energies with boundaries have been studied in \cite{Carrillo:2016}. 

Nonlocal interaction models have been studied for specific types of repulsive-attractive potentials \cite{balague_preprint,geometryminimizer,Filippov,Carrillo20130399,explicitsol,swarmdynamics}. In \cite{balague_preprint} the dimensionality of the patterns is analyzed for repulsive-attractive potentials that are strongly  or mildly repulsive  at the origin, i.e. potentials with a singular Laplacian at the origin satisfying $\Delta W(d)\sim -|d|^{-\beta}$ as $d\to 0$ for some $0<\beta<n$ in  $n$ dimensions and potentials whose Laplacian does not blow up at the origin  satisfying $W(d)\sim -|d|^{\alpha}$ as $d\to 0$ for some $\alpha>2$, respectively. In \cite{swarmdynamics} a specific example of a repulsive-attractive potential is studied, given by a Newtonian potential for the repulsive and a polynomial for the attractive part, respectively. 

In this work, we consider an evolutionary  particle model with an anisotropic interaction force in two dimensions. More precisely, we generalize the extensively studied model \eqref{eq:standardmodel} by considering an $N$ particle model of the form 
\begin{align}\label{eq:particlemodel}
\frac{\di x_j}{\di t}=\frac{1}{N}\sum_{\substack{k=1\\k\neq j}}^N F(x_j-x_k,T(x_j))
\end{align}
where   $F(x_j-x_k,T(x_j))\in\R^2$  describes the force exerted from $x_k$ on $x_j$. Here,  $T(x_j)$ denotes a tensor field at location $x_j$ which is given by $T(x):=\chi s(x)\otimes s(x) +l(x)\otimes l(x)$ for orthonormal  vector fields $s=s(x), l=l(x)\in\R^2$ and $\chi\in[0,1]$.

As in the standard particle model \eqref{eq:standardmodel} we assume that the  force $F(x_j-x_k,T(x_j))$  is the sum of  repulsion and  attraction forces. In \eqref{eq:standardmodel}, attraction and repulsion forces are  aligned along the distance vector $x_j-x_k$ so that the total force  $F(x_j-x_k)$ is also aligned along $x_j-x_k$. In the extended model \eqref{eq:particlemodel}, however, the   orientation of $F(x_j-x_k,T(x_j))$ depends not only on the distance vector $x_j-x_k$ but additionally on the tensor field $T(x_j)$ at  location $x_j$. More precisely, the attraction force will be assumed to be aligned along the vector $T(x_j)(x_j-x_k)$. Since $T$ depends on a parameter $\chi\in [0,1]$ the resulting force direction is regulated by  $\chi$. In particular,  alignment along the distance vector $x_j-x_k$  is included in \eqref{eq:particlemodel} for $\chi=1$.
 The additional dependence of \eqref{eq:particlemodel} on the parameter $\chi$ in the definition of the tensor field $T$ introduces an anisotropy to the equation. This anisotropy 
leads to more complex, anisotropic patterns that do not occur in the simplified model \eqref{eq:standardmodel}. Due to the dependence on parameter $\chi$  the force $F$ is  non-conservative  in general so that it cannot be derived from a potential. However, most of the analysis of the interaction models in the literature relies on the existence of an interaction potential as outlined above. A particle interaction model of the form \eqref{eq:particlemodel} with a non-conservative force term that depends on an underlying tensor field $T$ appears not to have been investigated mathematically in the literature yet. 

Due to the generality of the formulation of the anisotropic interaction model \eqref{eq:particlemodel} a better understanding of the pattern formation in \eqref{eq:particlemodel} can be regarded as a first step towards understanding anisotropic pattern formation in nature. An example of an $N$ particle model of the form \eqref{eq:particlemodel} is the model introduced by K\"{u}cken and Champod in \cite{Merkel}, describing the formation of fingerprint patterns based on the interaction of Merkel cells and mechanical stress in the epidermis \cite{fingerprintbiology}. Even though the K\"{u}cken-Champod model \cite{Merkel} seems to be capable to produce transient patterns that resemble fingerprint patterns, the pattern formation of the  K\"{u}cken-Champod model and its dependence on the model parameters have not been studied analytically or numerically before. In particular, the long-time behavior of solutions to the  K\"{u}cken-Champod model and its stationary solutions have not been understood yet. However, stationary solutions to the K\"{u}cken-Champod model  are of great interest for simulating fingerprints since fingerprint patterns only change in size and not in shape after we are born so that every person has the same fingerprints from infancy to adulthood. Clearly, fingerprint patterns are of great importance in forensic science. Besides, they are increasingly used in biometric applications. 
Hence, understanding the model, proposed in \cite{Merkel}, and in particular its pattern formation  result in a better understanding of the fingerprint pattern formation process. 

The goal of this work is to study the equilibria of the microscopic model \eqref{eq:particlemodel} and the associated mean-field PDE analytically and numerically. We investigate the existence of equilibria analytically. Since numerical simulations are crucial for getting a better understanding of the patterns which can be generated with the  K\"{u}cken-Champod model we investigate the impact of  the model parameters on the  resulting transient and steady patterns numerically. In particular, we study the transition of steady states with respect to the parameter $\chi$. Based on the results in this paper we study the solution to the K\"{u}cken-Champod model for non-homogeneous tensor fields, simulate the fingerprint pattern formation process and model fingerprint patterns with certain features in \cite{fingerprint_preparation}. 


Note that the modeling involves multiple scales which can be seen in several different ways. Given the particle model in \eqref{eq:particlemodel} we consider the associated particle density  to derive the mean-field limit. Here, the interaction force exhibits short-range repulsion and long-range attraction. The direction of the attraction force depends on the parameter $\chi$ which is responsible for different transient and steady state patterns. More precisely, ring equilibria obtained for $\chi=1$ evolve into ellipse patterns and stripe patterns as $\chi$ decreases.
Besides, large-time asymptotics are considered for determining the equilibria.

This work is organized as follows. In Section \ref{sec:descriptionmodel}, a general formulation of an anisotropic $N$ particle model \eqref{eq:particlemodel} and the associated mean-field PDE \eqref{eq:macroscopiceq} are introduced and a connection to the  K\"{u}cken-Champod model \cite{Merkel} is established via a specific class of interaction forces. The  solution to the mean-field PDE \eqref{eq:macroscopiceq} is analyzed in Section \ref{sec:analysis}. More precisely, we discuss the impact of the parameter $\chi$ on the force alignment and on the solution to the model. Besides, we study the impact of spatially homogeneous tensor fields and we show that the equilibria to the mean-field PDE \eqref{eq:macroscopiceq} for any spatially homogeneous tensor field can be regarded as a coordinate transform of the tensor field $T=\chi s\otimes s +l\otimes l$ where $s=(0,1)$ and $l=(1,0)$ for any parameter $\chi\in[0,1]$. Hence, we can restrict ourselves to this specific tensor field $T$ for the analysis.  We investigate the existence of equilibria to the mean-field PDE \eqref{eq:macroscopiceq} whose form depend on the choice of $\chi$. Under certain assumptions we  show that
for $\chi=1$ there exists at most one radius $R>0$ such that the ring state of radius $R$ is a nontrivial equilibrium mean-field PDE \eqref{eq:macroscopiceq} for spatially homogeneous tensor fields and uniqueness can be guaranteed under an additional assumption, while for $\chi\in[0,1)$ the ring state  is no equilibrium.  For $\chi\in[0,1]$ and $R>0$ sufficiently small there exists at most one $r>0$  such that an ellipse with major axis $R+r$ and minor axis $R$ whose major axis is aligned along $s$ is an equilibrium. Besides, the shorter the minor axis of the ellipse, the longer the major axis of possible ellipse steady states and the smaller the value of $\chi$ the longer the major and the shorter the minor axis of the possible ellipse equilibrium.  Section \ref{sec:numerics} contains a description of the numerical method and we discuss the simulation results for the K\"{u}cken-Champod model \eqref{eq:particlemodel}. The numerical results include an investigation of the stationary solutions and their dependence on different parameters in the model, including the impact of the parameter $\chi$ and the associated transition between the isotropic and anisotropic model. Besides, we compare the numerical with the analytical results.

\section{Description of the model}\label{sec:descriptionmodel}
K\"{u}cken and Champod introduced a particle model in \cite{Merkel} modeling the formation of fingerprint patterns by describing the interaction between so-called Merkel cells  on a  domain $\Omega\subseteq \R^2$. Merkel cells are epidermal cells that appear in the volar skin at about the 7th week of pregnancy. From that time onward they start to multiply and organize themselves in lines exactly where the primary ridges arise. The model introduced in \cite{Merkel} models this pattern formation process as the rearrangement of Merkel cells from a random initial configuration into roughly parallel ridges along the lines of smallest compressive stress.

In this section, we describe a general formulation of the anisotropic microscopic model, relate it to the K\"{u}cken-Champod particle model and formulate the corresponding mean-field PDE.

\subsection{General formulation of the anisotropic interaction model}
In the sequel we consider $N$ particles
 at positions $x_j=x_j(t)\in\R^2,~j=1,\ldots,N,$ at time $t$. The evolution of the particles can be described by \eqref{eq:particlemodel} with initial data $x_j(0)=x_j^{in},~j=1,\ldots,N$.
Here, $F(x_j-x_k,T(x_j))$ denotes the total force that particle $k$ exerts on particle $j$ subject to an underlying stress tensor field $T(x_j)$ at $x_j$, describing the local stress field. The dependence on $T(x_j)$ is based on the experimental results in \cite{KIM1995411} where an alignment of the particles along the local stress lines is observed, i.e. the evolution of particle $j$ at location $x_j$ depends on the the local stress tensor field $T(x_j)$. Note that \eqref{eq:particlemodel} results from Newton's second law by neglecting the inertia term. 

The total force $F$ in \eqref{eq:particlemodel} is given by
\begin{align}\label{eq:totalforce}
F(d(x_j,x_k),T(x_j))=F_A(d(x_j,x_k),T(x_j))+F_R(d(x_j,x_k))
\end{align} 
for the distance vector $d(x_j,x_k)=x_j-x_k\in\R^2$. Here, $F_R$ denotes the repulsion force that particle $k$ exerts on particle $j$ and $F_A$ is the attraction force exerted on particle $j$ by particle $k$. The tensor field $T(x_j)$ at $x_j$ encodes the direction of the fingerprint lines at $x_j$  
and is given by $T(x_j)=\chi s(x_j)\otimes s(x_j)+l(x_j)\otimes l(x_j)$ with $\chi\in[0,1]$. Here, $s=s(x_j)\in \R^2$ and $l=l(x_j)\in\R^2$ are orthonormal vectors, describing the directions of smallest and largest stress, respectively. The repulsion and attraction forces are of the form
\begin{align}\label{eq:repulsionforce}
F_R(d)=f_R(|d|)d
\end{align}
and
\begin{align}\label{eq:attractionforce}
F_A(d=d(x_j,x_k),T(x_j))=f_A(|d|)T(x_j)d,
\end{align}
respectively, where, again, $d=d(x_j,x_k)=x_j-x_k\in\R^2$. The direction of the interaction forces is determined by the parameter $\chi\in[0,1]$ in the definition of  $T$. For $\chi=1$ the tensor field $T$ is the two-dimensional unit matrix and the attraction force between two particles is aligned along their distance vector, while for $\chi=0$ the attraction between two particles is oriented along $l$. Depending on the choice of the coefficient functions $f_R$ and $f_A$ in \eqref{eq:repulsionforce} and \eqref{eq:attractionforce}, respectively, the forces are repulsive or attractive according to the following local definition:
\begin{definition}[Strictly repulsive (attractive) forces]\label{def:attractionrepulsion}
Let the vector field $G=G(x,y)$  be a continuous interaction force, i.e. the vector $G(x,y)$ is the force which is exerted on $x$ by $y$. Then $G$ at $x$ in direction $x-y$ is strictly repulsive (attractive) if
\begin{align*}
G(x,y)\cdot (x-y)>0\qquad (<0).
\end{align*}
\end{definition}

The meaning of this definition is the following. Let $y$ be fixed and let $X=X(t)$ be the trajectory given by
\begin{align*}
\frac{\di X}{\di t}=G(X,y),\quad X(0)=x,
\end{align*}
then $|X(t)-y|$ is locally at $t=0$ strictly monotonically increasing (decreasing).

To guarantee that $F_R$ and $F_A$ are repulsion and attractive forces, we make assumptions on the coefficient functions $f_R$ and $f_A$  in \eqref{eq:repulsionforce} and \eqref{eq:attractionforce}, respectively. Besides, assumptions on the total force $F$ in \eqref{eq:totalforce} are necessary to derive the mean-field PDE. In the sequel, we assume that the following conditions are satisfied:
\begin{ass}\label{ass:propertyforce}
We assume that $f_R\colon \R^2\to\R$ and $f_A\colon \R^2\to \R$ denote smooth, integrable coefficient functions satisfying
\begin{align}\label{eq:forcecoefficients}
f_R(|d|)\geq 0 \quad \text{and}\quad f_A(|d|)\leq 0\quad \text{for all}\quad d\in\R^2,
\end{align}
such that the total interaction force $F$ in \eqref{eq:totalforce} exhibits short-range repulsion and long-range attraction forces along $l$, i.e. there exists a $d_a>0$ such that
\begin{align*}
(f_A+f_R)(|d|)\leq 0 \enspace \text{for}\enspace |d|>d_a \qquad \text{and}\qquad  (f_A+f_R)(|d|)> 0 \enspace \text{for}\enspace 0\leq |d|<d_a.
\end{align*}
Also, $F\in C^1$ has bounded total derivatives, i.e. there exists some $L\geq 0$ such that
\begin{align*}
\begin{split}
\sup_{x,x'\in\R^2}|D_x F(d(x,x'),T(x))|\leq L \quad\text{and}\quad \sup_{x,x'\in\R^2}|D_{x'} F(d(x,x'),T(x))|\leq L,
\end{split}
\end{align*}
where $D_x$ denotes the total derivative with respect to $x$. This implies that $F$ is Lipschitz continuous in both arguments. In particular, $F$ grows at most linearly at infinitely.
\end{ass}

\begin{rem}[Existence of an interaction potential]\label{rem:existencepotential}
The repulsion force $F_R$ of the form \eqref{eq:repulsionforce} can  be derived from a radially symmetric potential. For the existence of interaction potentials for the attractive force $F_A$ we restrict ourselves to spatially homogeneous tensor fields first. Let $\chi\in[0,1]$, set $l=(1,0)$ and $s=(0,1)$, and let $\tilde{T}=\chi \tilde{s} \otimes \tilde{s}+\tilde{l}\otimes \tilde{l}$ denote a spatially homogeneous tensor field for orthonormal vectors $\tilde{l},\tilde{s}\in\R^2$. Then, 
\begin{align}\label{eq:proofvectortransform}
\tilde{s}=R_{\theta}s\quad\text{and}\quad \tilde{l}=R_{\theta}l,
\end{align}
  where the angle of rotation $\theta$ and the corresponding rotation matrix $R_{\theta}$ are given by
\begin{align}\label{eq:anglerotationproof}
\theta=\begin{cases}
\arccos(\tilde{s}_2) & \tilde{s}_1<0\\
2\pi-\arccos(\tilde{s}_2) & \tilde{s}_1>0
\end{cases},\qquad\text{and}\qquad
R_\theta=\begin{pmatrix}
\cos(\theta) & -\sin(\theta)\\
\sin(\theta) & \cos(\theta)
\end{pmatrix},
\end{align}
respectively, and we have $\tilde{T}=R_{\theta}T R_{\theta}^T$ with  $T=\chi s \otimes s+l\otimes l$. Hence, 
\begin{align*}
F_A\bl d,\tilde{T}\br=f_A\bl|d|\br\begin{pmatrix}
\cos^2\bl \theta\br+\chi\sin^2\bl\theta\br & \bl 1-\chi\br \sin\bl\theta\br\cos\bl\theta\br\\
\bl 1-\chi\br\sin\bl\theta\br \cos\bl\theta\br & \chi \cos^2\bl\theta\br+\sin^2\bl\theta\br
\end{pmatrix}d
\end{align*}
by \eqref{eq:attractionforce}, where $d=(d_1,d_2)\in \R^2$. The condition $$\frac{\partial (F_A)_1}{\partial d_2}=\frac{\partial (F_A)_2}{\partial d_1}$$
for $F_A$ being a conservative force implies $$\cos^2\bl\theta\br+\chi\sin^2\bl\theta\br=\chi\cos^2\bl\theta\br+\sin^2\bl\theta\br \quad\text{and}\quad \bl 1-\chi\br \sin\bl\theta\br\cos\bl\theta\br=0,$$ which can only be satisfied simultaneously for $\chi=1$ and $\theta\in[0,2\pi)$ arbitrary. Thus, the attraction force for spatially homogeneous tensor fields is conservative for $\chi=1$ only and the associated potential is radially symmetric. This also implies that there exists a potential for $\chi=1$ for any tensor field, while for $\chi\in[0,1)$ there exists no potential. In particular, a potential that is not radially symmetric cannot be constructed for the attraction force $F_A$ for $\chi\in[0,1)$.
\end{rem}

The associated mean-field model  for the distribution function $\rho=\rho(t,x)$ at position $x\in \R^2$ and time $t\geq 0$ can be  derived rigorously in the 1-Wasserstein metric from the microscopic model \eqref{eq:particlemodel}
following the procedure described in \cite{Golse, meanfieldlimit}. The Cauchy problem for the mean-field PDE reads
\begin{align}\label{eq:macroscopiceq}
\begin{split}
\partial_t \rho(t,x)+\nabla_x\cdot \left[ \rho(t,x)\bl F\bl\cdot,T(x)\br \ast \rho(t,\cdot)\br\bl x\br\right]=0\qquad \text{in }\R_+\times \R^2
\end{split}
\end{align}
with initial condition $\rho|_{t=0}=\rho^{in}$  in $\R^2$.

\subsection{K\"{u}cken-Champod particle model}
Consider as an example the K\"{u}cken-Champod particle model \cite{Merkel} with a spatially homogeneous tensor field $T$ producing straight parallel ridges, e.g. $$T=\begin{pmatrix}
1& 0\\ 0 &\chi
\end{pmatrix}, $$   is considered for studying the pattern formation. For more realistic patterns  the tensor field is generated from 3D finite element simulations  \cite{fingerprintformation1,fingerprintformation2} or from images of real fingerprints. 
The coefficient function $f_R$ of the repulsion force $F_R$ \eqref{eq:repulsionforce} in  the K\"{u}cken-Champod model \eqref{eq:particlemodel} is given by
\begin{align}\label{eq:repulsionforcemodel}
f_R(d)=(\alpha |d|^2+\beta)\exp(-e_R |d|)
\end{align}
for  $d\in\R^2$ and nonnegative parameters $\alpha$, $\beta$ and $e_R$. The coefficient function $f_A$  of the attraction force \eqref{eq:attractionforce} is given by
\begin{align}\label{eq:attractionforcemodel}
f_A(d)=-\gamma|d|\exp(-e_A|d|)
\end{align}
for $d\in\R^2$ and nonnegative constants $\gamma$ and $e_A$. For the case that the total force \eqref{eq:totalforce} exhibits short-range repulsion and long-range attraction along $l$, we choose the parameters as follows:
\begin{align}\label{eq:parametervaluesRepulsionAttraction}
\begin{split}
\alpha&=270, \quad \beta=0.1, \quad \gamma=35, \quad
e_A=95, \quad e_R=100, \quad \chi\in[0,1].
\end{split}
\end{align} 
The coefficient functions  \eqref{eq:repulsionforcemodel} and \eqref{eq:attractionforcemodel} for the repulsion and attraction forces \eqref{eq:repulsionforce} and \eqref{eq:attractionforce} in the K\"{u}cken-Champod model \eqref{eq:particlemodel}  are plotted in Figure \ref{fig:forces} for the parameters in \eqref{eq:parametervaluesRepulsionAttraction} and one can easily check that they satisfy Assumption \ref{ass:propertyforce}. 
If not stated otherwise, we consider the parameter values in \eqref{eq:parametervaluesRepulsionAttraction} for the force coefficient functions \eqref{eq:repulsionforcemodel} and \eqref{eq:attractionforcemodel} in the sequel.  The interaction forces between two particles with distance vectors $d=|d|l$ and $d=|d|s$ are given by $((f_R+f_A)\cdot\text{id})l$ and $((f_R+\chi f_A)\cdot\text{id})s$, respectively, and are shown in Figure \ref{fig:sumofforces} for $\chi=0.2$, while the corresponding coefficient functions  are illustrated in Figure \ref{fig:forces}. 
For the choice of parameters in \eqref{eq:parametervaluesRepulsionAttraction}   repulsion dominates for short distances along $l$ to prevent the collision of particles. Besides, the total force exhibits   long-range attraction along $l$ whose absolute value decreases with the distance between particles. Along $s$ the particles are always repulsive for $\chi=0.2$, independent of the distance, though the  repulsion force gets weaker for longer distances. 
\begin{figure}[htbp]
    \centering
    \subfloat[Force coefficients  $f_R$ and  $f_A$]{\includegraphics[width=0.45\textwidth]{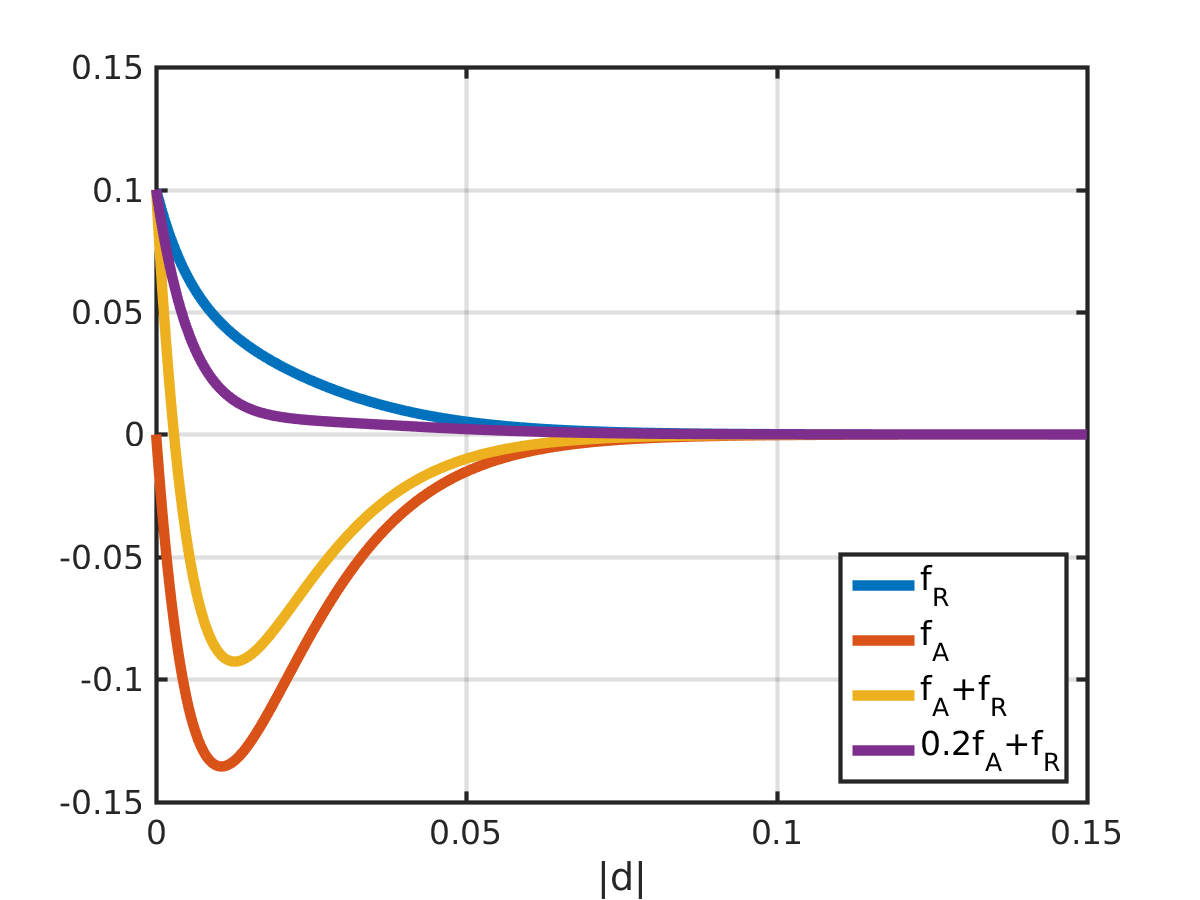}\label{fig:forces}}
\subfloat[Total force coefficients along $l$ and $s$]{\includegraphics[width=0.45\textwidth]{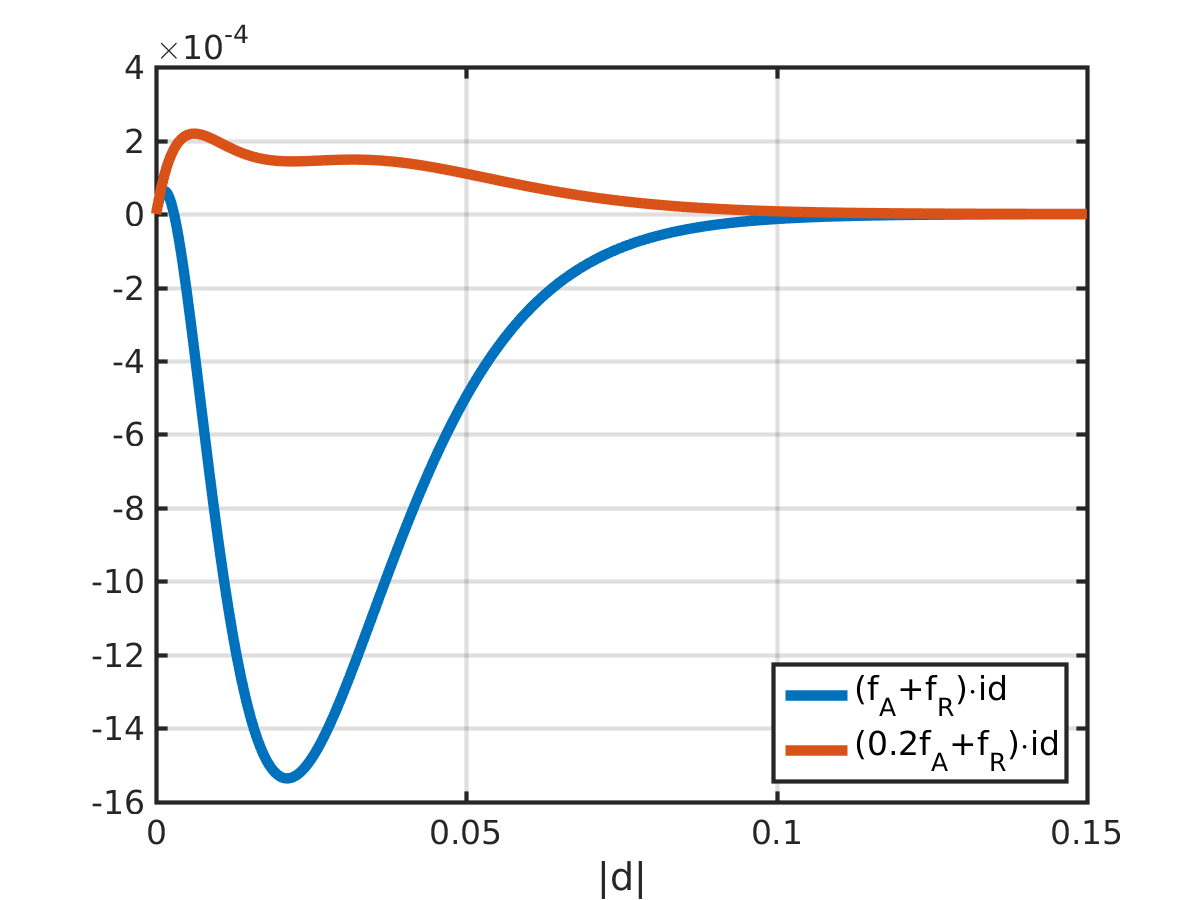}\label{fig:sumofforces}}
    \caption{Coefficients $f_R$ in \eqref{eq:repulsionforcemodel} and $f_A$ in \eqref{eq:attractionforcemodel} of repulsion force \eqref{eq:repulsionforce} and attraction force \eqref{eq:attractionforce}, respectively, as well as the total interaction force   along $l$ and $s$ for $\chi=0.2$ (i.e. $(f_A+f_R)\cdot\text{id}$ and $(0.2f_A+f_R)\cdot\text{id}$, respectively) and its coefficients (i.e. $f_A+f_R$ and $0.2f_A+f_R$) for parameter values in  \eqref{eq:parametervaluesRepulsionAttraction}}
\end{figure}

\section{Analysis of the model}\label{sec:analysis}
We analyze the equilibria of the mean-field PDE \eqref{eq:macroscopiceq} in terms of the parameter $\chi\in[0,1]$  for the general formulation of the model, i.e. the total force is given by \eqref{eq:totalforce} where the repulsion and the attraction forces are of the form \eqref{eq:repulsionforce} and \eqref{eq:attractionforce}, respectively.

\subsection{Interpretation of the total force}\label{sec:interpretationforce}
The alignment of the attraction force \eqref{eq:attractionforce} and thus the pattern formation strongly depend on the choice of the parameter $\chi\in[0,1]$. For $\chi=1$  the total force $F$ in \eqref{eq:totalforce} can be derived from a radially symmetric potential and the mean-field PDE \eqref{eq:macroscopiceq} reduces to the isotropic interaction equations \eqref{eq:standardmodelmacroscopic}. In particular, the solution to \eqref{eq:macroscopiceq} is radially symmetric for $\chi=1$ for radially symmetric initial data \cite{finiteTimeBlowup}.
 
For $\chi\in[0,1)$ the attraction force $F_A$ of the form \eqref{eq:attractionforce} is not conservative by Remark \ref{rem:existencepotential} and can be written as the sum of a conservative and a non-conservative force, given by
$F_A=F_{A,1}+F_{A,2}$
with
\begin{align*}
F_{A,1}(d)=f_A(|d|)d
\end{align*}
and
\begin{align*}
F_{A,2}(d(x_j,x_k),T(x_j))=f_A(|d|)(T(x_j)-I)d=f_A(|d|)\left(\chi-1\right) \bl s(x_j)\cdot d\br s(x_j),
\end{align*}
where  $d=d(x_j,x_k)=x_j-x_k$ and $I$ denotes the two-dimensional identity matrix. In particular, $F_{A,1}$ does not depend on $\chi$ and is equal to the attraction force in \eqref{eq:attractionforce} with $\chi=1$. Since the coefficient function $f_A(\chi-1)$ of $F_{A,2}$ is nonnegative, $F_{A,2}$ is a repulsion force aligned along $s(x_j)$  and leads to an additional advection along $s(x_j)$ compared to the case $\chi=1$. This repulsion force along $s(x_j)$ is the larger, the smaller $\chi$. In particular, for the force coefficients $f_A$ and $f_R$ in the K\"ucken-Champod model \eqref{eq:particlemodel}, given by \eqref{eq:attractionforcemodel} and \eqref{eq:repulsionforcemodel} with parameters in \eqref{eq:parametervaluesRepulsionAttraction}, the total force along $s$ is  purely repulsive  for $\chi$ sufficiently small as illustrated in Figure \ref{fig:sumofforcesVarychi}.

For the spatially homogeneous tensor field $T=\chi s \otimes s+l\otimes l$ with   $l=(1,0)$ and $s=(0,1)$ the solution is stretched along the vertical axis for $\chi<1$. The smaller the value of $\chi$, the larger the repulsion force and the more the solution is stretched along the vertical axis. For $\chi$ sufficiently small  stretching along the entire vertical axis is possible for solutions to the K\"ucken-Champod model \eqref{eq:particlemodel} because of purely repulsive forces along $s$.

\begin{figure}[htbp]
    \centering
    \subfloat[Along $l$]{\includegraphics[width=0.45\textwidth]{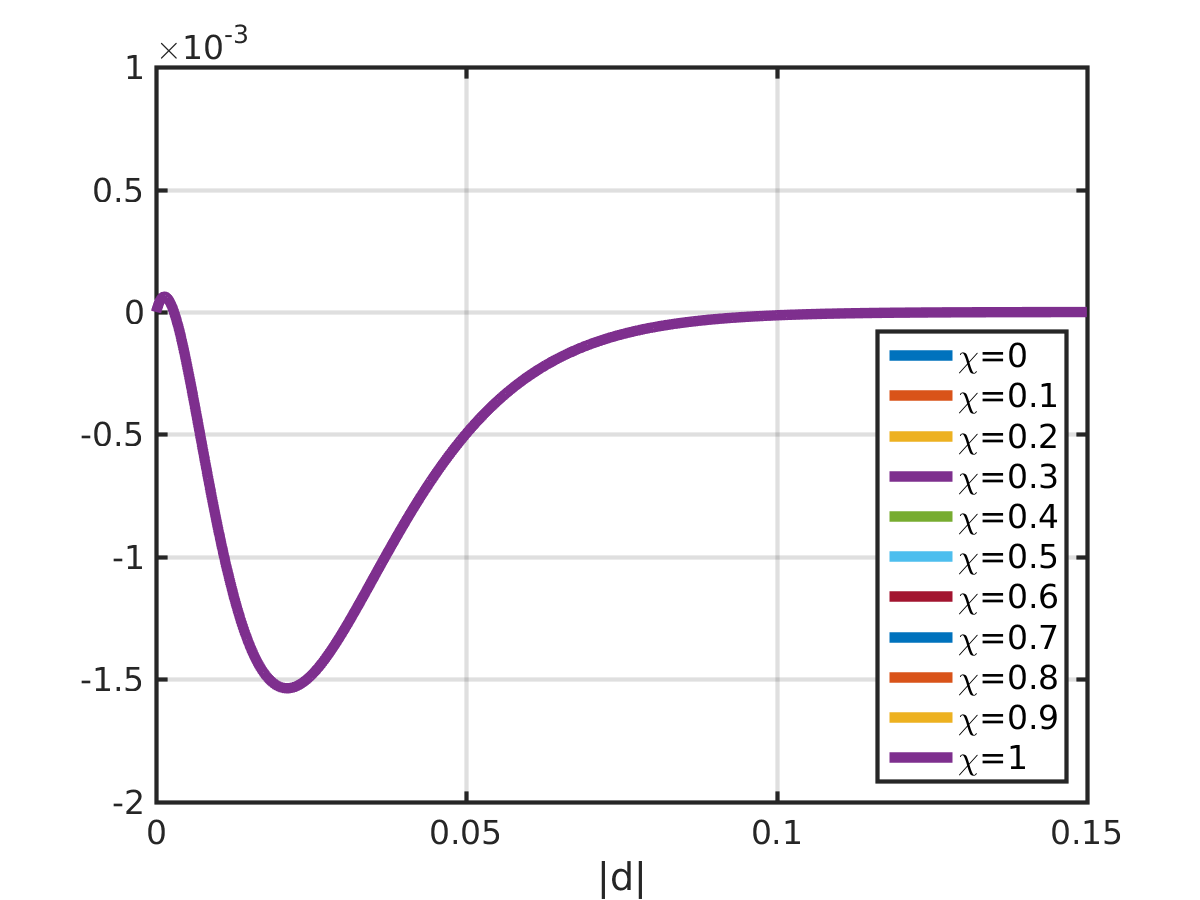}}
    \subfloat[Along $s$]{
        \includegraphics[width=0.45\textwidth]{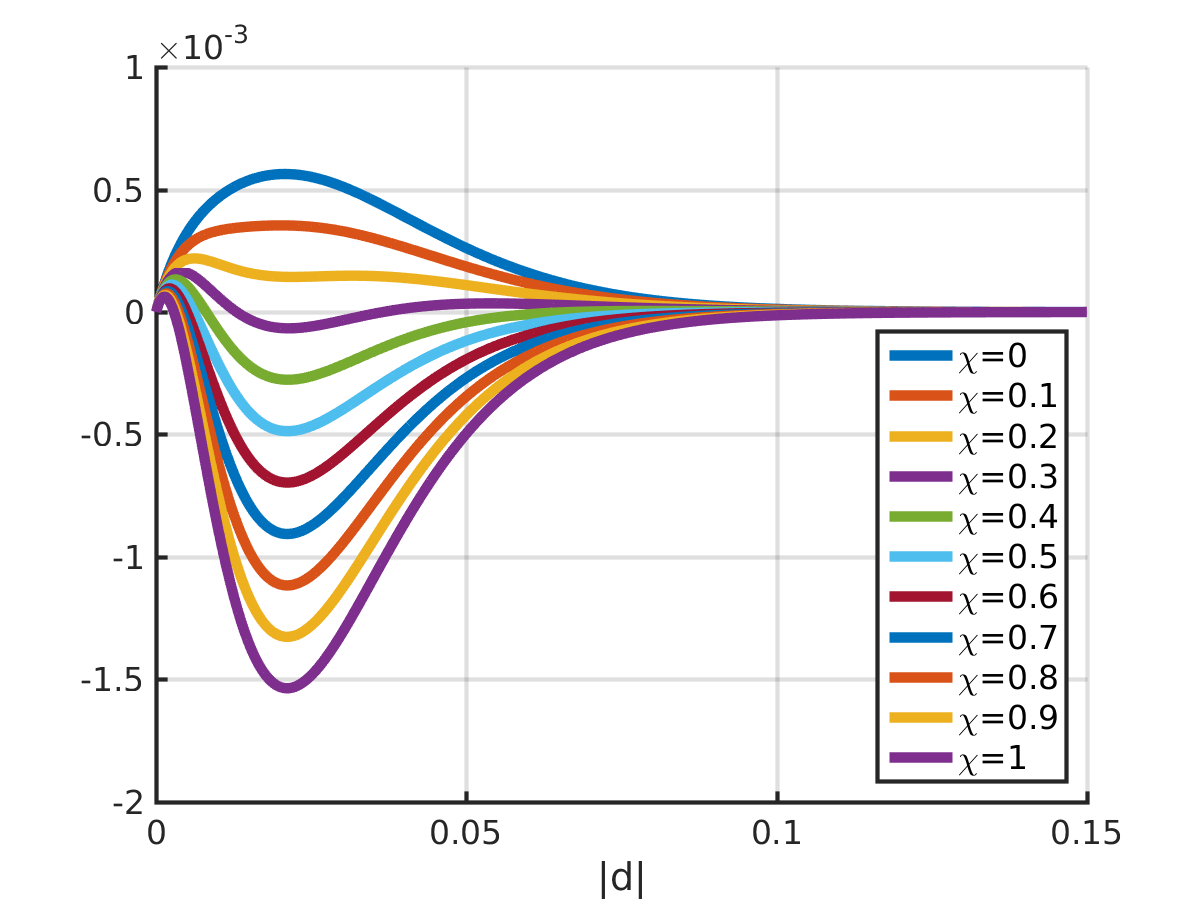}}
    \caption{Total force  along $l$ and $s$ (i.e. $(f_A+f_R)\cdot\text{id}$ and $(\chi f_A+f_R)\cdot\text{id}$ for coefficients $f_A$ in \eqref{eq:attractionforcemodel} and $f_R$ in \eqref{eq:repulsionforcemodel} of the attraction  and the repulsion force  for parameter values in \eqref{eq:parametervaluesRepulsionAttraction}, respectively) for different values of $\chi$}\label{fig:sumofforcesVarychi}
\end{figure}

\subsection{Impact of  spatially homogeneous tensor fields}\label{sec:interpretationvectorfield}
Let $\chi\in[0,1]$ and  consider the spatially homogeneous tensor field $T=\chi s \otimes s+l\otimes l$ with $l=(1,0)$ and $s=(0,1)$. The solution of the particle model \eqref{eq:particlemodel}  for any spatially homogeneous tensor field $\tilde{T}$ is a coordinate transform of the solution of the particle model \eqref{eq:particlemodel} for the  tensor field $T$. Similarly, for the analysis of equilibria of the microscopic model \eqref{eq:particlemodel} for $\tilde{T}$ it is sufficient to study the equilibria of \eqref{eq:particlemodel} for $T$. 
For a similar statement for the mean-field PDE \eqref{eq:macroscopiceq} we define the concept of an equilibrium  state. 
\begin{definition}[Equilibrium state of \eqref{eq:macroscopiceq}]\label{def:steadystate}
A Borel probability measure $\mu\in \mathcal{P}(\R^2)$ is said to be an equilibrium state of the mean-field PDE \eqref{eq:macroscopiceq} if 
\begin{align}\label{eq:steadystate}
K\in L^1_{\text{loc}}\bl \di \mu\br \quad\text{and}\quad K=0~\text{on}~\text{supp}(\mu)~ \mu\text{-a.e.}
\end{align}
where $K= F(\cdot,T)\ast \mu ~\mu\text{-a.e.}$
\end{definition} 
An equilibrium state of the mean-field equation \eqref{eq:macroscopiceq} for any spatially homogeneous tensor field $\tilde{T}$ is the coordinate transform of an equilibrium state to the mean-field equation \eqref{eq:macroscopiceq} for the tensor field $T$. For detailed computations see Appendix \ref{sec:appendix_computation}.

\subsection{Existence of equilibria}\label{sec:interpretationsteadystates}
Based on the discussion on the action of the total force in Section~\ref{sec:interpretationforce}   possible shapes of equilibria of the mean-field PDE \eqref{eq:macroscopiceq} depend on the choice of the parameter $\chi\in[0,1]$. To analyze the equilibria of the mean-field PDE \eqref{eq:macroscopiceq} in two  dimensions for any spatially homogeneous tensor field,   it is sufficient to consider the tensor field $T=\chi s \otimes s+l\otimes l$ with $l=(1,0)$ and $s=(0,1)$  in the sequel as outlined in Section \ref{sec:interpretationvectorfield}. Note that the forces along $l$ are assumed to be repulsive-attractive, while the forces along $s$ depend significantly on the choice of $\chi$ and may be repulsive,  repulsive-attractive-repulsive or repulsive-attractive. Further note that the forces only depend on the distance vector for spatially homogeneous tensor fields. To simplify the analysis, we make the following assumption on $F$ in addition to Assumption \ref{ass:propertyforce}  in this section: 
\begin{ass}\label{ass:propertyforceadditional}
We assume that $F$ is strictly decreasing along $l$ and $s$ on the interval $[0,d_e]$ for some $d_e>d_a$ where $d_a$ is defined in Assumption \ref{ass:propertyforce}. In particular, there exits $d_e>d_a$ such that $\chi f_A+f_R$ is strictly decreasing on $[0,d_e]$ for all $\chi\in[0,1]$.
\end{ass}


%

\subsubsection{Ellipse pattern}
Solutions to the the mean-field PDE \eqref{eq:macroscopiceq} for  $T=\chi s \otimes s+l\otimes l$ with $l=(1,0)$ and $s=(0,1)$ are stretched along the vertical axis by the discussion in Section \ref{sec:interpretationforce}. This motivates us to consider an ellipse whose major axis is parallel to the vertical axis. Because of  the spatial homogeneity of the tensor field it is sufficient to restrict ourselves to probability measures with center of mass $(0,0)$. 

\begin{definition}
Let $R>0$ and let $r\geq 0$. The ellipse state whose minor and major axis are of lengths $R$ and $R+r$, respectively,   is the probability measure which is uniformly distributed on $\{ x=(x_1,x_2)\in\R^2 \colon x_1=R\cos\phi,~x_2=(R+r)\sin\phi\}$. We denote this probability measure by $\delta_{(R,r)}$.
\end{definition}

First, we restrict ourselves to nontrivial ring states $\delta_{(R,0)}$ of radius $R>0$, i.e. we consider the special case of ellipse states where $r=0$.
The existence of ring equilibria for repulsive-attractive potentials that do not decay faster than $1/d^2$ as $d\to \infty$ has already been discussed in \cite{nonlocalinteraction}. However, the force coefficients \eqref{eq:repulsionforcemodel} and \eqref{eq:attractionforcemodel} in the K\"{u}cken-Champod model \cite{Merkel} decay exponentially fast as $d\to \infty$. Besides, a repulsive-attractive potential  exists for $\chi=1$ only by Remark \ref{rem:existencepotential}.
To analyze the ring equilibria we distinguish between the two cases $\chi=1$ and $\chi\in[0,1)$, starting with the case $\chi=1$. 
\begin{lemma}\label{lem:condRingsteadystate}
Let  $\chi=1$. The probability measure $\delta_{(R,0)}$ is  a  nontrivial ring equilibrium to \eqref{eq:macroscopiceq} for radius $R>0$ if and only if
\begin{align}\label{eq:condRingsteadystate}
\int_0^{\pi} (f_A+f_R)\bl R\sqrt{(1-\cos\phi)^2+\sin^2\phi}\br(1-\cos\phi)\di \phi=0.
\end{align}
\end{lemma}
\begin{proof}
By Definition \ref{def:steadystate} and by symmetry it suffices to show for $\chi=1$ that there exists  $R>0$ such that 
\begin{align*}
(F(\cdot,T)\ast\delta_{(R,0)})((R,0))
=\int_0^{2\pi} F(R(1-\cos\phi,-\sin \phi),T)R\di \phi=0
\end{align*}
for nontrivial ring equilibria. A change of variables yields
$$\int_{\pi}^{2\pi} F(R(1-\cos\phi,-\sin \phi),T)R\di \phi=\int_{0}^{\pi} F(R(1-\cos\phi,\sin \phi),T)R\di \phi. $$
Hence, by using the simplified form of $F(d,T)=(f_A+f_R)(|d|)d$ for $\chi=1$  this implies that it is sufficient to show the existence of $R>0$ such that 
$$R^2\int_0^{\pi} (f_A+f_R)\bl R\sqrt{(1-\cos\phi)^2+\sin^2\phi}\br(1-\cos\phi)\di \phi=0.$$
Since we are interested in nontrivial ring equilibria with radius $R>0$ the condition finally reduces to \eqref{eq:condRingsteadystate}.
\end{proof}
\begin{proposition}\label{prop:ringsteadystate}
Let $\chi=1$. There exists at most one radius $\bar{R}\in (0,\frac{d_e}{2}]$ such that the ring state $\delta_{(\bar{R},0)}$ of radius $\bar{R}$ is a nontrivial  equilibrium to the mean-field PDE \eqref{eq:macroscopiceq}. If
\begin{align}\label{eq:conduniquesteadystate}
\int_0^{\pi} (f_A+f_R)\bl \frac{d_e}{2}\sqrt{(1-\cos\phi)^2+\sin^2\phi}\br(1-\cos\phi)\di \phi<0
\end{align}
there exists a unique $\bar{R}\in (\frac{d_a}{2},\frac{d_e}{2}]$ such that the ring state $\delta_{(\bar{R},0)}$ of radius $\bar{R}$ is a  nontrivial  equilibrium.
\end{proposition}
\begin{proof}
Consider the left-hand side of \eqref{eq:condRingsteadystate} as a function of $R$ denoted by $G(R)$. By deriving $G(R)$ with respect to $R$ and using Assumption \ref{ass:propertyforceadditional} one can easily see that $G(R)$ is strictly decreasing as a function of $R$ on $[0,\frac{d_e}{2}]$. Note that $G(0)>0$, $G(R)>0$ for $R\leq\frac{d_a}{2}$ and $f_A, f_R$ are continuous by Assumption \ref{ass:propertyforce} on the total force. Since \eqref{eq:conduniquesteadystate} is equivalent to $G(\frac{d_e}{2})<0$ this concludes the proof. 
\end{proof}

One can easily check that \eqref{eq:conduniquesteadystate} is satisfied for the force coefficients \eqref{eq:repulsionforcemodel} and \eqref{eq:attractionforcemodel} in the K\"{u}cken-Champod model \eqref{eq:particlemodel} with parameter values in \eqref{eq:parametervaluesRepulsionAttraction} if $d_e$ is the argument of the minimum of $f_A+f_R$, see Assumption \ref{ass:propertyforceadditional}. In particular, this implies that there exists a unique nontrivial ring equilibrium of radius $R\in (\frac{d_a}{2},\frac{d_e}{2}]$ to the mean-field PDE \eqref{eq:macroscopiceq} for the forces in the K\"{u}cken-Champod model  for $\chi=1$.

The case $\chi\in [0,1)$ can be analyzed similarly as the  one for $\chi=1$ for ring patterns  except that some of the symmetry arguments do not hold.
\begin{proposition}
Let $\chi\in[0,1)$. There exists no $R\in (0,\frac{d_e}{2}]$ such that the ring state $\delta_{(R,0)}$ is an equilibrium to the mean-field PDE \eqref{eq:macroscopiceq}. 
\end{proposition}
\begin{proof}
For $\chi=1$, $(F(\cdot,T)\ast\delta_{(R,0)})((R,0))=0$ is equivalent to \eqref{eq:condRingsteadystate} by Lemma \ref{lem:condRingsteadystate}, based on the property  $(F(\cdot,T)\ast\delta_{(R,0)})((R,0))\cdot s=0$. Since $$F(d,T)=f_A(|d|)(d_1,\chi d_2)+f_R(|d|)d$$ where $d=(d_1,d_2)$, \eqref{eq:condRingsteadystate} also has to be satisfied for $\chi\in [0,1)$. Similarly as in the proof of Lemma~\ref{lem:condRingsteadystate} one can show that 
$$(F(\cdot,T)\ast\delta_{(R,0)})((0,R))=\int_0^{2\pi} F(R(-\cos \phi,1-\sin\phi),T)R\di \phi=0$$
is equivalent to 
\begin{align}\label{eq:condRingsteadystateChigeneral}
\int_{\pi/2}^{3\pi/2} \bl \chi f_A+f_R\br\bl R\sqrt{\cos^2\phi+(1-\sin\phi)^2}\br\bl 1-\sin\phi\br\di \phi=0
\end{align}
for $R>0$. Note that \eqref{eq:condRingsteadystate} is equivalent to \eqref{eq:condRingsteadystateChigeneral} for $\chi=1$ by symmetry so that the equilibrium of radius  $R\in(0,\frac{d_e}{2}]$ from Proposition \ref{prop:ringsteadystate} satisfies \eqref{eq:condRingsteadystate} and \eqref{eq:condRingsteadystateChigeneral} simultaneously for $\chi=1$. However, \eqref{eq:condRingsteadystate} and \eqref{eq:condRingsteadystateChigeneral} are not satisfied simultaneously for any $R\in (0,\frac{d_2}{2}]$ and any $\chi\in[0,1)$ which concludes the proof.
\end{proof}
Next, we analyze the ellipse pattern.
\begin{corollary}\label{lem:ellipsehelp}
Let $\chi\in[0,1]$ be given and define $$w_1(\phi,R,r)=\sqrt{R^2(1-\cos \phi)^2+(R+r)^2\sin^2\phi},\quad w_2(\phi,R,r)=\sqrt{R^2\sin^2\phi+(R+r)^2\cos^2\phi}.$$
Then, necessary conditions for a stationary ellipse state $\delta_{(R,r)}$ where $R,r\geq 0$ are given by
\begin{align}\label{eq:condEllipsesteadystate1}
\int_0^{\pi} (f_A+f_R)\bl w_1(\phi,R,r)\br R\bl 1-\cos\phi\br w_2(\phi,R,r)\di \phi=0
\end{align}
and
\begin{align}\label{eq:condEllipsesteadystate2}
\int_{\pi/2}^{3\pi/2} \bl \chi f_A+f_R\br\bl w_3(\phi,R,r)\br(R+r)\bl 1-\sin\phi\br w_2(\phi,R,r)\di \phi=0
\end{align}
where $w_3(\phi,R,r)=\sqrt{R^2\cos^2\phi+(R+r)^2(1-\sin\phi)^2}$.
\end{corollary}
\begin{proof}
For ellipse equilibria we require
$(F(\cdot,T)\ast\delta_{(R,r)})((R,0))=0$ implying $$\int_0^{2\pi} F((R(1-\cos \phi),-(R+r)\sin\phi),T)\sqrt{R^2\sin^2\phi+(R+r)^2\cos^2\phi}\di \phi=0.$$
Since $e_2\cdot (F(\cdot,T)\ast\delta_{(R,r)})((R,0))=0$ by symmetry for any $\chi\in[0,1)$ where $e_2=(0,1)$ and $$F(d,T)=\bl f_A(|d|)+f_R(|d|)\br\begin{pmatrix}
1 & 0\\0 & \chi
\end{pmatrix}d$$
this implies that it is sufficient to require \eqref{eq:condEllipsesteadystate1} 
where $$w_2(\phi,R,r)=\sqrt{R^2\sin^2\phi+(R+r)^2\cos^2\phi}.$$
Similarly,  $$(F(\cdot,T)\ast\delta_{(R,r)})((0,R+r))=C\int_0^{2\pi} F((-R\cos \phi,(R+r)(1-\sin\phi)),T)w_2(\phi,R,r)\di \phi=0$$ for a normalization constant $C$ reduces to the necessary condition \eqref{eq:condEllipsesteadystate2}.
\end{proof}

In the sequel, we denote the left-hand side of \eqref{eq:condEllipsesteadystate1} by $G(R,r)$.

\begin{ass}\label{ass:condellipsedecrease}
Given $r\in[0,d_e)$ we assume that there exists $R_{int}\in(0,R_e)$ such that 
\begin{align*}
\frac{\di }{\di R} G(R,r)>0\quad\text{for}\quad R\in\bl 0,R_{int}\br\qquad \text{and}\qquad\frac{\di }{\di R} G(R,r)<0\quad\text{for}\quad R\in\bl R_{int},R_e\br.
\end{align*}
\end{ass}

\begin{rem}
Since $G(0,r)=0$ and Assumption \ref{ass:condellipsedecrease} implies that for $r\in[0,d_e)$ given we have $G(R,r)>0$ for all $R\in\bl 0,R_{int}\br$. Besides, the uniqueness of stationary ellipse states $\delta_{(R,r)}$ for  $r\in[0,d_e)$ given is guaranteed by Assumption \ref{ass:condellipsedecrease}.
\end{rem}
We have the following existence result for nontrivial ellipse states, including rings for $R>0$ and $r=0$.
\begin{corollary}\label{cor:existenceellipse}
Let $r\in[0,d_e)$ and let $R_e>0$ such that
\begin{align}\label{eq:condEllipsesteadystatehelp}
w_1(\phi,R,r)\leq d_e \quad \text{for all}\quad \phi\in [0,\pi],~R\in [0,R_e]
\end{align}
is satisfied and assume that
\begin{align}\label{eq:condEllipsesteadystateR}
\int_0^{\pi} (f_A+f_R)\bl w_1(\phi,R,r)\br \bl 1-\cos\phi\br\bl R_e^2+R_e r\cos^2\phi\br \di \phi<0.
\end{align}
holds. Further define $$G_1(R,r)=\int_0^{\pi}(f_A+f_R)\bl w_1(\phi,R,r)\br \bl 1-\cos\phi\br\di \phi$$
and $$G_2(R,r)=\int_0^{\pi}(f_A+f_R)\bl w_1(\phi,R,r)\br \bl 1-\cos\phi\br\cos^2\phi\di \phi.$$
If  $r$  satisfies 
\begin{align}\label{eq:condEllipsesteadystaterpos}
\min\left\{G_1(0,r),G_2(0,r)\right\}>0
\end{align}
there exists an $R\in(0,R_e)$ such that the necessary condition \eqref{eq:condEllipsesteadystate1}  for a nontrivial stationary ellipse state $\delta_{(R,r)}$  to the mean-field PDE \eqref{eq:macroscopiceq} are satisfied. For $r$ satisfying
\begin{align}\label{eq:condEllipsesteadystaterneg}
\max\left\{G_1(0,r),G_2(0,r)\right\}<0
\end{align}
there exists no $R\in(0,R_e)$ such that the ellipse $\delta_{(R,r)}$ is an equilibrium to the mean-field PDE \eqref{eq:macroscopiceq} and the trivial ellipse state $\delta_{(0,r)}$ is the only equilibrium.
If, for $r\in[0,d_e)$ given, Assumption~\ref{ass:condellipsedecrease} is satisfied, then there exists a unique $R\in (R_{int},R_e)$ such that the necessary condition \eqref{eq:condEllipsesteadystate1} for a nontrivial stationary ellipse state $\delta_{(R,r)}$ is satisfied.
\end{corollary}
\begin{rem}
Condition \eqref{eq:condEllipsesteadystatehelp} is related to the assumption that $f_A+f_R$ is strictly decreasing on $[0,d_e]$ in Assumption \eqref{eq:condEllipsesteadystatehelp}. Condition \eqref{eq:condEllipsesteadystateR} can be interpreted as the long-range attraction forces being larger than the short-range repulsion forces. Besides, given $r\in[0,d_e)$ condition \eqref{eq:condEllipsesteadystaterneg} can be interpreted as the attractive forces being too strong for the existence of a stationary ellipse patterns $\delta_{(0,r)}$ and hence for any stationary ellipse pattern $\delta_{(R,r)}$ for $R\geq 0$ because the forces are attractive for  $R$ sufficiently large. Condition \eqref{eq:condEllipsesteadystaterpos} implies that the forces are too repulsive along the vertical axis for a stationary ellipse state $\delta_{(0,r)}$, but as $R$ increases the forces become more attractive which may result in  stationary ellipse state $\delta_{(R,r)}$ for $R>0$. Assumption \ref{ass:condellipsedecrease} relaxes condition \eqref{eq:condEllipsesteadystaterpos}, but  requires additionally that $G(\cdot,r)$ first increases and then decreases to guarantee the uniqueness of a stationary ellipse pattern. In Figure \ref{fig:evalrhsradiuspairs} the function $G$ is evaluated for certain values of $r\in[0,d_e)$ for the forces in the K\"ucken-Champod model \eqref{eq:particlemodel} and one can clearly see that Assumption \ref{ass:condellipsedecrease} is satisfied and there exists a unique zero $R>0$, as stated in Corollary \ref{cor:existenceellipse}. 
\end{rem}
\begin{proof}
Let $r\in(0,d_e)$ be given. Note that the left-hand side of \eqref{eq:condEllipsesteadystatehelp} is equal to $w_1(\phi,R,r)$ for all $\phi\in[0,\pi]$ and $w_1(\phi,R,r)\in[0,\max\{2R,R+r\}]$ for all $\phi\in[0,\pi]$. Since $f_A+f_R$ is strictly decreasing on $[0,d_e]$ by Assumption \ref{ass:propertyforceadditional} we only consider $R\geq 0$ such that $w_1(\phi,R,r)\in [0,d_e]$ for all $\phi\in[0,\pi]$. Clearly, there exists $R_e>0$ such that \eqref{eq:condEllipsesteadystatehelp}  is satisfied.

Since $w_2\bl \phi,R,r\br\sim R+r\cos^2\phi$ we  approximate \eqref{eq:condEllipsesteadystate1} by
\begin{align}\label{eq:condEllipsesteadystate1new}
\int_0^{\pi} (f_A+f_R)\bl w_1(\phi,R,r)\br \bl 1-\cos\phi\br\bl R^2+Rr\cos^2\phi\br \di \phi=0
\end{align}
Note that $(f_A+f_R)\bl w_1(\phi,R,r)\br \bl 1-\cos\phi\br$ for $\phi\in(0,\pi)$ is strictly decreasing as a function of $R$ because $f_A+f_R$ is a strictly decreasing function by Assumption \ref{ass:propertyforceadditional} and $w_1(\phi,R,r)$ is strictly increasing in $R$ for $\phi\in(0,\pi)$ fixed. Hence, $G_1(R,r)$ is strictly decreasing in $R$ and has a unique zero $R_1\in[0,R_e]$, provided $r\geq 0$ satisfies $G_1(0,r)>0$ and \eqref{eq:condEllipsesteadystateR}. Similarly, one can argue that $G_2(R,r)$ is strictly decreasing in $R$ and has a unique zero $R_2\in[0,R_e]$ if $r\geq 0$ such that $G_2(0,r)>0$ and \eqref{eq:condEllipsesteadystateR} are satisfied. The left-hand side of \eqref{eq:condEllipsesteadystate1new} is the rescaled sum of $G_1$ and $G_2$ where $\text{id}^2\cdot G_1(\cdot,r)$ as a function of $R$ is nonnegative on $[0,R_1]$ and negative on $(R_1,R_e]$, while $\text{id}\cdot rG_2(\cdot,r)$ as a function of $R$ is nonnegative on $[0,R_2]$ and negative on $(R_2,R_e]$. In particular, the left-hand side of \eqref{eq:condEllipsesteadystate1new} has a  zero $R\in[\min\{R_1,R_2\},\max\{R_1,R_2\}]$ on $(0,R_e)$ if $r\geq 0$  satisfies \eqref{eq:condEllipsesteadystaterpos} and \eqref{eq:condEllipsesteadystateR}, while there exists no zero on $(0,R_e)$ if $r\geq 0$  satisfies \eqref{eq:condEllipsesteadystaterneg}. If Assumption \ref{ass:condellipsedecrease} is satisfied, then $G(\cdot,r)$ with $r\in[0,d_e)$ given has a zero at $R=0$ and at an $R\in(0,d_e)$ because $G(\cdot,r)>0$ on $(0,R_{int})$, $G(\cdot,r)$ strictly decreasing on $(R_{int},R_e)$ and $G(R_e,r)<0$ by \eqref{eq:condEllipsesteadystateR}. This concludes the proof.
\end{proof}
Since the equilibrium condition \eqref{eq:condEllipsesteadystate1} for trivial ellipse states with $R=0$ is clearly satisfied for all $r\geq 0$ we rewrite $G(R,r)=Rg(R,r)$ for a smooth function $g$ and require $g(0,r)=0$.  Since we are interested in nontrivial states, i.e. $r>0$, we define 
\begin{align*}
\bar{g}(r)=\int_0^{\pi} (f_A+f_R)\bl r|\sin\phi|\br \bl 1-\cos\phi\br  |\cos\phi|\di \phi=0
\end{align*}
and
and it is sufficient to require $\bar{g}(r)=0$ for an $r>0$. Note that $\bar{g}(r)>0$ for all $r\in(0,d_a]$ since $f_A+f_R$ is repulsive on $[0,d_a]$. Assuming that $\bar{g}(d_e)<0$ which is a natural condition for long-range attraction forces being stronger than short-range repulsive forces there exists a unique $\bar{r}\in(0,d_e)$ such that $\bar{g}(\bar{r})=0$ because $\bar{g}$ strictly decreases on $(0,d_e)$. Besides, the necessary condition \eqref{eq:condEllipsesteadystate2} reduces to 
\begin{align*}
\int_{\pi/2}^{3\pi/2} \bl \chi f_A+f_R\br\bl \bar{r}|1-\sin\phi|\br\bl 1-\sin\phi\br |\cos\phi|\di \phi=0.
\end{align*}
Since $f_A\leq 0$ and $f_R\geq 0$ by the definition of the attractive and repulsive force, cf. Assumption~\ref{ass:propertyforce}, there exists a unique $\bar{\chi}\in(0,1)$ such that condition \eqref{eq:condEllipsesteadystate2} is satisfied, given by
\begin{align}\label{eq:chitrivialellipse}
\bar{\chi}=-\frac{\int_{\pi/2}^{3\pi/2} f_R\bl \bar{r}|1-\sin\phi|\br\bl 1-\sin\phi\br |\cos\phi|\di \phi}{\int_{\pi/2}^{3\pi/2}  f_A\bl \bar{r}|1-\sin\phi|\br\bl 1-\sin\phi\br |\cos\phi|\di \phi}>0.
\end{align}
Note that $\bar{\chi}<1$ by the assumption that the long-range attraction forces are stronger than the short-range repulsive forces. In summary, we have the following result.
\begin{lemma}\label{lem:pseudostate}
There exists a unique $\bar{r}\in(0,d_e)$ such that the necessary condition \eqref{eq:condEllipsesteadystate1} for a stationary ellipse state $\delta_{(0,\bar{r})}$   with $\bar{g}(\bar{r})=0$ is satisfied. In this case, the second necessary condition \eqref{eq:condEllipsesteadystate2} is satisfied for a unique $\bar{\chi}\in[0,1]$, defined by \eqref{eq:chitrivialellipse}. 
\end{lemma}

\begin{ass}\label{ass:ellipsetuples}
Assume that 
\begin{enumerate}
\item If $G(\tilde{R},\tilde{r})=0$ for $\tilde{R}>0,~\tilde{r}\geq 0$, then  $G(\tilde{R},r)<0$.\label{ass: rdependenceass1}
\item There exists $R>0$ such that $G(R,0)<0$.\label{ass: rdependenceass2}
\item For all $R>0$ there exists $r\geq 0$ such that $G(R,r)<0$.\label{ass: rdependenceass3}
\end{enumerate}
\end{ass}

\begin{rem}
Note that \eqref{ass: rdependenceass1} in Assumption~\ref{ass:ellipsetuples} implies that if the equilibrium condition for an ellipse state is satisfied for a specific tuple $(\tilde{R},\tilde{r})$, then the forces are too attractive for any ellipse state $(\tilde{R},r)$ with longer major axis $\tilde{R}+r\geq \tilde{R}+\tilde{r}$ for $r\geq \tilde{r}$. Condition \eqref{ass: rdependenceass2} in Assumption~\ref{ass:ellipsetuples} together with Assumption~\ref{ass:condellipsedecrease} implies the existence of a ring equilibrium. Besides, \eqref{ass: rdependenceass3} in Assumption~\ref{ass:ellipsetuples} states that for an ellipse state with a minor axis of length $R>0$ one can choose the major axis $R+r$ sufficiently long so that the given forces are too attractive for the ellipse state $\delta_{(R,r)}$ to be stationary. Note that one can easily check that these assumptions are satisfied for the forces in the K\"ucken-Champod model with parameters in \eqref{eq:parametervaluesRepulsionAttraction}.
\end{rem}

\begin{proposition}\label{prop:ellipsetuples}
Let $0\leq r_1<r_2<d_e$ and let $R_1,R_2\geq 0$ such that $$w_1(\phi,R,r)\leq d_e\quad\text{for all}\quad \phi\in [0,\pi],\quad R\in [0,\max\{R_1,R_2\}]$$
and the necessary condition \eqref{eq:condEllipsesteadystate1} for  $\delta_{(R_1,r_1)}$ and $\delta_{(R_2,r_2)}$ being stationary ellipse states are satisfied. Suppose that Assumption \ref{ass:ellipsetuples} and Assumption \ref{ass:condellipsedecrease} hold. Then, $R_1<R_2$, i.e. the longer the major axis of the stationary ellipse state, the shorter the minor axis. Besides, there exists a continuous function $q(t)=(R(t),r(t))$ for $t\in[0,1]$ where $R(t)$ is strictly decreasing, $r(t)$ is strictly increasing, $q(0)=(0,\bar{r})$ for the pseudo-ellipse state $\delta_{(0,\bar{r})}$ with $\bar{r}>0$ in Lemma \ref{lem:pseudostate} and $q(1)=(\bar{R},0)$ for the unique ring state of radius $\bar{R}$ in Proposition \ref{prop:ringsteadystate}. 
\end{proposition}
\begin{proof}
Note that $G(0,r)=0$ for all $r\geq 0$. Further note that $(f_A+f_R)(0)>0$ since $F$ is a short-range repulsive, long-range attractive force by Assumption \ref{ass:propertyforce}, implying that for all $R\in(0,d_a/4]$ and all $r\in[0,d_a/4]$ we have $G(R,r)>0$. By continuity and since $G(R,0)<0$ for some $R>0$ there exists $\tilde{R}>0$ such that $G(\tilde{R},0)=0$. Besides, Assumption \ref{ass:ellipsetuples} implies that $G(\tilde{R},r)<0$ for all $r>0$. In particular, $G(\tilde{R},r_1)<0$ and $G(\tilde{R},r_2)<0$ for $r_2>r_1>0$  implies together  with Assumption \ref{ass:condellipsedecrease} that there exists a unique $\tilde{R}_1\in[0,\tilde{R})$ such that $G(\tilde{R}_1,r_1)=0$ which implies that $G(\tilde{R}_1,r_2)<0$ and that there exists $\tilde{R}_2\in[0,\tilde{R}_1)$ such that $G(\tilde{R}_2,r_2)=0$.
\end{proof}
In Figure \ref{fig:rvsRcontinuous} the tuples $(R,r)$ are plotted such that the necessary condition \eqref{eq:condEllipsesteadystate1} for ellipse equilibria is satisfied. In particular, these tuples $(R,r)$ can be determined independently from $\chi$ from \eqref{eq:condEllipsesteadystate1}.

\begin{corollary}\label{cor:chidependenceellipse}
Let $H(R,r,\chi)$ denote the left-hand side of \eqref{eq:condEllipsesteadystate2} and assume that $H(q_1,q_2,1)$ is strictly increasing where the function $q(t)=(q_1(t),q_2(t))$, $t\in[0,1]$, is defined in Proposition \ref{prop:ellipsetuples}.
For every tuple $(R,r)$ with $R,r\geq 0$ such that the condition \eqref{eq:condEllipsesteadystate1} is satisfied there exists a unique $\chi\in[0,1]$ so that \eqref{eq:condEllipsesteadystate2} is also satisfied. If additionally $H(q_1,q_2,\chi)$ for all $\chi\in [\bar{\chi},1]$ then there exists a unique tuple $(R,r)$ such that the corresponding ellipse pattern $\delta_{(R,r)}$ is an equilibrium for any given $\chi\in[\bar{\chi},1]$. In particular, there exists a continuous, strictly increasing function $p=p(t)$ for $t\in[0,1]$ with $p(0)=\bar{\chi}$ and $p(1)=1$ such that for $t\in[0,1]$ given the ellipse state $\delta_{(q_1(t),q_2(t))}$ is stationary for a unique value of the parameter $\chi=p(t)$. In other words, the smaller the value of $\chi\in[\bar{\chi},1]$ the longer the major and the shorter the minor axis for ellipse equilibria, i.e. the smaller the value of $\chi$ the more the  ellipse is stretched along the vertical axis.
\end{corollary}
\begin{proof}
Note that \eqref{eq:condEllipsesteadystate2} can be rewritten as 
\begin{align}\label{eq:condEllipsesteadystate2rewrite}
\int_{0}^{\pi} \bl \chi f_A+f_R\br\bl w_3(\phi+\pi/2,R,r)\br(R+r)\bl 1-\cos\phi\br w_2(\phi+\pi/2,R,r)\di \phi=0
\end{align}
where $$w_3(\phi+\pi/2,R,r)=\sqrt{R^2\sin^2\phi+(R+r)^2(1-\cos\phi)^2}.$$ In particular, \eqref{eq:condEllipsesteadystate2rewrite} is equal to \eqref{eq:condEllipsesteadystate1} for $\chi=1$ and $r=0$, i.e. $H(q_1(1),q_2(1),1)=0$. However, for any tuple $(R,r)$ with $r>0$ satisfying \eqref{eq:condEllipsesteadystate1} we have $H(R,r,1)<0$ since $H(q_1,q_2,1)$ is strictly increasing on $[0,1]$ and $H(q_1(1),q_2(1),1)=0$. Besides, $H(q_1,q_2,0)>0$ on $[0,1]$ since by the definition of the repulsive force coefficient in Assumption \ref{ass:propertyforce} we have $1-\cos\phi\geq 0$ on $[0,\pi]$, $f_R\geq 0$ and $w_2\geq 0$.  Since  $H(q_1(t),q_2(t),\cdot)$ is strictly decreasing as a function of $\chi$ for any $t\in[0,1]$ fixed by the properties of the attractive force coefficient in Assumption \ref{ass:propertyforce} for each $t\in[0,1]$ there exists a unique $\chi\in[0,1]$ by continuity of $H$ such that the tuple  $q(t)=(q_1(t),q_2(t))$ satisfies condition \eqref{eq:condEllipsesteadystate2}.

To show that for any  $\chi\in[\bar{\chi},1]$ there exists a unique tuple $(R,r)$ such that $\delta_{(R,r)}$ is a stationary ellipse state note that $H(\bar{R},0,1)=0$ by the definition of $\bar{R}$ in Proposition \ref{prop:ringsteadystate} and $H(\bar{R},0,\chi)>0$ for $\chi\in(0,1]$ since $H(\bar{R},0,\cdot)$ strictly decreasing. Similarly, $H(0,\bar{r},\bar{\chi})=0$ and $H(0,\bar{r},\chi)<0$ for all $\chi\in(\bar{\chi},1]$. Since  $H(q_1,q_2,\chi)$ is strictly increasing for any $\chi\in[\bar{\chi},1]$ by assumption the function  $H(q_1,q_2,\chi)$ for $\chi\in[\bar{\chi},1]$ fixed has a unique zero, i.e. there exists a unique tuple $(R,r)$ such that $\delta_{(R,r)}$ is a stationary ellipse state. Besides, if $\delta_{(R_1,r_1)}$ and $\delta_{(R_2,r_2)}$ are stationary ellipse states with $R_1<R_2$ and $r_1>r_2$ for $\chi_1,\chi_2\in[\bar{\chi},1]$, respectively, then $\chi_1<\chi_2$ since there exist $t_1,t_2\in[0,1]$ with $t_1<t_2$ such that $q(t_1)=(R_1,r_1)$ and $q(t_2)=(R_2,r_2)$ and $H(q_1,q_2,\chi)$ strictly increasing for any $\chi\in[\bar{\chi},1]$.
\end{proof}

In Figure \ref{fig:evalrhsradiuspairschicontinuous} the functional $H(q_1,q_2,\chi)$ is evaluated for different values of $\chi$ and one can see that for every $\chi$ there exists a unique tuple $(R,r)$ such that the equilibrium condition \eqref{eq:condEllipsesteadystate2} is satisfied. The eccentricity $e=\sqrt{1-(R/(R+r))^2}$ of the ellipse is illustrated as a function of $\chi$ in Figure \ref{fig:eccentricitycontinuous} and one can see how the eccentricity increases as $\chi$ decreases which corresponds to the evolution of the ring pattern into a stationary ellipse pattern whose minor axis becomes shorter and whose major axis becomes longer as $\chi$ decreases, proven in Corollary \ref{cor:chidependenceellipse}.

\begin{figure}[htbp]
    \centering
    \subfloat[Evaluation of LHS of \eqref{eq:condEllipsesteadystate1}]{\includegraphics[width=0.45\textwidth]{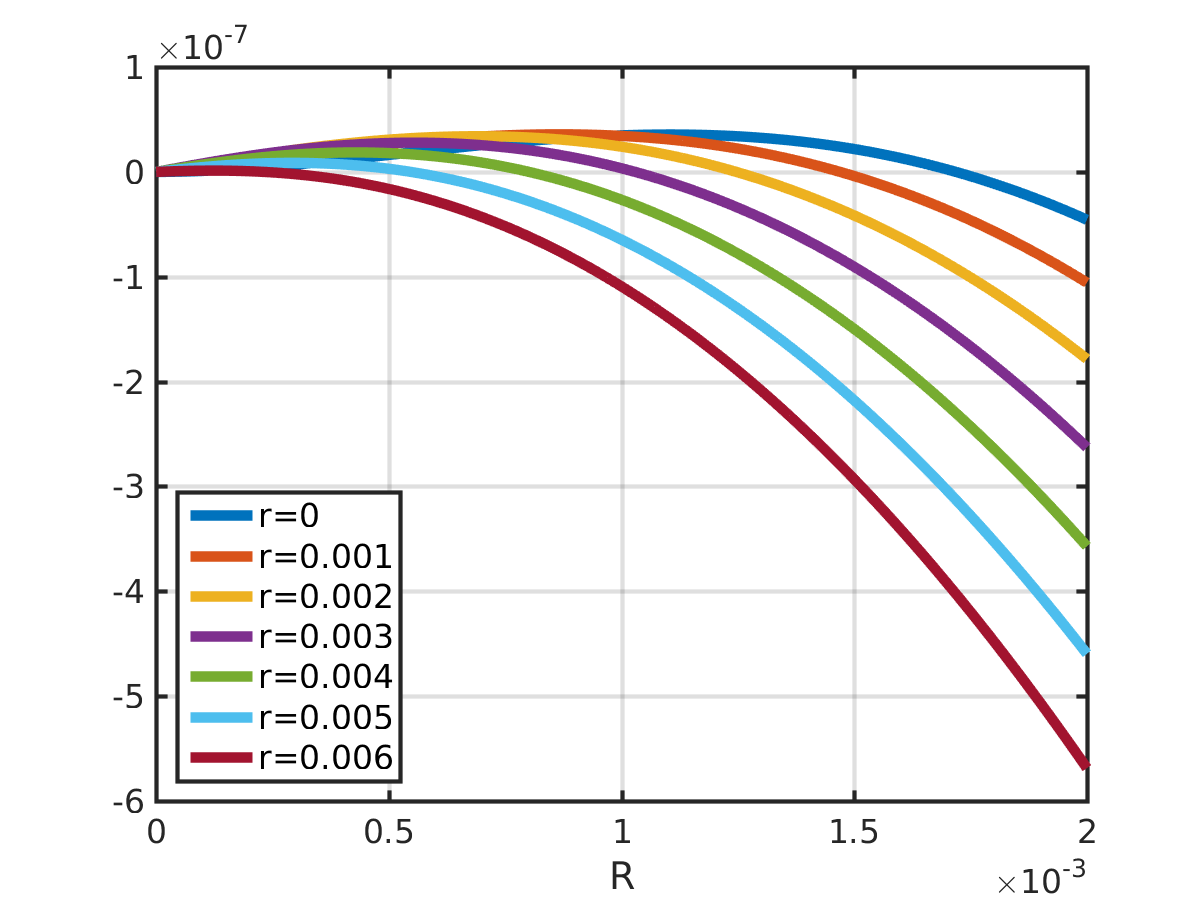}\label{fig:evalrhsradiuspairs}}
\subfloat[Tuples $(R,r)$] {\includegraphics[width=0.45\textwidth]{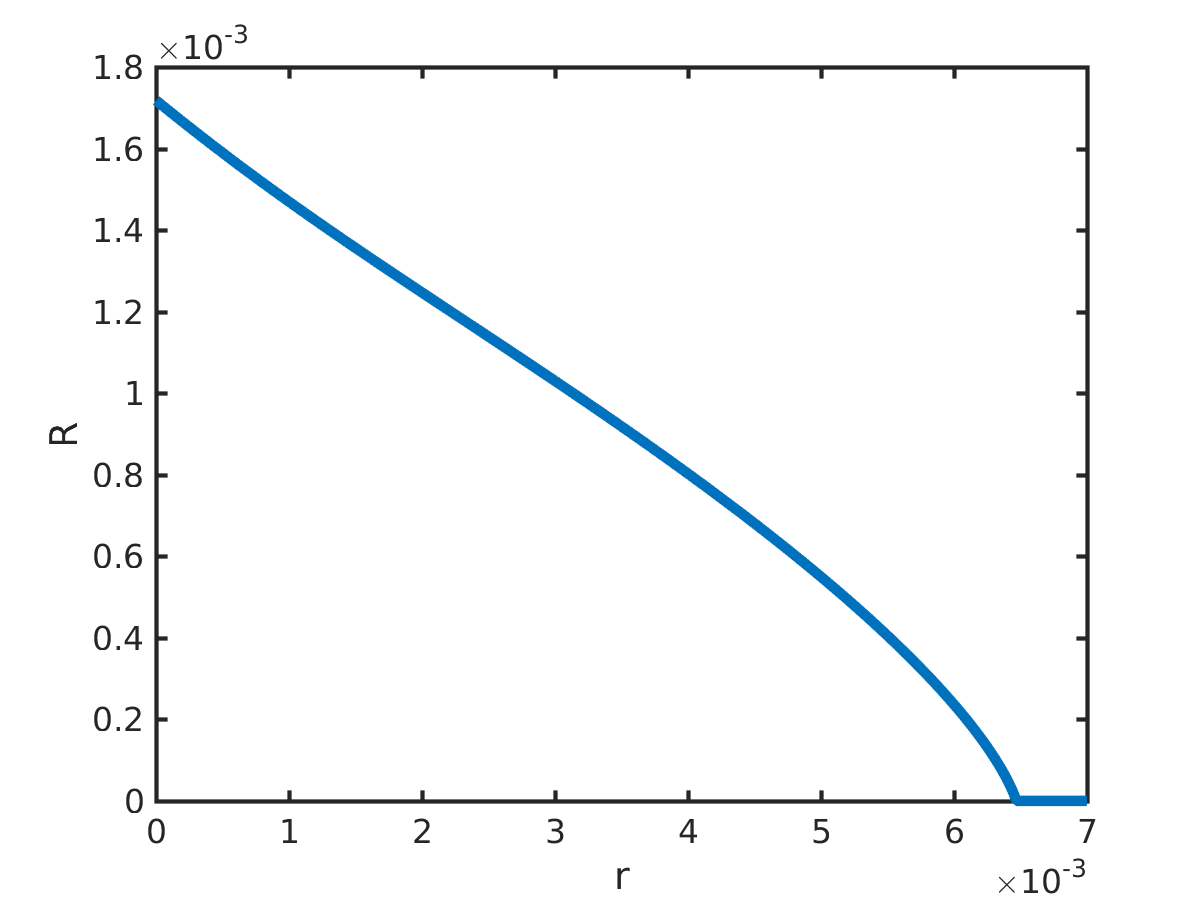}\label{fig:rvsRcontinuous}}

\subfloat[Evaluation of LHS of \eqref{eq:condEllipsesteadystate2}] {\includegraphics[width=0.45\textwidth]{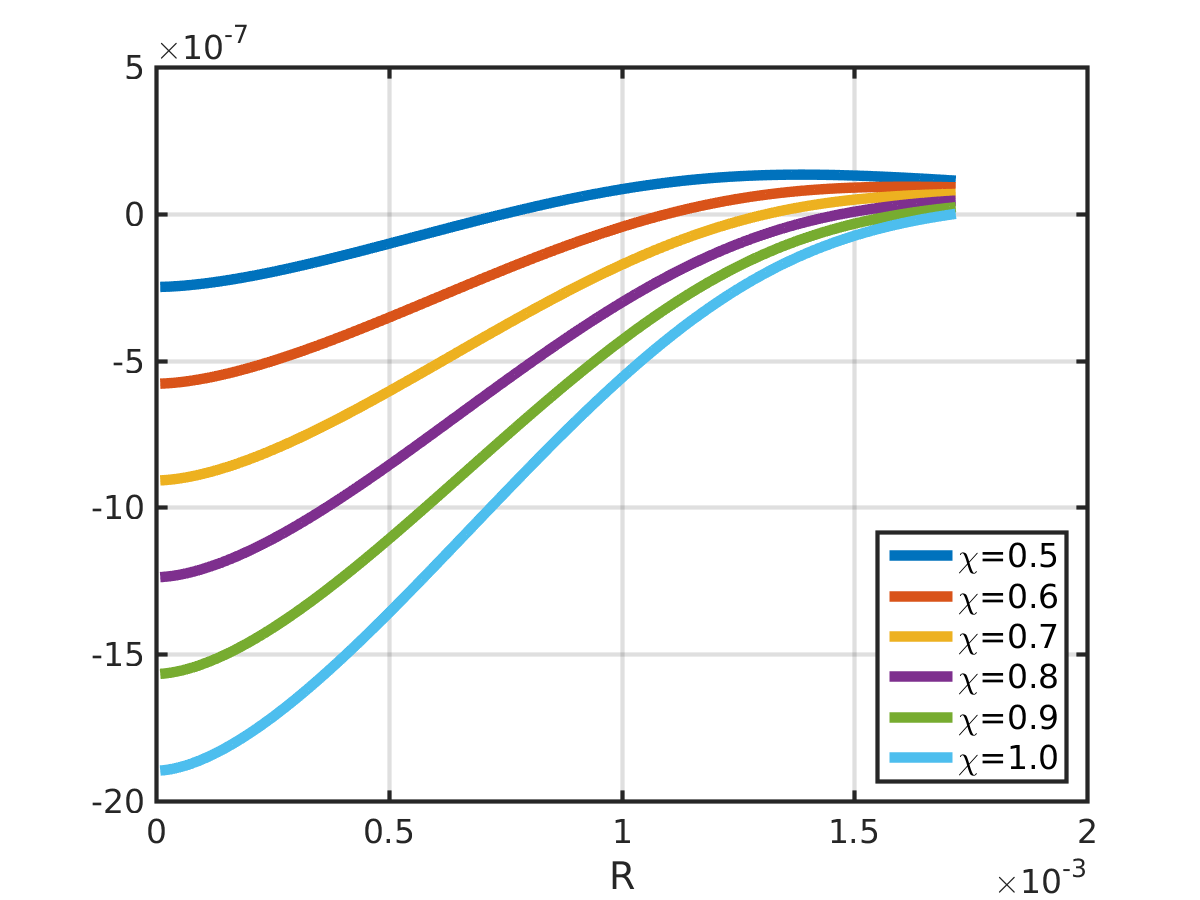}\label{fig:evalrhsradiuspairschicontinuous}}
\subfloat[$e=e(\chi)$] {\includegraphics[width=0.45\textwidth]{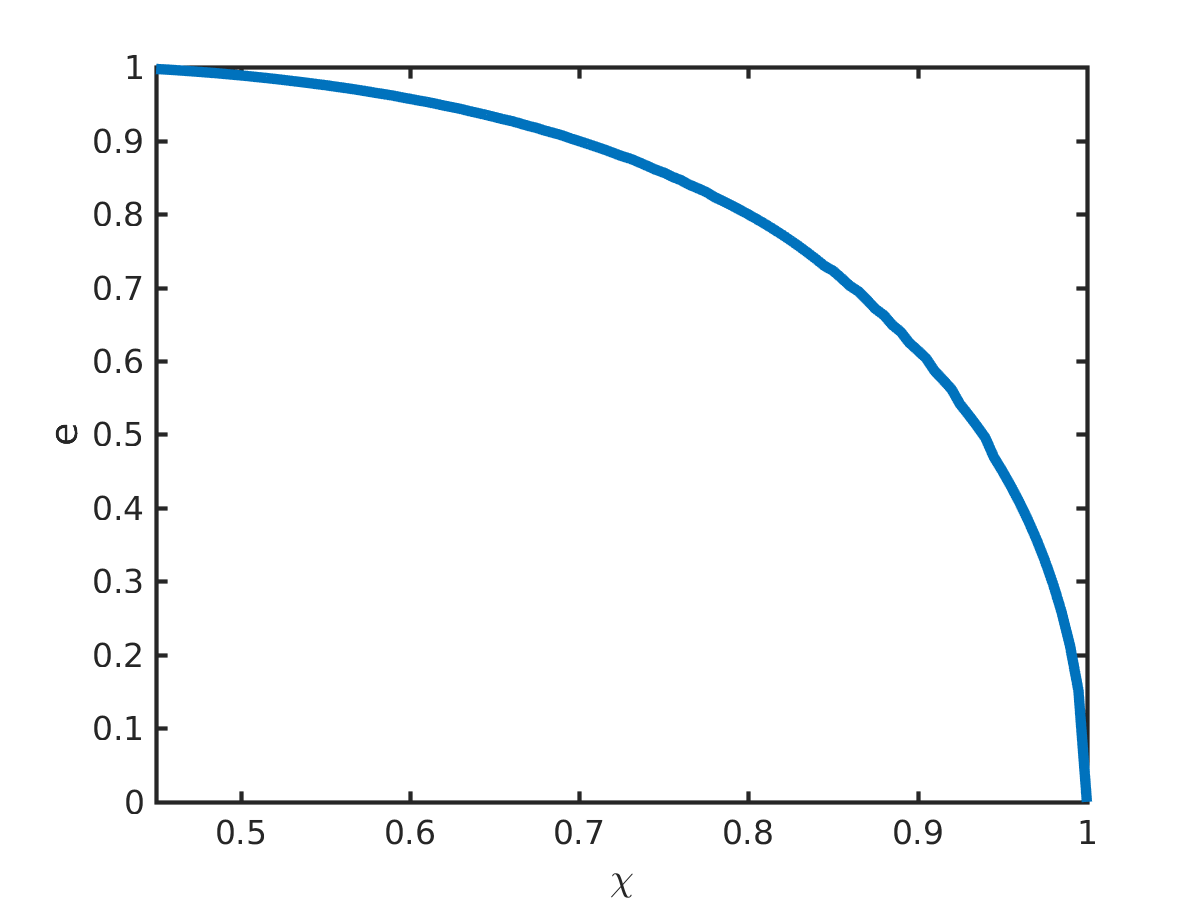}\label{fig:eccentricitycontinuous}}
    \caption{Tuples $(R,r)$ for stationary ellipse patterns to the mean-field equation \eqref{eq:macroscopiceq} satisfying equilibrium conditions \eqref{eq:condEllipsesteadystate1} and \eqref{eq:condEllipsesteadystate2} for different values of $\chi$ and eccentricity $e$ as a function of $\chi$ for the forces in the K\"ucken-Champod model for parameter values in \eqref{eq:parametervaluesRepulsionAttraction}}
\end{figure}

\subsubsection{Stripe pattern}
Based on the discussion in Section \ref{sec:interpretationforce} for the tensor field $T=\chi s \otimes s+l\otimes l$ with $l=(1,0)$ and $s=(0,1)$, we consider different shapes of vertical stripe patterns in $\R^2$ and discuss whether they are equilibria.

\begin{definition}
Let  the center of mass be denoted by $x_c=(x_{c,1},x_{c,2})\in \R^2$. Then we
define the measure $\delta_{\bl x_{c,1},\cdot\br}$ by
\begin{align*}
\delta_{\bl x_{c,1},\cdot\br}(A)= \lambda\bl A\cap\bl\{x_{c,1}\}\times\R\br\br
\end{align*}
for all measurable sets $A\subset\R^2$ where $\lambda$ denotes the one-dimensional Lebesgue measure.
\end{definition}

The  measure  $\delta_{\bl x_{c,1},\cdot\br}$ is a locally finite measure, but not a probability measures and satisfies condition \eqref{eq:steadystate} for equilibria of the mean-field PDE \eqref{eq:macroscopiceq} for any force satisfying Assumption \ref{ass:propertyforce} and any $\chi\in [0,1]$ since  $F(x-x',T)=-F(-(x-x'),T)$ for all $x,x'\in\R^2$. Note that fully repulsive forces along the vertical axis are necessary for the occurrence of stable stripe patterns $\delta_{\bl x_{c,1},\cdot\br}$. Further note that as $\chi$ decreases the attraction forces disappear along the vertical direction and the mass leaks to infinity driven by purely repulsive forces along the vertical axis so that $\delta_{\bl x_{c,1},\cdot\br}$ cannot be the limit of an ellipse pattern. Hence,   vertical lines are not stable equilibria with Definition \ref{def:steadystate} for the K\"ucken-Champod model \eqref{eq:particlemodel} posed in the plane. 

To obtain measures  concentrating on vertical lines as solutions to the  K\"ucken-Champod model \eqref{eq:particlemodel} and to guarantee the conservation of mass under the variation of parameter $\chi$, 
we consider the associated probability measure on the two-dimensional unit torus $\mathbb{T}^2$ instead of the full space $\R^2$. Another possibility to obtain measures concentrating on vertical lines as solutions is to consider confinement forces, see \cite{Mora}.

Solutions to the mean-field PDE \eqref{eq:macroscopiceq} satisfying condition \eqref{eq:steadystate} include  measures which are uniformly distributed on certain intervals along the vertical axis, i.e. on $\{x=(x_1,x_2)\in\R^2\colon x_1=x_{c,1},~x_2\in [a,b]\}$ for some constants $a<b$, as well as  measures which are uniformly distributed on  unions of distinct intervals. The former occur if the total force is repulsive-attractive so that the attraction force restricts the stretching of the solution to certain subsets of the vertical axis. The latter which look like  dashed lines parallel to the vertical axis can be realized by repulsive-attractive-repulsive forces, i.e. repulsive-attractive forces may lead to accumulations on subsets of the vertical axis while the additional repulsion force acting on long distances is responsible for the separation of the different subsets.  

After considering these one-dimensional patterns, the question arises whether the corresponding two-dimensional vertical stripe pattern of width $\Delta$ satisfies the equilibrium condition \eqref{eq:steadystate} for any $\Delta>0$.
Let $\Delta>0$ and consider the two-dimensional vertical stripe pattern of width $\Delta$, given by
\begin{align*}
g_{\Delta}(x)=g_{\Delta}(x_1,x_2)=\begin{cases}\frac{1}{\Delta},& x_1\in \left[x_{c,1}-\frac{\Delta}{2},x_{c,1}+\frac{\Delta}{2}\right],\\0, & \text{otherwise}.
\end{cases}
\end{align*}
We assume that $g_{\Delta}$ satisfies the equilibrium condition \eqref{eq:steadystate} for the mean-field PDE \eqref{eq:macroscopiceq}, i.e. 
$g_{\Delta} \bl F\ast g_{\Delta}\br=0,$
implying
$$\int_{\left[x_{c,1}-\frac{\Delta}{2},x_{c,1}+\frac{\Delta}{2}\right]\times \R}F(x-x',T)\di x'=0 \quad\text{for all}\quad x\in\left[x_{c,1}-\frac{\Delta}{2},x_{c,1}+\frac{\Delta}{2}\right]\times \R.$$
By linear transformations this reduces to
$$\int_{\left[-\frac{\Delta}{2},\frac{\Delta}{2}\right]\times \R}F((x_1,0)-x',T)\di x'=0 \quad\text{for all}\quad x_1\in\left[-\frac{\Delta}{2},\frac{\Delta}{2}\right].$$
Since $F(x-x',T(x))=-F(-(x-x'),T(x))$ for all $x,x'\in\R^2$ we have
$$e_1\cdot \int_{\left[-\frac{\Delta}{2},\frac{\Delta}{2}\right]\times \R}F((x_1,0)-x',T)\di x'=0 \quad\text{for all}\quad x_1\in\left[-\frac{\Delta}{2},\frac{\Delta}{2}\right] $$
and symmetry implies 
\begin{align}\label{eq:condStripePattern}
e_1\cdot\int_{\left[x_1,\Delta-x_1\right]\times \R}F(x',T)\di x'=0 \quad\text{for all}\quad x_1\in\left[0,\frac{\Delta}{2}\right).
\end{align}
Hence the equilibrium state  can only occur for special choices of the interaction force $F$. In general, \eqref{eq:condStripePattern} is not satisfied and thus $g_{\Delta}$ is not an equilibrium state  of the mean-field PDE.

\section{Numerical methods and results}\label{sec:numerics}
In this section, we investigate the long-time behavior of solutions to the K\"{u}cken-Champod model \eqref{eq:particlemodel} and the pattern formation process numerically and we discuss the numerical results by comparing them  to the analytical results of the model in Section \ref{sec:analysis}. These numerical simulations are necessary for getting a better understanding of the long-time behavior of solutions to the K\"{u}cken-Champod model \eqref{eq:particlemodel} and its stationary states.  Since the mean-field limit shows that the particle method is convergent with a order given by $N^{-1/2}\ln(1+N)$ \cite{convergenceWasserstein, Golse} it is sufficient to use particle simulations instead of the mean-field solvers.

We consider the domain $\Omega=\mathbb{T}^2$ where $\mathbb{T}^2$ is  the $2$-dimensional unit torus that can be identified with the unit square $[0,1)\times [0,1)\subset \R^2$ with periodic boundary conditions. 
To guarantee that particles  can only interact within a finite range we assume that they cannot interact with each other if they are separated by a distance  of at least 0.5 in each spatial direction, i.e. for $i\in\{1,2\}$ and all $x\in \Omega$ we require that $F(x-x',T(x))\cdot e_i=0$ for $|x-x'|\geq 0.5$ where $e_i$ denotes the standard basis for the Euclidean plane. This  property of the total interaction force $F$ in \eqref{eq:totalforce} is 
 referred to as the minimum image criterion \cite{periodicbc}.
Note that the coefficient functions $f_R$ and $f_A$  in \eqref{eq:repulsionforcemodel} and \eqref{eq:attractionforcemodel} in the K\"{u}cken-Champod model \eqref{eq:particlemodel} satisfy the minimum image criterion 
if a spherical cutoff radius of length $0.5$ is introduced for the repulsion and attraction forces.

\begin{rem}[Minimum image criterion]
The minimum image criterion is a natural condition for large systems of interacting particles on a domain with periodic boundary conditions. In numerical simulations, it is sufficient to record and propagate only the particles in the original simulation box. Besides, the minimum image criterion guarantees that the size of the domain is large enough compared to the range of the total force. In particular, non-physical artifacts due to periodic boundary conditions are prevented.
\end{rem}

\subsection{Numerical methods}\label{sec:numericalmethods}

To solve the $N$ particle ODE system \eqref{eq:particlemodel} we consider periodic boundary conditions and apply either the simple explicit Euler scheme or higher order methods such as the Runge-Kutta-Dormand-Prince method, all resulting in very similar simulation results.

\subsection{Numerical results}\label{sec:numericalresults}
We show numerical results for the K\"{u}cken-Champod model \eqref{eq:particlemodel} on the domain $\Omega=T^2$ where the force coefficients are given by \eqref{eq:repulsionforcemodel} and \eqref{eq:attractionforcemodel}. In particular, we investigate the patterns of the corresponding stationary solutions. Unless stated otherwise we consider the parameter values in \eqref{eq:parametervaluesRepulsionAttraction} and the spatially homogeneous tensor field  $T=\chi s \otimes s+l\otimes l$ with $l=(1,0)$ and $s=(0,1)$. Besides, we assume that the initial condition is a Gaussian with mean $\mu=0.5$ and  standard deviation $\sigma=0.005$ in each spatial direction.

\subsubsection{Dependence on the initial distribution}
The stationary solution to \eqref{eq:particlemodel} for $N=1200$ particles is shown in Figure \ref{fig:numericalsol_initial} for $\chi=0.2$ and $\chi=0.7$, respectively, for different initial data. One can clearly see that the long-time behavior of the solution depends on the chosen initial conditions and the choice of $\chi$. As discussed in Section \ref{sec:interpretationforce} the absence of attraction forces along $s=(0,1)$ for $\chi=0.2$ leads to a solution stretched along the entire vertical axis and particles in a neighborhood of these line patterns are attracted. For $\chi=0.7$ the domain of attraction is significantly smaller and the particles remain isolated or build small clusters  if they are initially too far apart from  other particles. This results in many accumulations of smaller numbers of particles for $\chi=0.7$. Note that these accumulations have the shape of ellipses for $\chi=0.7$ which is consistent with the analysis in Section \ref{sec:analysis}, independent of the choice of the initial data. Because of the significantly larger number of clusters for randomly uniformly distributed initial data the resulting ellipse patterns consist of fewer particles compared to Gaussian initial data with a small standard deviation. 
Since initial data  spread over the entire simulation domain leads to multiple copies of the patterns which occur for concentrated initial data, this motivates to consider concentrated initial data for getting a better understanding of the patterns which can be generated. In the sequel we restrict ourselves to concentrated initial data so that all particles can initially interact with each other. Besides, it is sufficient to consider smaller numbers of particles to get a better understanding of the formation of the stationary pattern to increase the speed of convergence. Further note that  for $\chi=0.2$ and randomly uniformly distributed initial data the convergence to the stationary solution, illustrated in Figure \ref{fig:numericalsol_uniform02}, is very slow which implies that the fingerprint formation might also be slow. However, the K\"ucken-Champod model \eqref{eq:particlemodel} is able to generate very interesting patterns over time $t$, as shown in Figure \ref{fig:numericalsol_uniform02}.  
Besides, it is of interest how the resulting patterns depend on the initial data and whether the ellipse pattern is stable for $\chi=0.7$. In Figure \ref{fig:numericalsol_initialnoperturb} we consider $N=600$ particles and Gaussian initial data with mean $\mu=0.5$ and standard deviation $\sigma=0.005$ in each spatial direction. Given the initial position of the particles for the simulation in Figure \ref{fig:numericalsol_initialnoperturb} we perturb the initial position of each particle $j$ by $\delta Z_j$ where $Z_j$ is drawn from a bivariate standard normal distribution and $\delta\in\{0.0001, 0.001, 0.01, 0.1\}$. The corresponding stationary patterns are illustrated in Figures \ref{fig:numericalsol_initialperturb00001} to \ref{fig:numericalsol_initialperturb01} and one can clearly see that the ellipse pattern is stable under small perturbations. 
\begin{figure}[ht]
\begin{minipage}{\textwidth}
\centering
    \subfloat[$\chi=0.2$]{\includegraphics[width=0.245\textwidth]{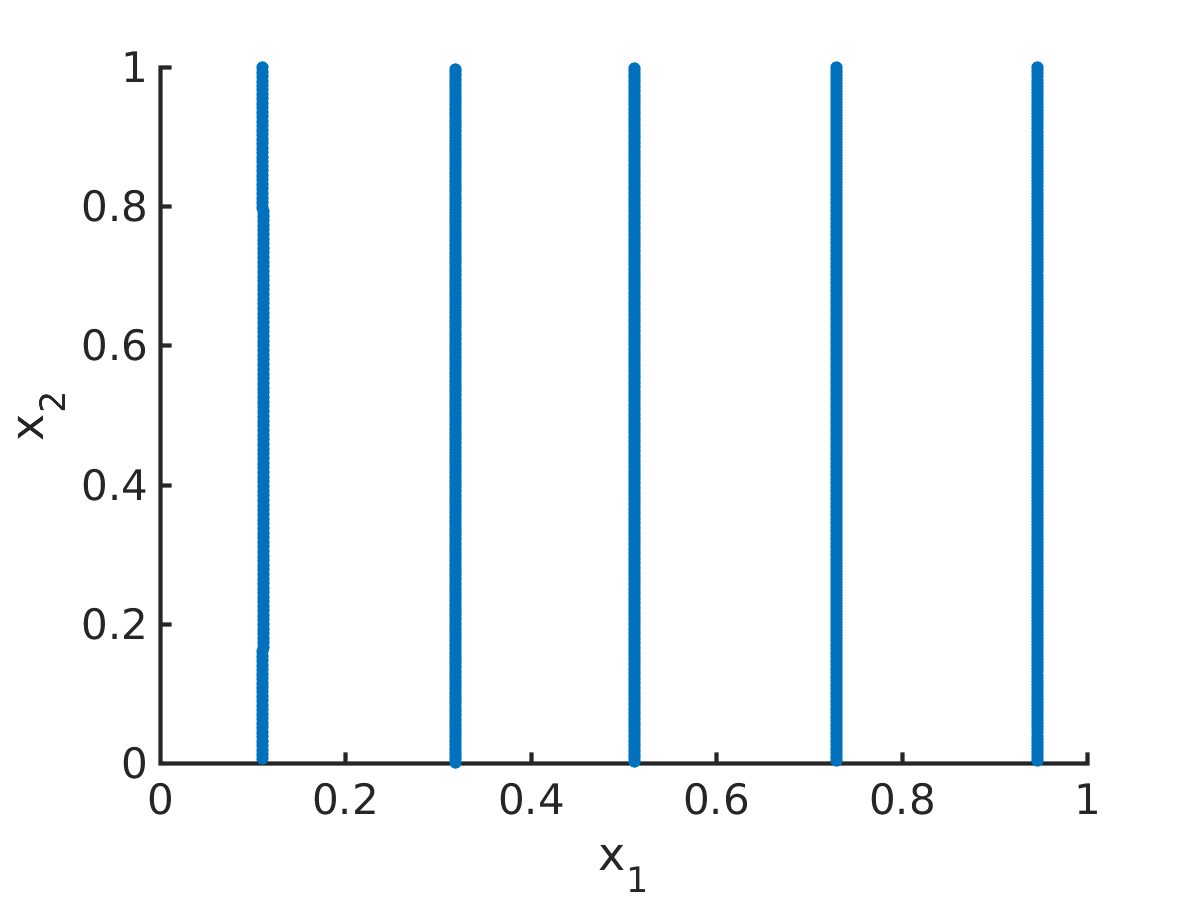}}        
    \subfloat[$\chi=0.7$]{\includegraphics[width=0.245\textwidth]{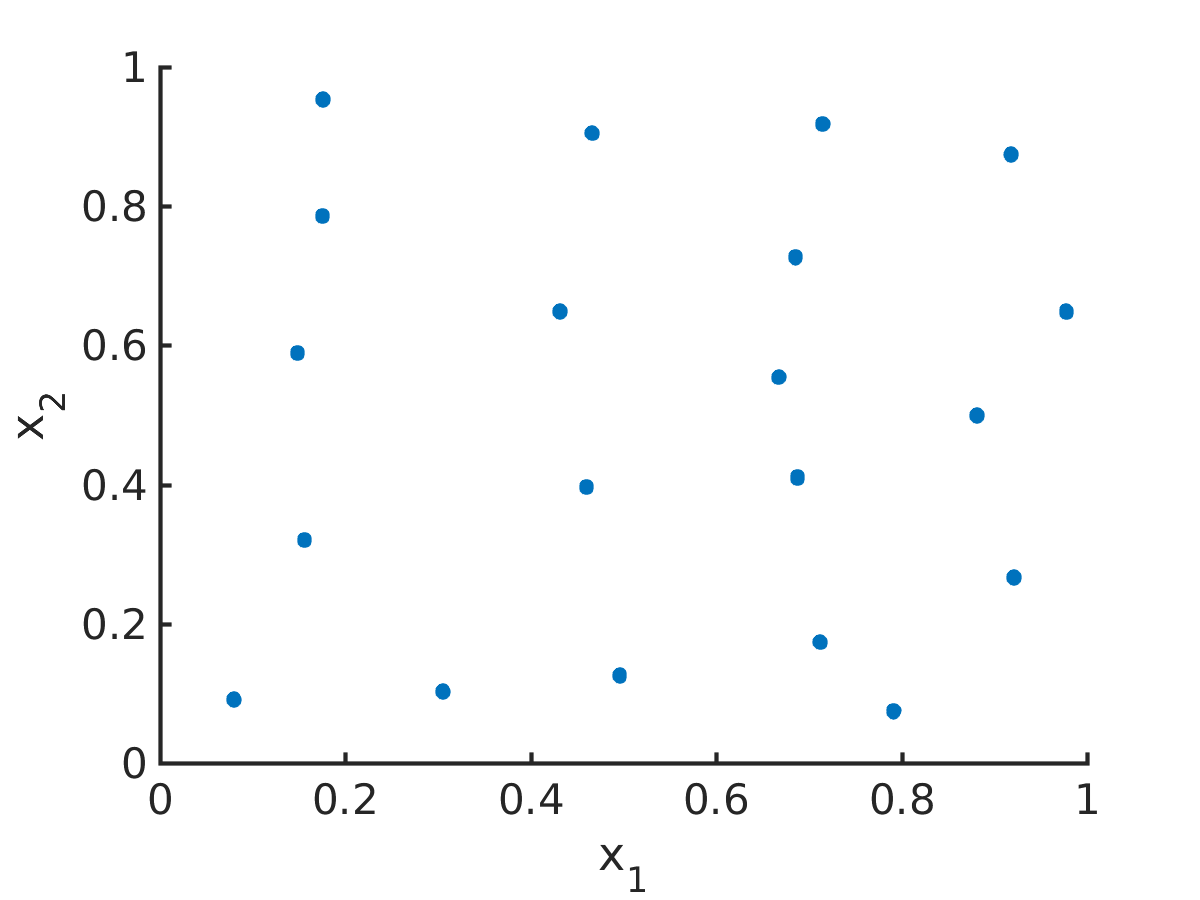}}
    \subfloat[$\chi=0.7$, \-enlarged]{\includegraphics[width=0.245\textwidth]{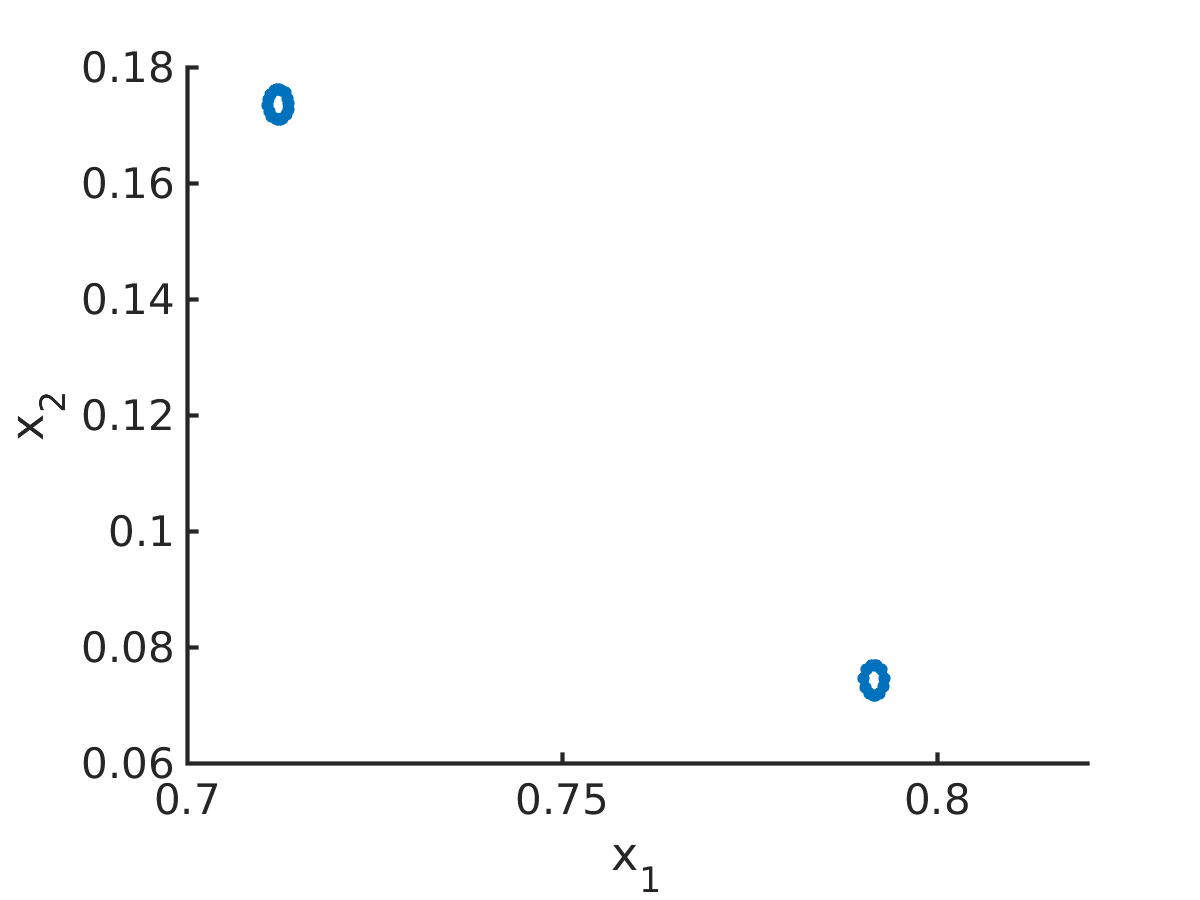}}
    \subfloat[$\chi=0.7$, \-enlarged]{\includegraphics[width=0.245\textwidth]{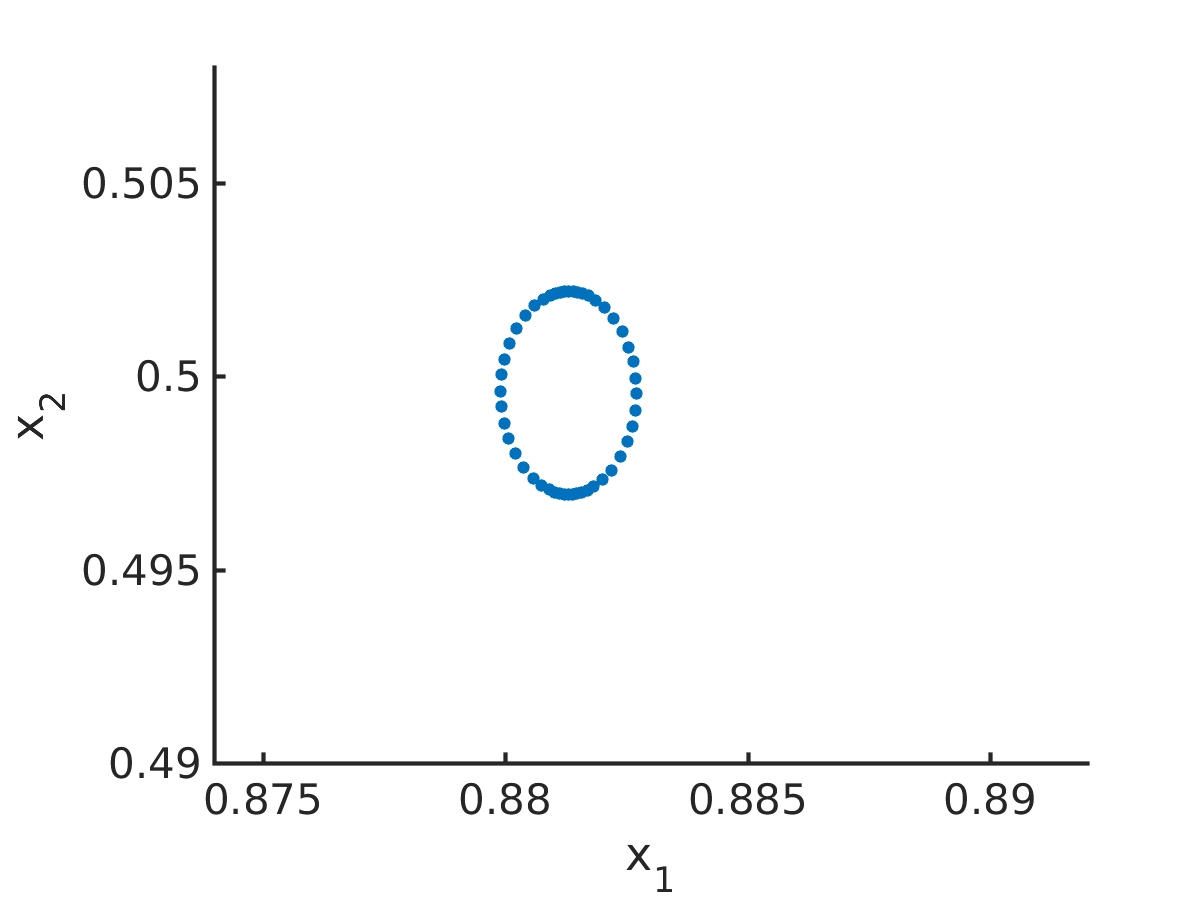}}
    \vspace{3mm}
Randomly uniformly distributed
\end{minipage}

\begin{minipage}{\textwidth}
\centering
    \subfloat[$\chi=0.2$]{\includegraphics[width=0.245\textwidth]{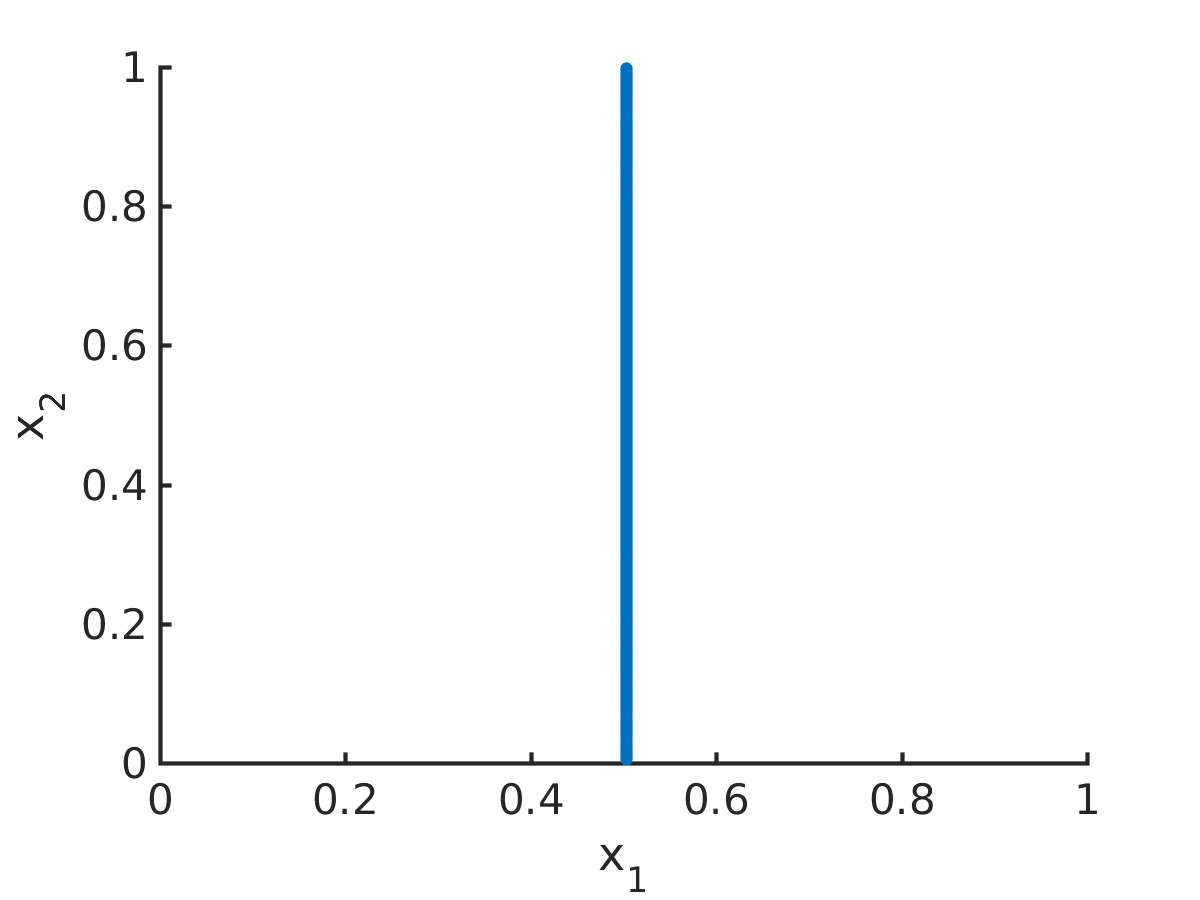}}        
    \subfloat[$\chi=0.7$]{\includegraphics[width=0.245\textwidth]{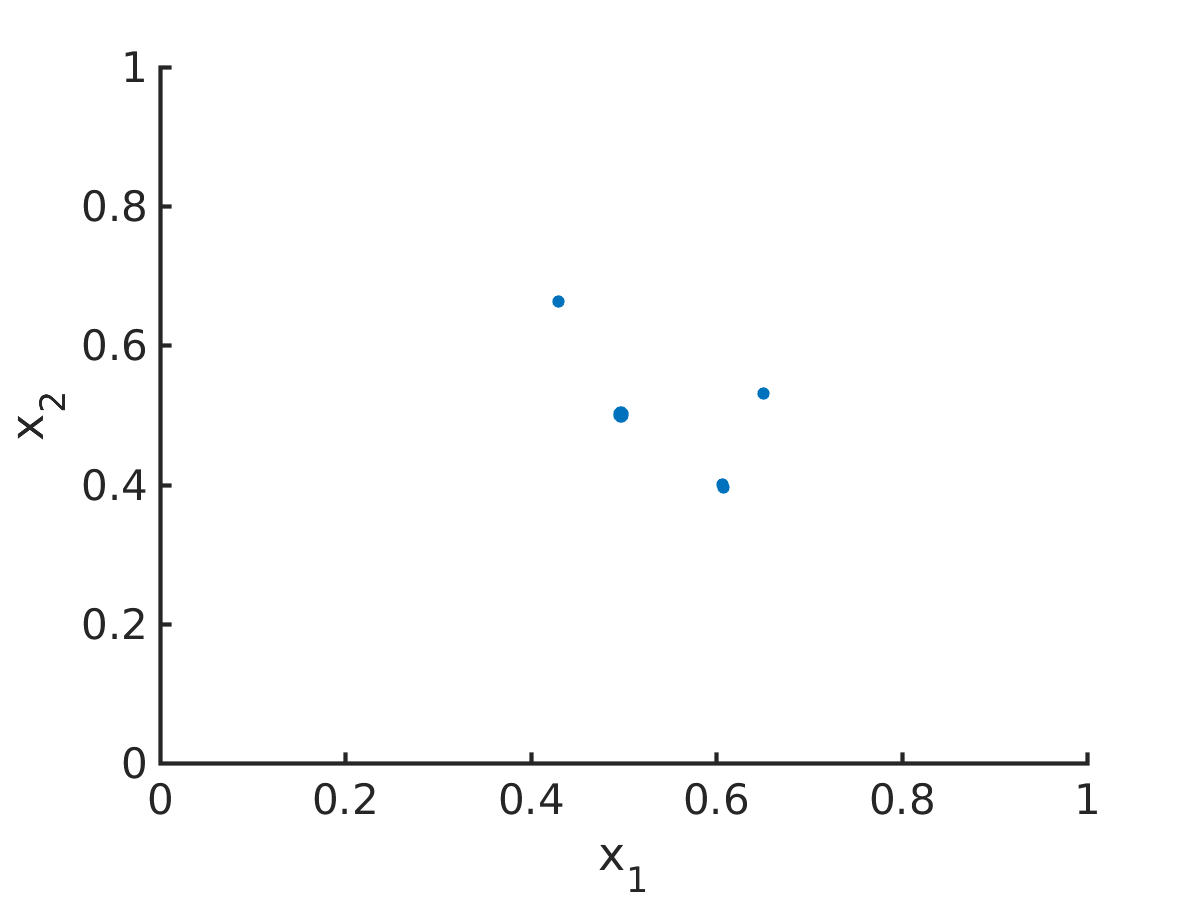}}
    \subfloat[$\chi=0.7$, \-enlarged]{\includegraphics[width=0.245\textwidth]{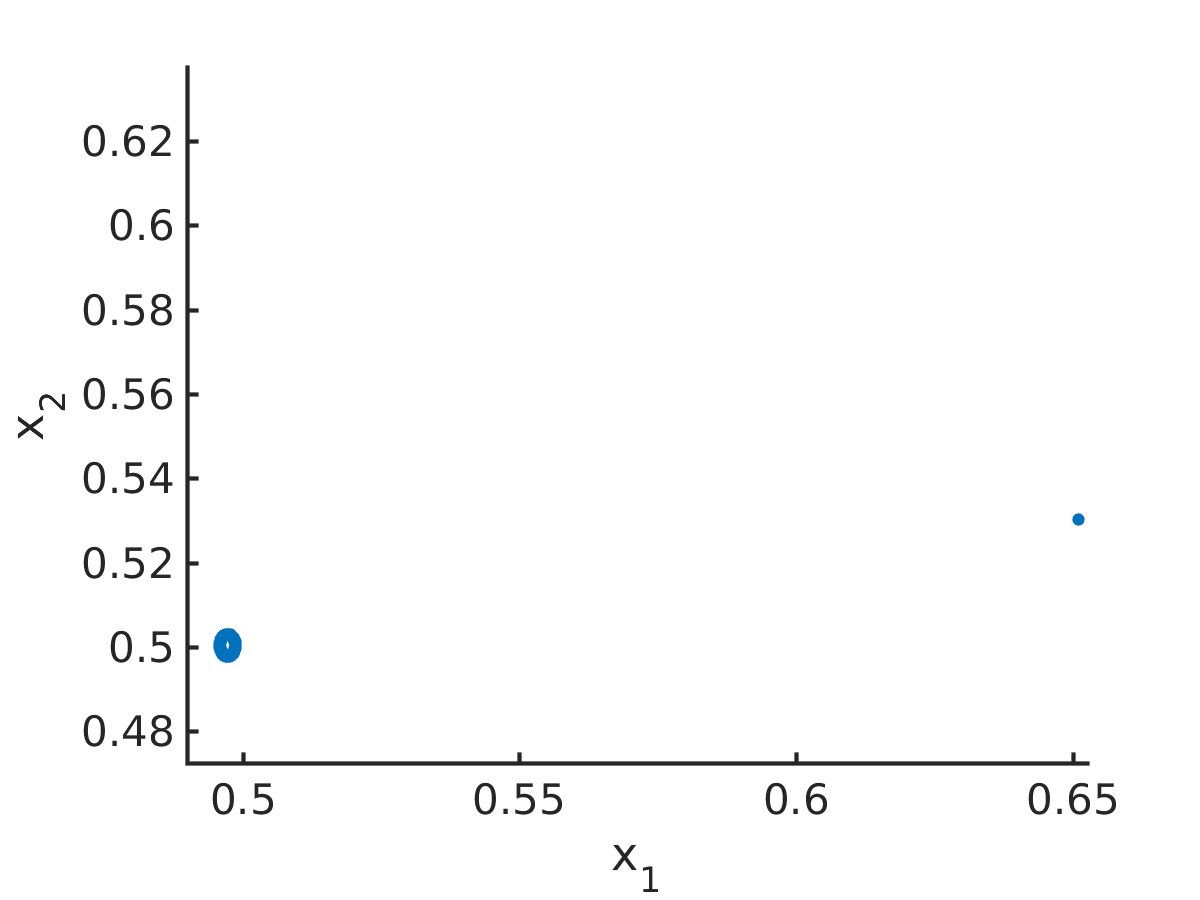}}
    \subfloat[$\chi=0.7$, \-enlarged]{\includegraphics[width=0.245\textwidth]{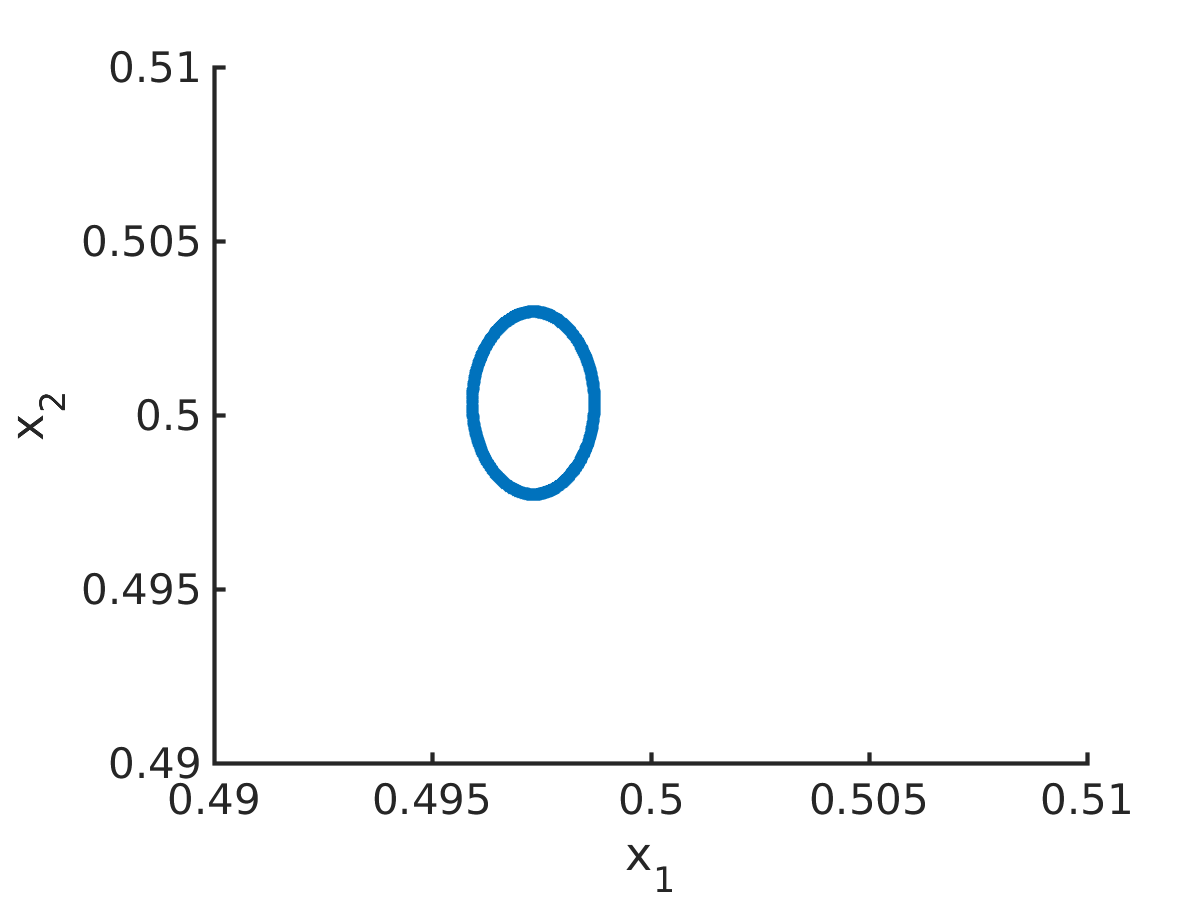}}
 \vspace{3mm}   
Gaussian with  $\sigma=0.05$
\end{minipage}
    \caption{Stationary solution to the K\"{u}cken-Champod model  \eqref{eq:particlemodel} for $N=1200$ and different initial data for $\chi=0.2$ (left) and $\chi=0.7$ (right)}\label{fig:numericalsol_initial}
    \end{figure}
    
\begin{figure}[ht]
    \centering
    \subfloat[$t=0$]{\includegraphics[width=0.24\textwidth]{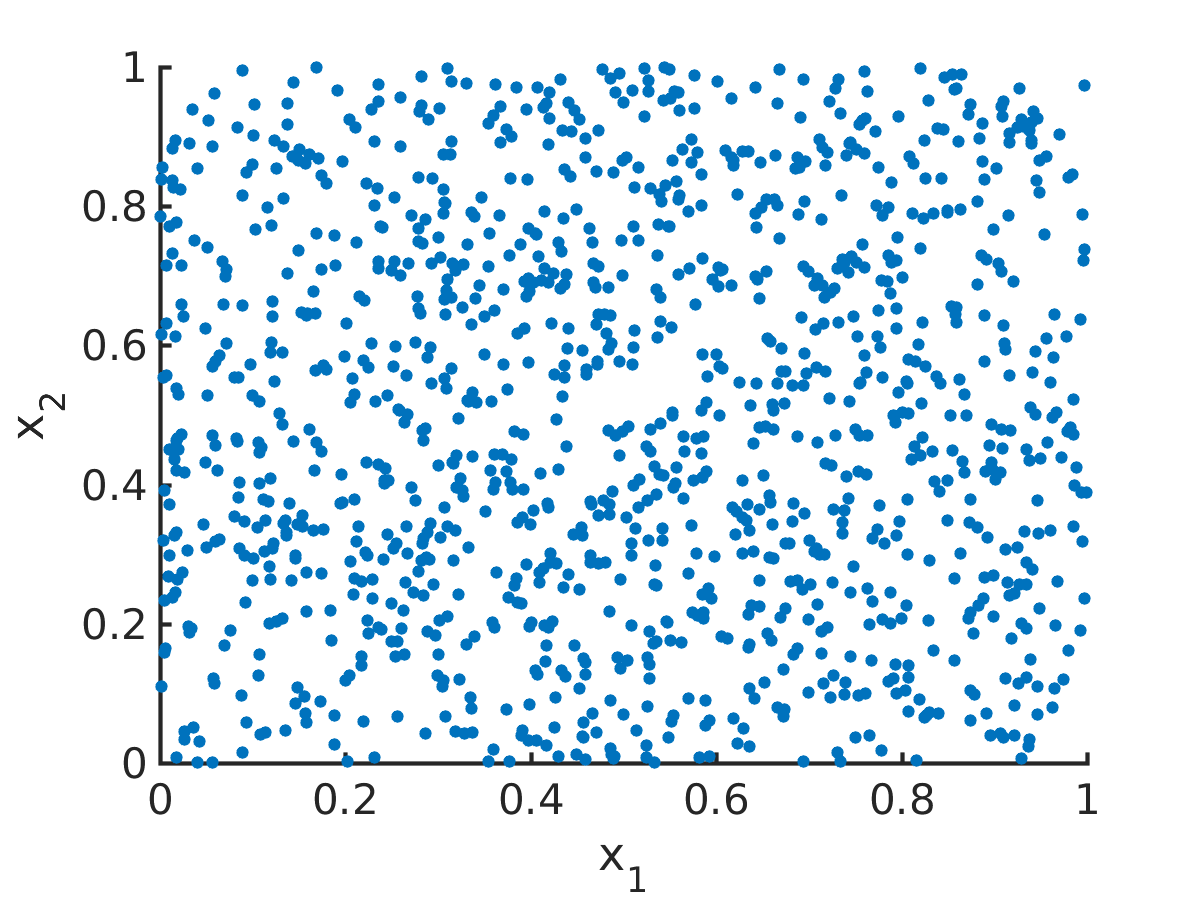}}    
    \subfloat[$t=50000$]{\includegraphics[width=0.24\textwidth]{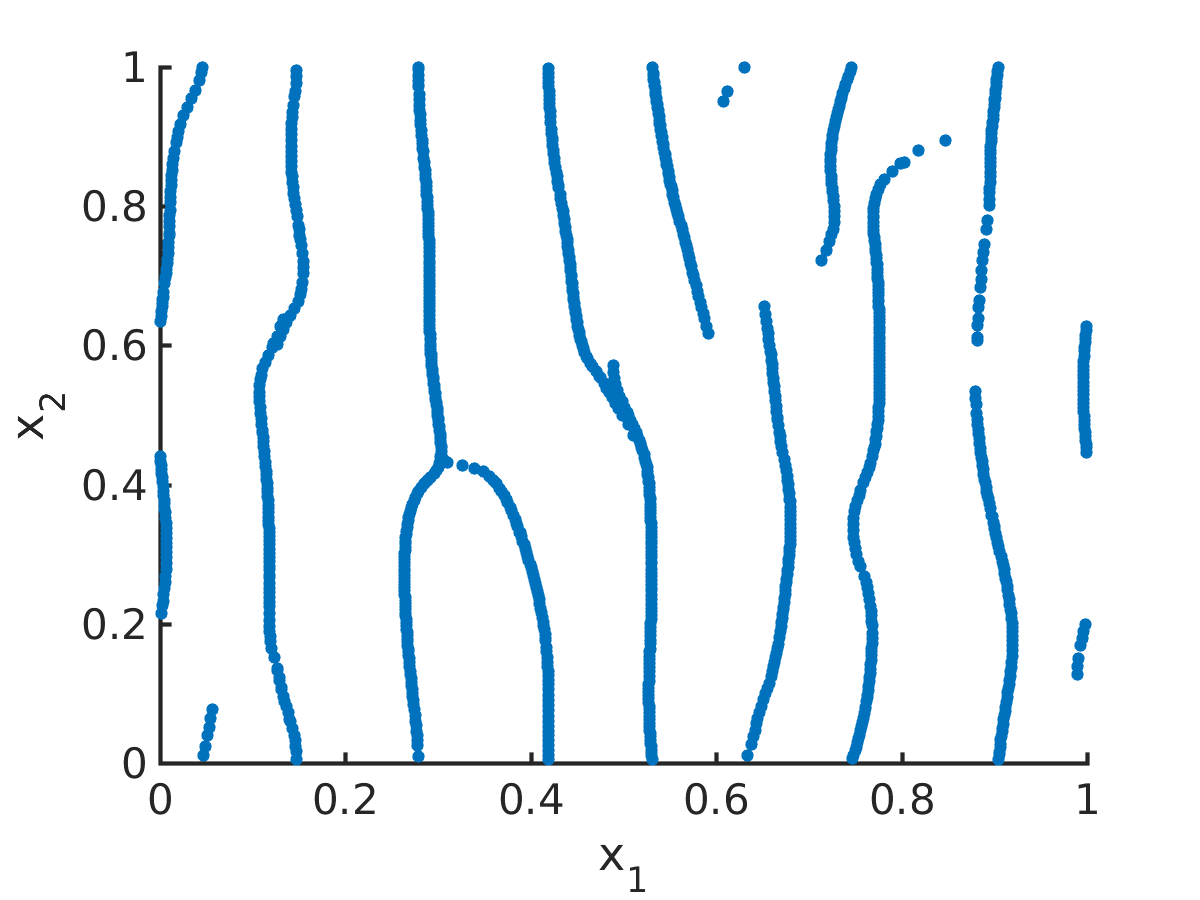}}
\subfloat[$t=110000$]{\includegraphics[width=0.24\textwidth]{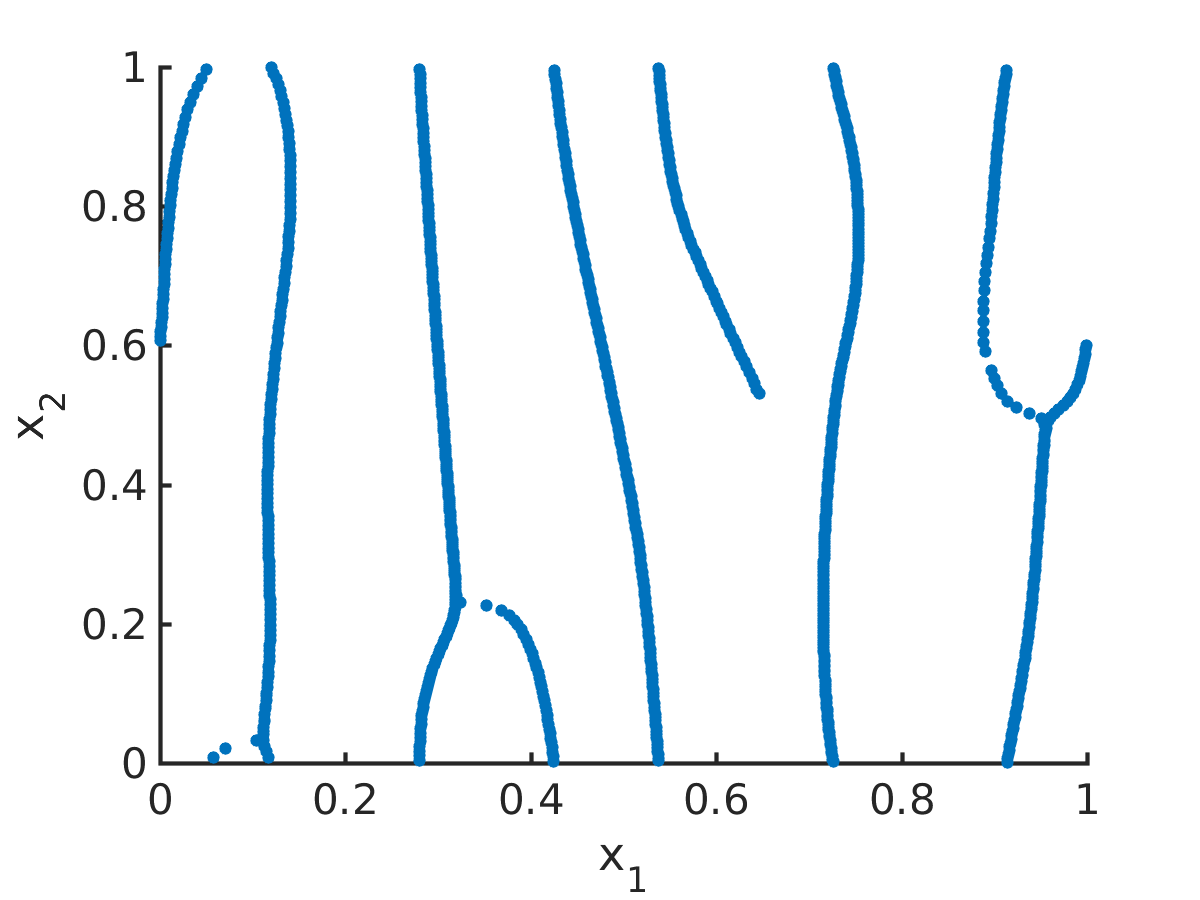}}
\subfloat[$t=710000$]{\includegraphics[width=0.24\textwidth]{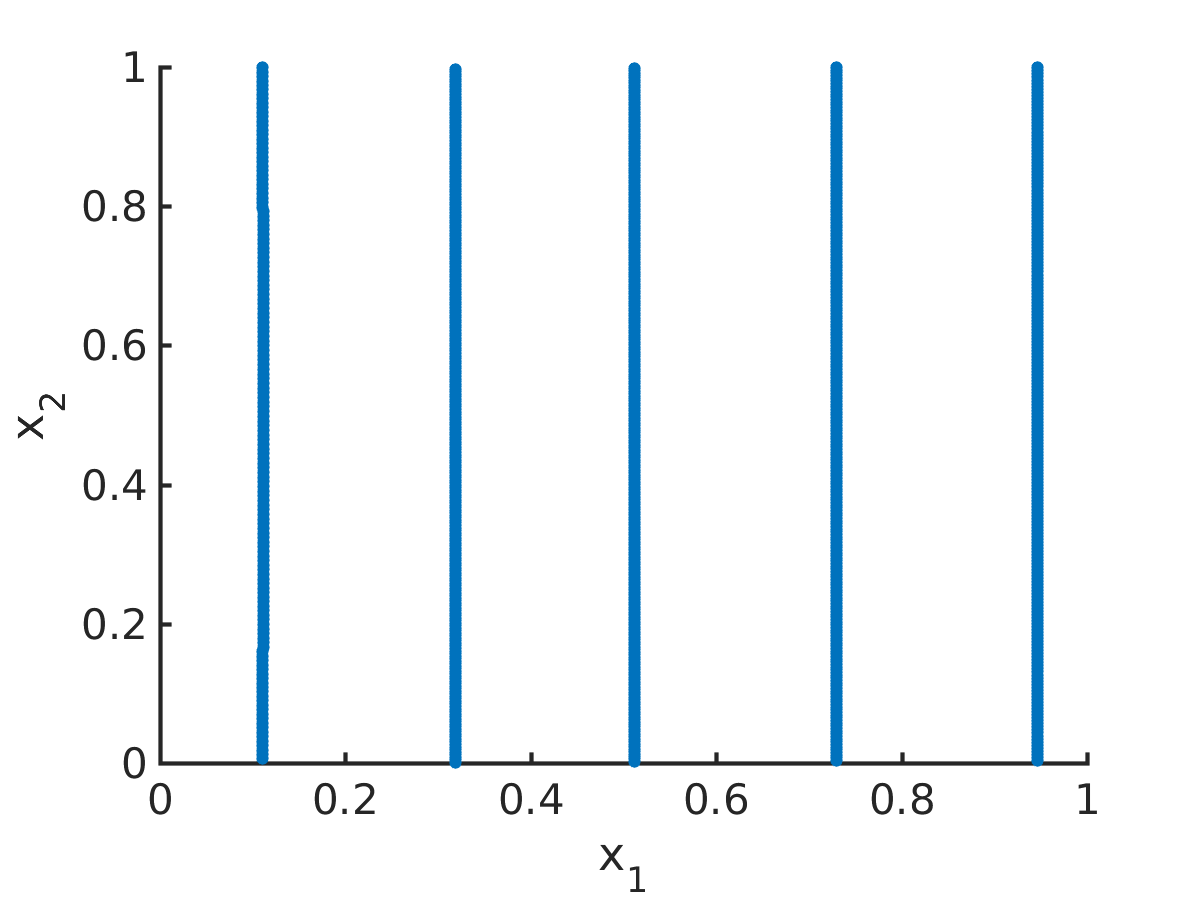}}
    \caption{Numerical solution to the K\"{u}cken-Champod model  \eqref{eq:particlemodel} for $N=1200$ and randomly uniformly distributed initial data for $\chi=0.2$ and different times $t$}\label{fig:numericalsol_uniform02}    
\end{figure}

\begin{figure}[ht]
    \centering
\subfloat[$\delta=0$]{\includegraphics[width=0.198\textwidth]{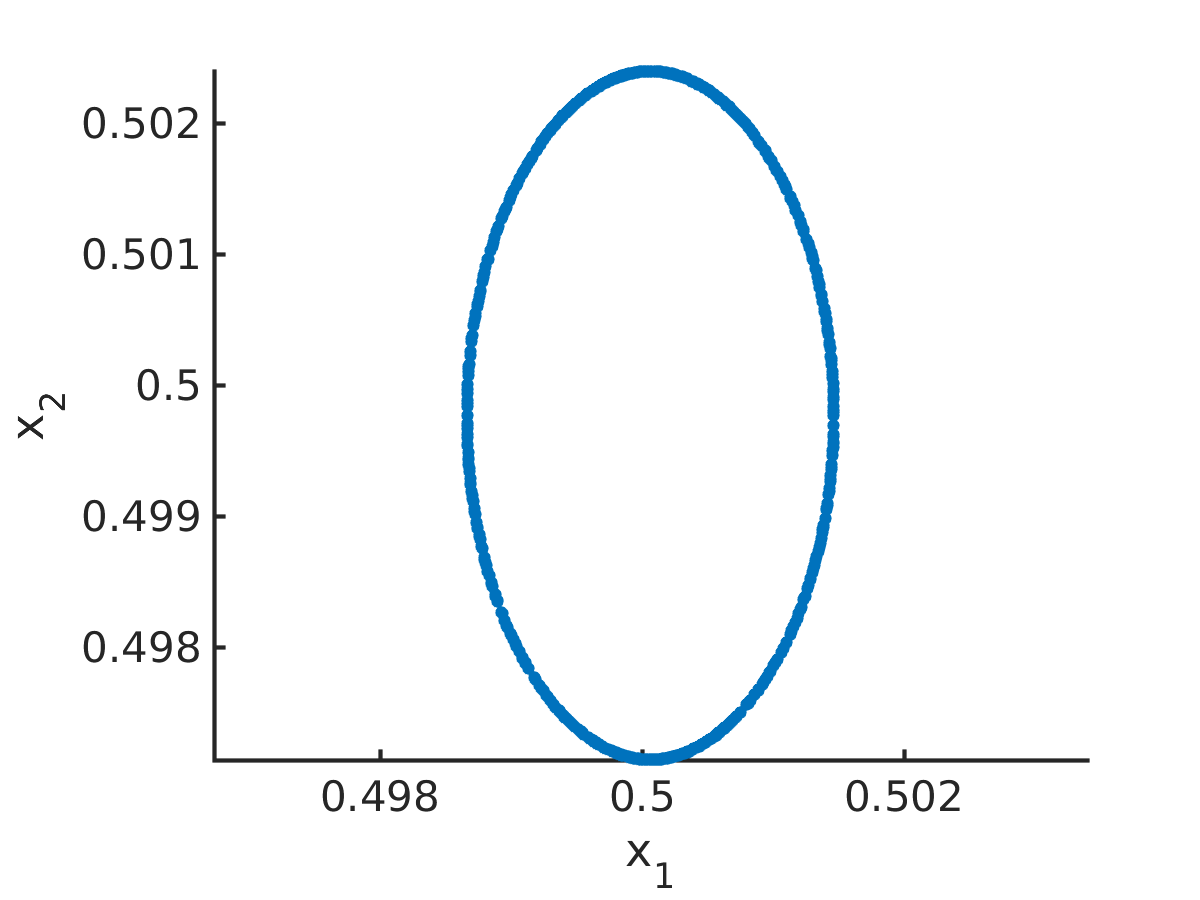}\label{fig:numericalsol_initialnoperturb}}    
\subfloat[$\delta=0.0001$]{\includegraphics[width=0.198\textwidth]{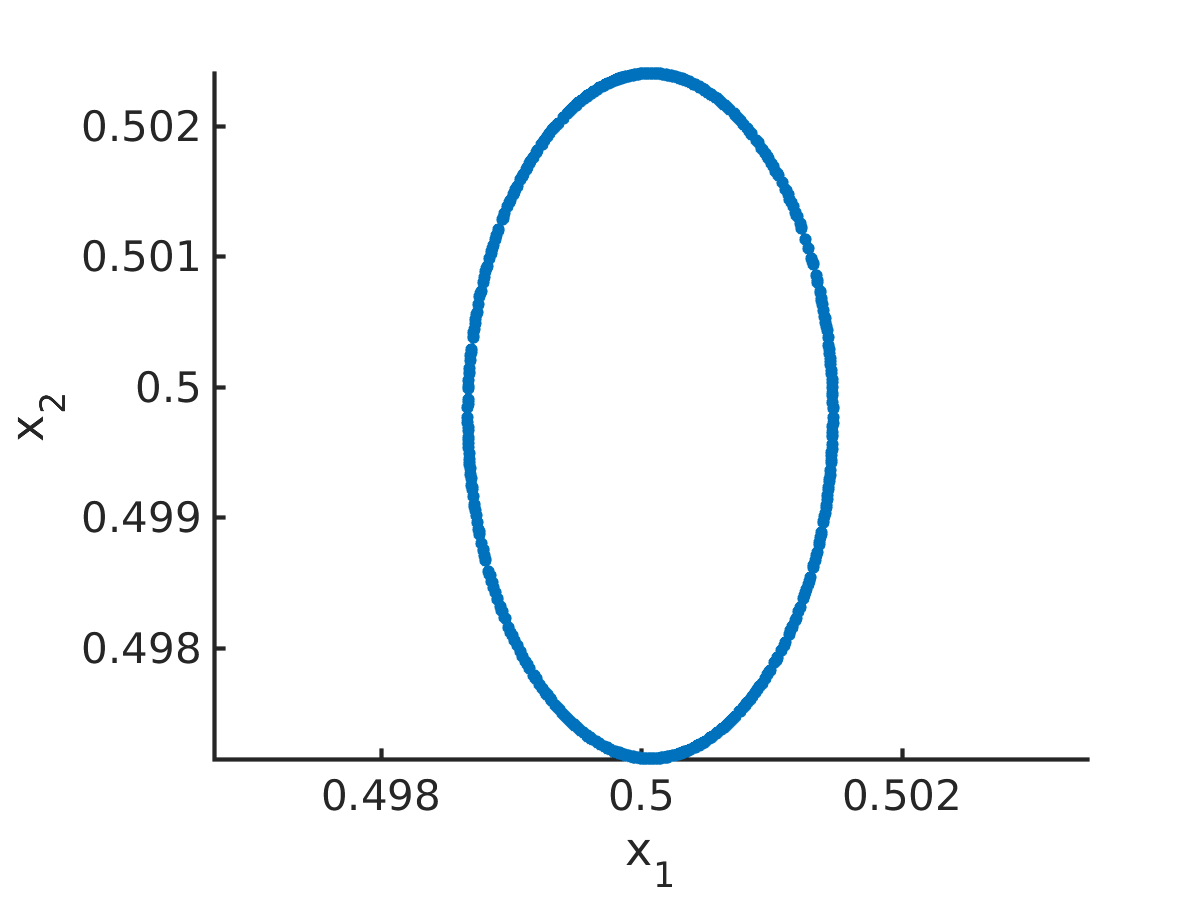}\label{fig:numericalsol_initialperturb00001}}\subfloat[$\delta=0.001$]{\includegraphics[width=0.198\textwidth]{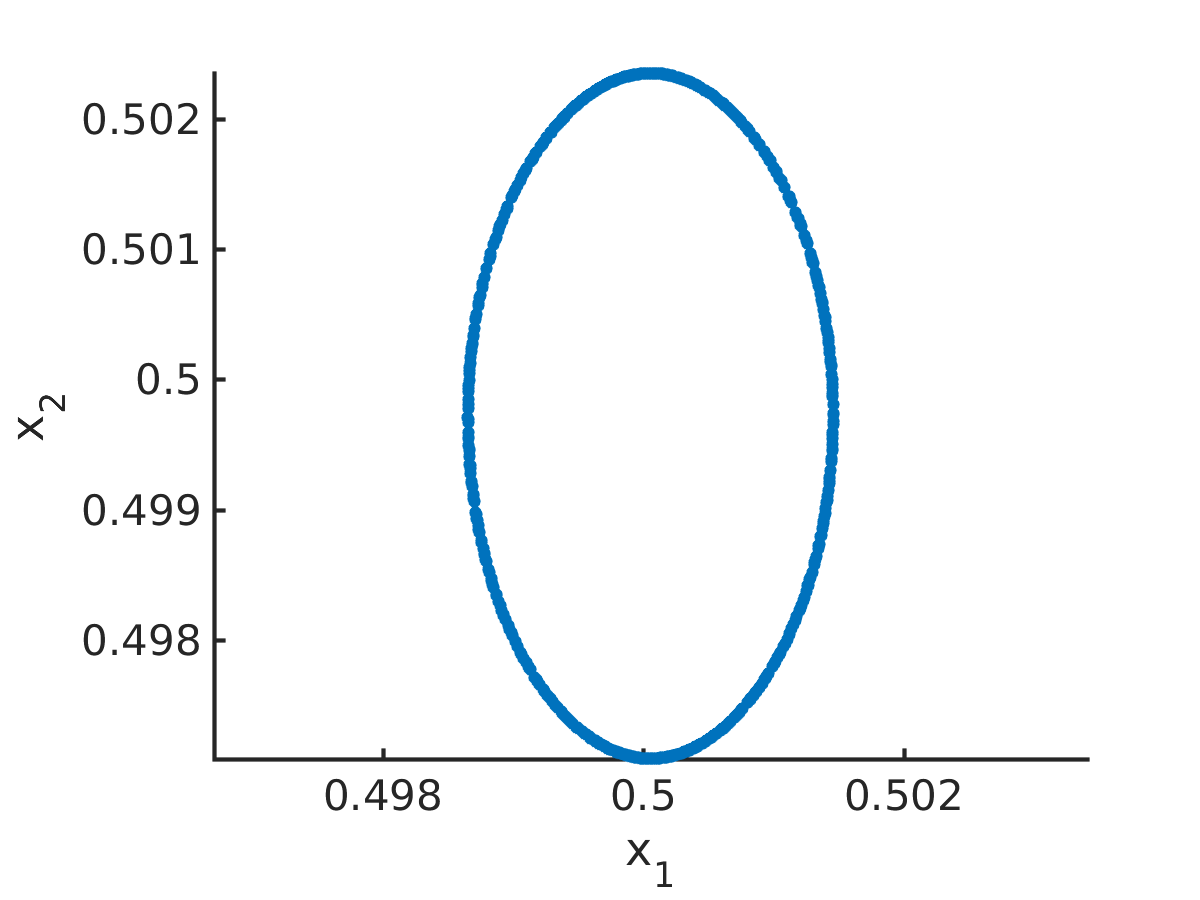}}
\subfloat[$\delta=0.01$]{\includegraphics[width=0.198\textwidth]{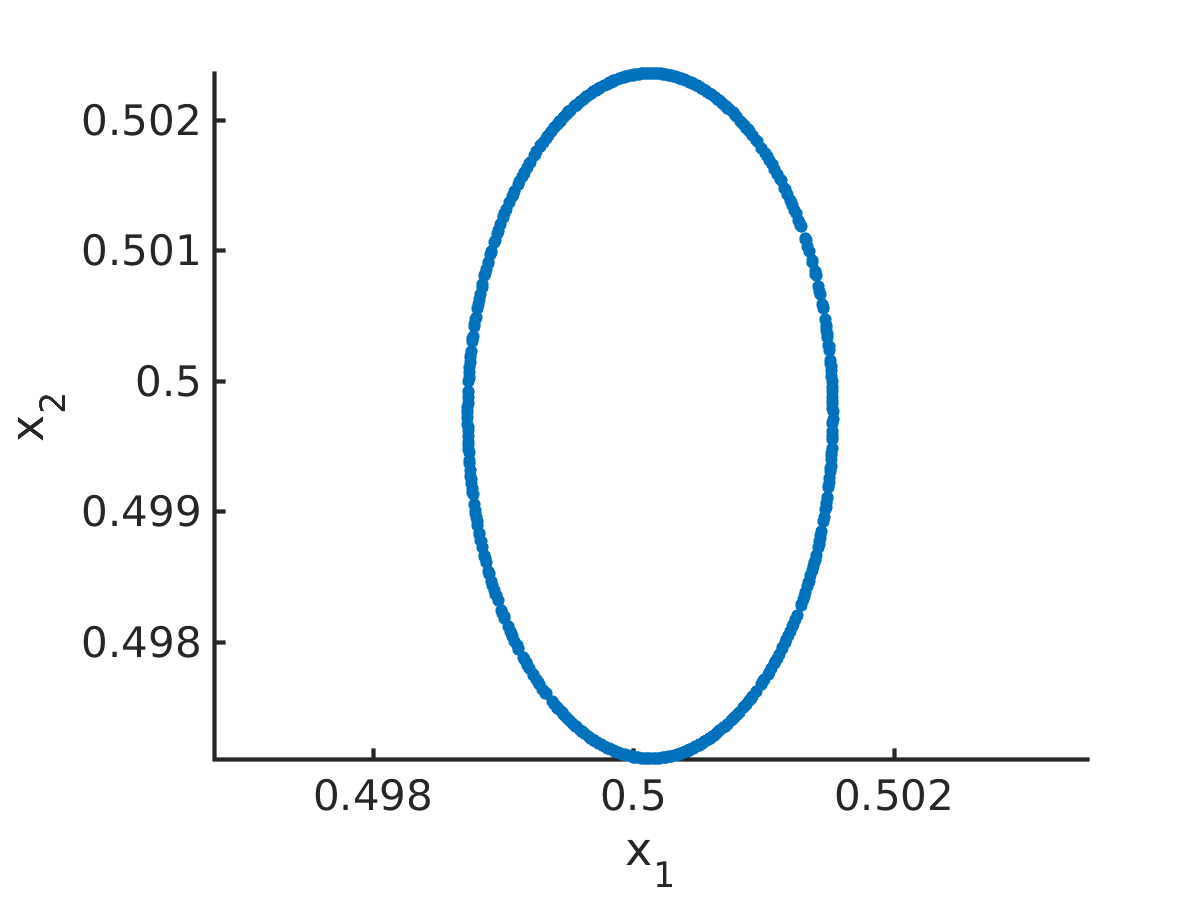}}
\subfloat[$\delta=0.1$]{\includegraphics[width=0.198\textwidth]{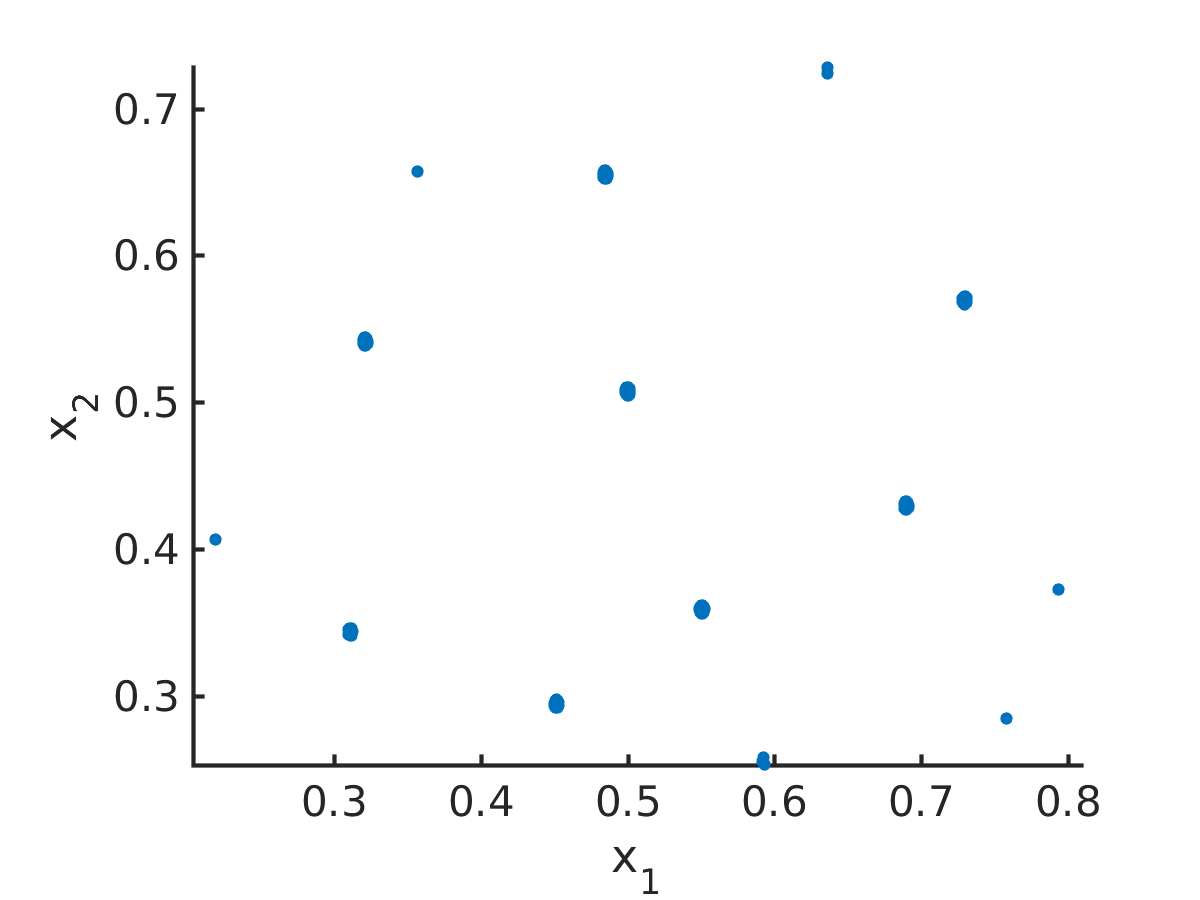}\label{fig:numericalsol_initialperturb01}}
    \caption{Stationary solution to the K\"{u}cken-Champod model \eqref{eq:particlemodel} for $N=600$ and Gaussian initial data ($\mu=0.5$,  $\sigma=0.005$) in each spatial direction in Figure \ref{fig:numericalsol_initialnoperturb} and perturbation of the initial position of each particle $j$ by $\delta Z_j$ where $Z_j$ is drawn from a bivariate standard normal distribution and $\delta\in\{0.0001, 0.001, 0.01, 0.1\}$ in Figures  \ref{fig:numericalsol_initialperturb00001} to \ref{fig:numericalsol_initialperturb01}}\label{fig:numericalsol_initialperturbation}
\end{figure}

\subsubsection{Evolution of the pattern}\label{sec:numericsevolution}
In Figure \ref{fig:evolution_pattern}, the numerical solution of the particle model \eqref{eq:particlemodel} on $\Omega=\mathbb{T}^2$ for $N=1200$ is shown for $\chi=0$, $\chi=0.2$ and $\chi=1.0$ for different times $t$ for Gaussian initial data with mean $\mu=0.5$ and standard deviation $\sigma=0.005$ in each spatial direction. Compared to the initial data one can clearly see  that the solution for $\chi=0$ and $\chi=0.2$, respectively, is stretched  along the vertical axis, i.e. along $s=(0,1)$, as time increases. This is consistent with the observations in Section \ref{sec:interpretationforce} since  the forces along the vertical axis for $\chi=0$ and $\chi=0.2$ are purely repulsive. In contrast, the long-range attraction forces for $\chi=1$ prohibit stretching of the solution and the isotropic forces for $\chi=1$ lead to ring as stationary solution whose radius is approximately $0.0017$. The different sizes of the stationary patterns are also illustrated in Figure \ref{fig:evolution_pattern} where the solutions for $\chi=0$ and $\chi=0.2$ are shown on the unit square, while a smaller axis scale is considered for $\chi=1$ because of the small radius of the ring for $\chi=1$.  Besides,  the convergence to the equilibrium state is very fast for $\chi=1$ compared to $\chi=0$ and $\chi=0.2$.

\begin{figure}[ht]
    \centering
           \subfloat[$\chi=0$]{
        \includegraphics[width=0.33\textwidth]{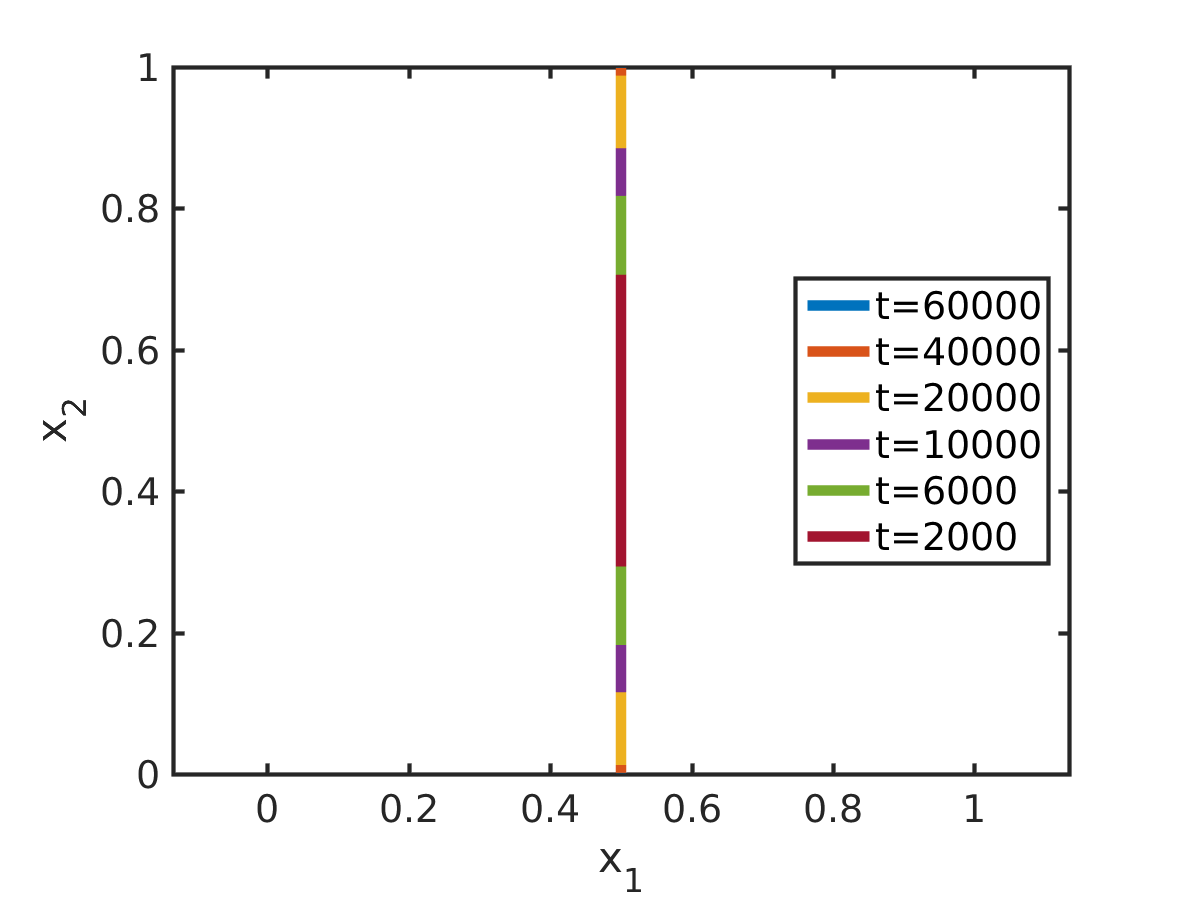}\label{fig:evolution_pattern_chi0}}
    \subfloat[$\chi=0.2$]{\includegraphics[width=0.33\textwidth]{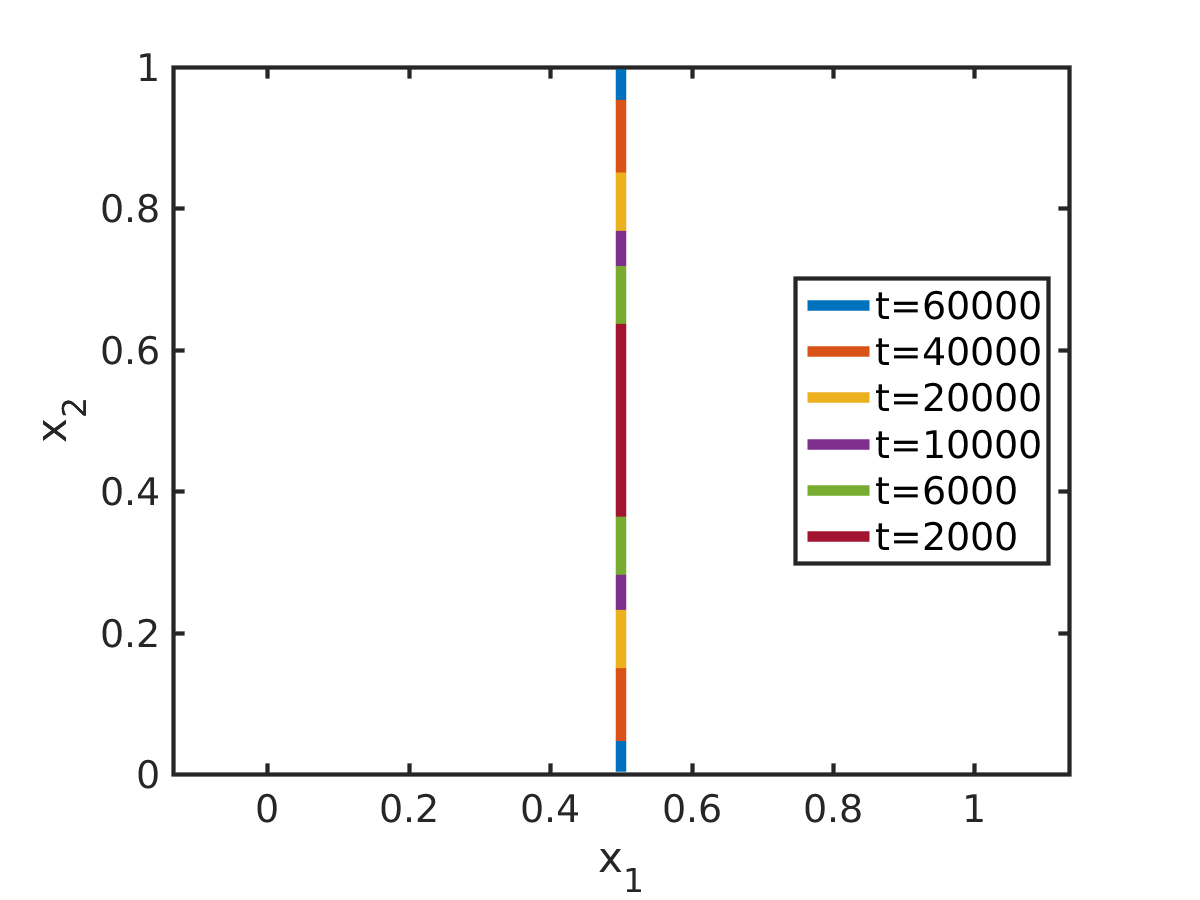}}
   \subfloat[$\chi=1$]{\includegraphics[width=0.33\textwidth]{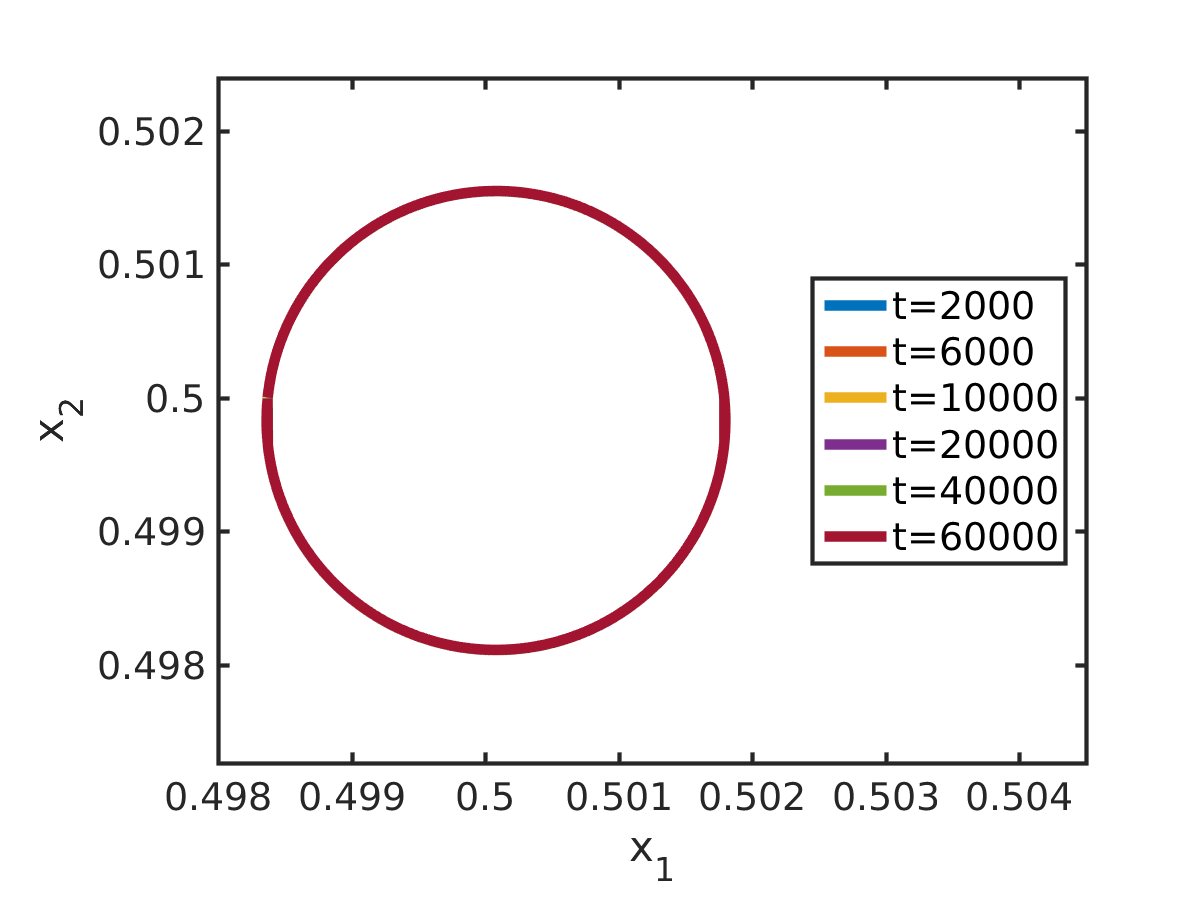}}
 
 \caption{Numerical solution  to the K\"{u}cken-Champod model \eqref{eq:particlemodel} for different times $t$ and different values of $\chi$ for $N=1200$ and Gaussian initial data ($\mu=0.5$, $\sigma=0.005$) in each spatial direction  \label{fig:evolution_pattern}}
\end{figure}

\subsubsection{Dependence on parameter $\chi$}
In this section we investigate the dependence of the equilibria to \eqref{eq:particlemodel} on the parameter $\chi$ which strongly influences the pattern formation. Given $N=600$ particles which are initially equiangular distributed on a circle with center $(0.5,0.5)$ and radius $0.005$  the stationary solution to \eqref{eq:particlemodel}  is displayed for different values of $\chi$ in Figures \ref{fig:numericalsolDependenceChiNozoom} and \ref{fig:numericalsolDependenceChi}. Note that the same simulation results are shown in Figures \ref{fig:numericalsolDependenceChiNozoom} and \ref{fig:numericalsolDependenceChi} for different axis scales.
In Figure \ref{fig:numericalsolDependenceChiNozoom} one can see that the size of the pattern is significantly larger for small values of $\chi$ due to stretching along the vertical axis (cf. Section \ref{sec:analysis}). For small values of $\chi$ the stationary solution is  a 1D stripe pattern of equally distributed particles along the entire vertical axis, while for larger values of $\chi$ the stationary solution can be a shorter vertical line or accumulations in the shape of  lines and ellipses.
\begin{figure}[htbp]
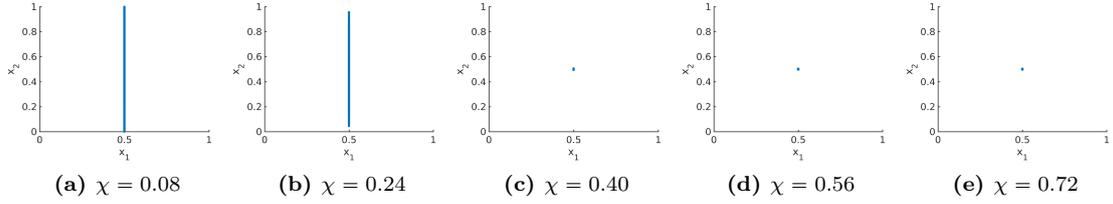

\foreach \x in {8}{%
    \subfloat[$\chi=0.0\x$]{\includegraphics[width=0.198\textwidth]
{Particle_Euler_periodic_ParameterForNiceForces_simultaneous_chi_nozoom00\x_N600_circleinit_dt02.png}}
}%
           \foreach \x in {24,40, ..., 85}{%
    \subfloat[$\chi=0.\x$]{\includegraphics[width=0.198\textwidth]
{Particle_Euler_periodic_ParameterForNiceForces_simultaneous_chi_nozoom0\x_N600_circleinit_dt02.png}}
}%
\caption{Comparison of the size of the stationary solution to the K\"{u}cken-Champod model \eqref{eq:particlemodel} for different values of $\chi$ where $N=600$ and the initial data is equiangular distributed on a circle with center $(0.5,0.5)$ and radius $0.005$ }\label{fig:numericalsolDependenceChiNozoom}
\end{figure}
The stationary patterns for  different values of $\chi$ are enlarged in Figure \ref{fig:numericalsolDependenceChi} by considering different axis scales. As $\chi$ increases the stationary pattern evolves from a straight line into a standing ellipse and finally into a ring for $\chi=1.0$. Since the same particle numbers and the same initial data, as well as the same parameters except for the parameter $\chi$ are considered in these simulations, the different stationary patterns strongly depend on the choice of $\chi$.  Note that the length of the minor axis of the ellipse increases as $\chi$ increases, while the length of the major axis of the ellipse gets shorter. Further note that we have a continuous transition of the stationary patterns as $\chi$ increases due to the smoothness of the forces and the continuous dependence of the forces on parameter $\chi$ in the K\"ucken-Champod model \eqref{eq:particlemodel}.

\begin{figure}[htbp]
\centering
           \foreach \x in {12,16, ..., 90}{%
    \subfloat[$\chi=0.\x$]{\includegraphics[width=0.2\textwidth]
{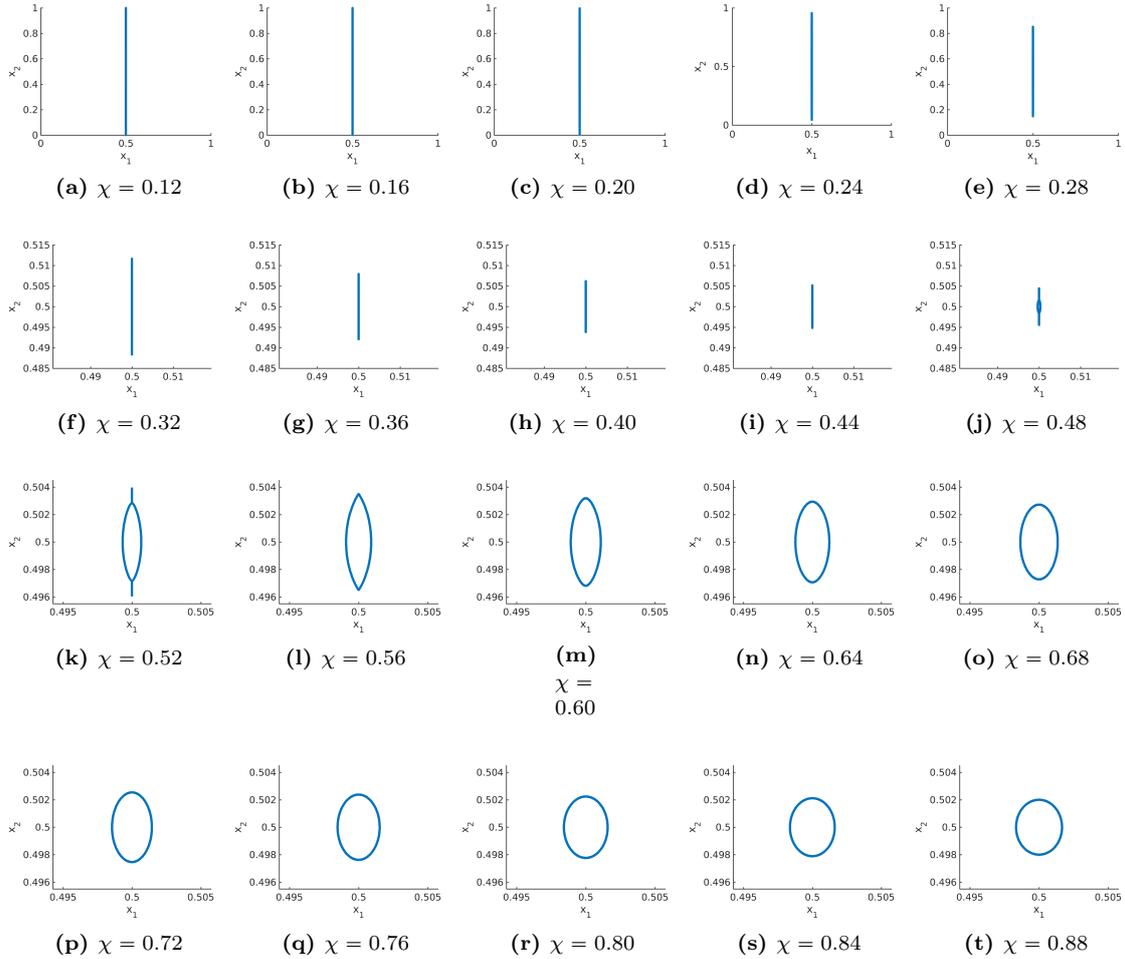}}\hfill
}%
\caption{Stationary solution to the K\"{u}cken-Champod model \eqref{eq:particlemodel} for different values of $\chi$ where $N=600$ and the initial data is equiangular distributed on a circle with center $(0.5,0.5)$ and radius $0.005$}\label{fig:numericalsolDependenceChi}
\end{figure}

\subsubsection{Dependence on parameter $e_R$}
In Figure \ref{fig:numericalsolDependenceEr} the stationary solution to \eqref{eq:particlemodel} for $N=1200$ and $\chi=0.2$ is shown for different values of $e_R$ where a ring of radius $0.005$ with center $(0.5,0.5)$ is chosen as initial data. One can clearly see that the size of accumulations increases for $e_R$ increasing due to strong long-range repulsion forces for smaller values of $e_R$. Besides, the stationary solution is spread over the entire domain for smaller values of $e_R$. The spreading of the solution  along the entire horizontal axis  can be explained by the fact that for smaller values of $e_R$ the total force along $l$, i.e. along the horizontal axis, is no longer short-range repulsive and long-range attractive, but short-range  repulsive,  medium-range  attractive and  long-range  repulsive and the long-range repulsion is the stronger the smaller the value of $e_R$. 
\begin{figure}[htbp]
\foreach \x in {40,50,...,80}{%
    \subfloat[$e_R=\x$]{\includegraphics[width=0.197\textwidth]
{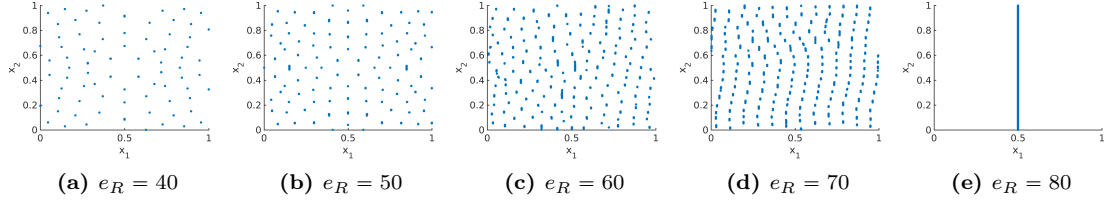}}
}%
\caption{Stationary solution to the K\"{u}cken-Champod model \eqref{eq:particlemodel} for $\chi=0.2$ and different values for $e_R$}\label{fig:numericalsolDependenceEr}
\end{figure} 

\subsubsection{Dependence on the size of the attraction force}
In this section, we assume that the total force is given by $F(d,T)=\delta F_A(d,T)+F_R(d)$ for $\delta\in [0,1]$ for the  spatially homogeneous tensor field $T=\chi s \otimes s+l\otimes l$ with $l=(1,0)$ and $s=(0,1)$  instead of \eqref{eq:totalforce}. We consider $N=600$ particles which are initially equiangular distributed on a circle with center $(0.5,0.5)$ and radius $0.005$ and we investigate the influence of the size of the attraction force $F_A$ on stationary patterns by varying its coefficients. While the force is repulsive for small values of $\delta$, resulting in a stationary solution spread over the entire domain, stripe patterns and ring patterns for $\chi=0.2$ and $\chi=1$, respectively, arise as stationary patterns as $\delta$ increases as shown in Figures \ref{fig:numericalsol_dependence_attraction}. Note that the radius of the stationary ring pattern decreases as $\delta$ increases due to an increasing attraction force.

\begin{figure}[htbp]
\centering
\foreach \x in {1,3,...,9}{%
    \subfloat[$\delta=0.\x$]{\includegraphics[width=0.2\textwidth]
{Particle_Euler_periodic_ParameterForNiceForces_simultaneous_chi02_Repulsion0\x_Attraction_N600_circleinit_dt02.png}}
           }%
           
\subfloat[$\delta=0.1$]{\includegraphics[width=0.2\textwidth]
{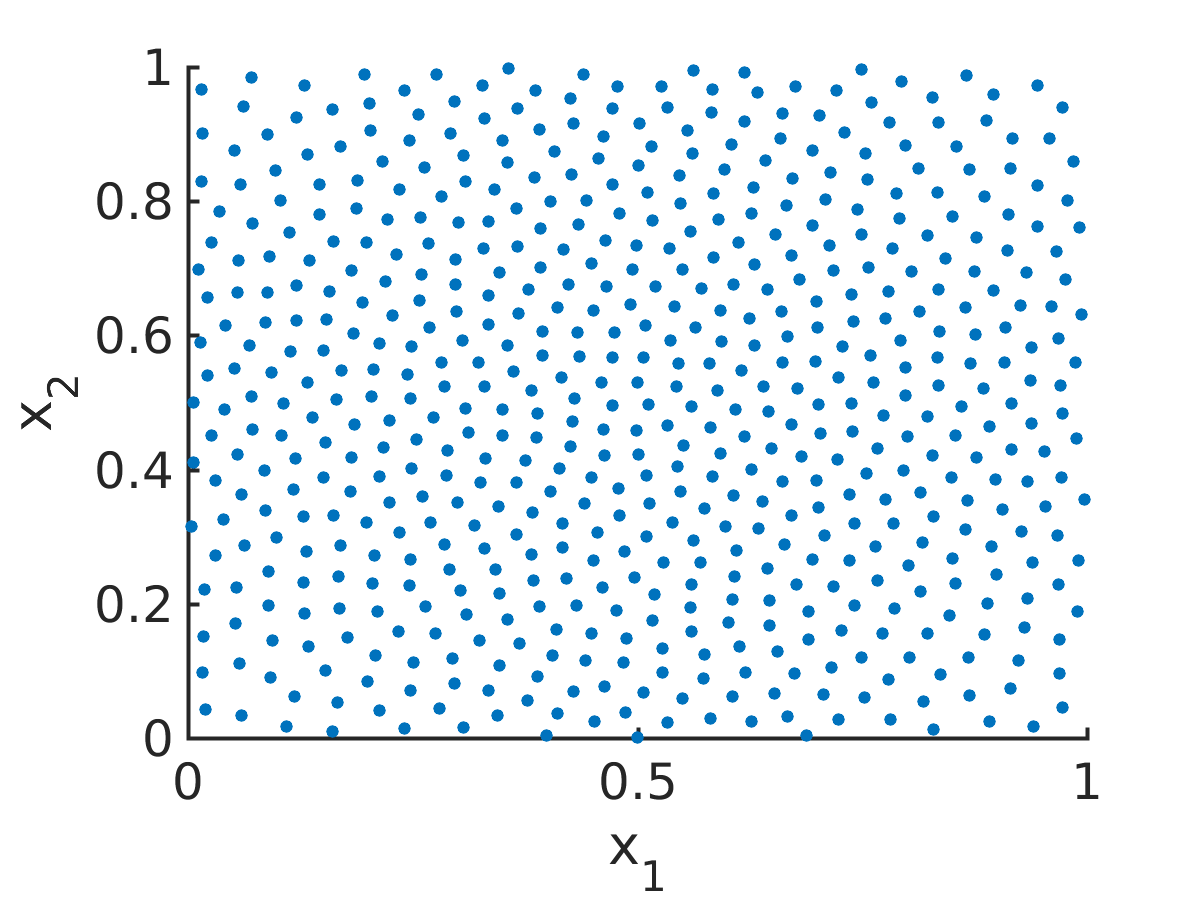}}
\foreach \x in {3,5,...,9}{%
    \subfloat[$\delta=0.\x$]{\includegraphics[width=0.2\textwidth]
{Particle_Euler_periodic_ParameterForNiceForces_simultaneous_chi1_Repulsion0\x_Attraction_N600_circleinit_dt02.png}}
           }%
\caption{Stationary solution to the K\"{u}cken-Champod model \eqref{eq:particlemodel} for  force $F(d,T)= \delta F_A(d,T)+ F_R(d)$  for different values of $\delta$ (i.e. different sizes of the attraction force $F_A$) where $\chi=0.2$ and $\chi=1$ (for different axis scalings) in the first and second row, respectively,  where $N=600$ and the initial data is equiangular distributed on a circle with center $(0.5,0.5)$ and radius $0.005$}\label{fig:numericalsol_dependence_attraction}
\end{figure}

\subsubsection{Dependence on the size of the repulsion force}

In this section, we consider a force of the form $F(d,T)= F_A(d,T)+\delta F_R(d)$ for $\delta\in [0,1]$ for the  spatially homogeneous tensor field $T=\chi s \otimes s+l\otimes l$ with $l=(1,0)$ and $s=(0,1)$ instead of \eqref{eq:totalforce} and we consider $N=600$ particles which are initially equiangular distributed on a circle with center $(0.5,0.5)$ and radius $0.005$. The stationary solution to \eqref{eq:particlemodel} for $\chi=0.2$ stretches along the vertical axis as $\delta$ increases due to an additional repulsive force as illustrated in Figure \ref{fig:numericalsol_dependence_repulsion_chi02}. For $\chi=1$, the radius of the ring pattern increases as $\delta$, see Figure~\ref{fig:numericalsol_dependence_repulsion_chi1}.

\begin{figure}[htbp]
\foreach \x in {1,3,...,9}{%
    \subfloat[$\delta=0.\x$]{\includegraphics[width=0.2\textwidth]
{Particle_Euler_periodic_ParameterForNiceForces_simultaneous_chi02_Attraction0\x_Repulsion_N600_circleinit_dt02_nozoom.png}}
           }%
\caption{Stationary solution to the K\"{u}cken-Champod model \eqref{eq:particlemodel} for $\chi=0.2$ and force $F(d,T)= F_A(d,T)+\delta F_R(d)$  for different values of $\delta$ (i.e. different sizes of the repulsion force $F_R$)  where $N=600$ and the initial data is equiangular distributed on a circle with center $(0.5,0.5)$ and radius $0.005$}\label{fig:numericalsol_dependence_repulsion_chi02}
\end{figure}

\begin{figure}[htbp]
\centering
\includegraphics[width=0.5\textwidth]{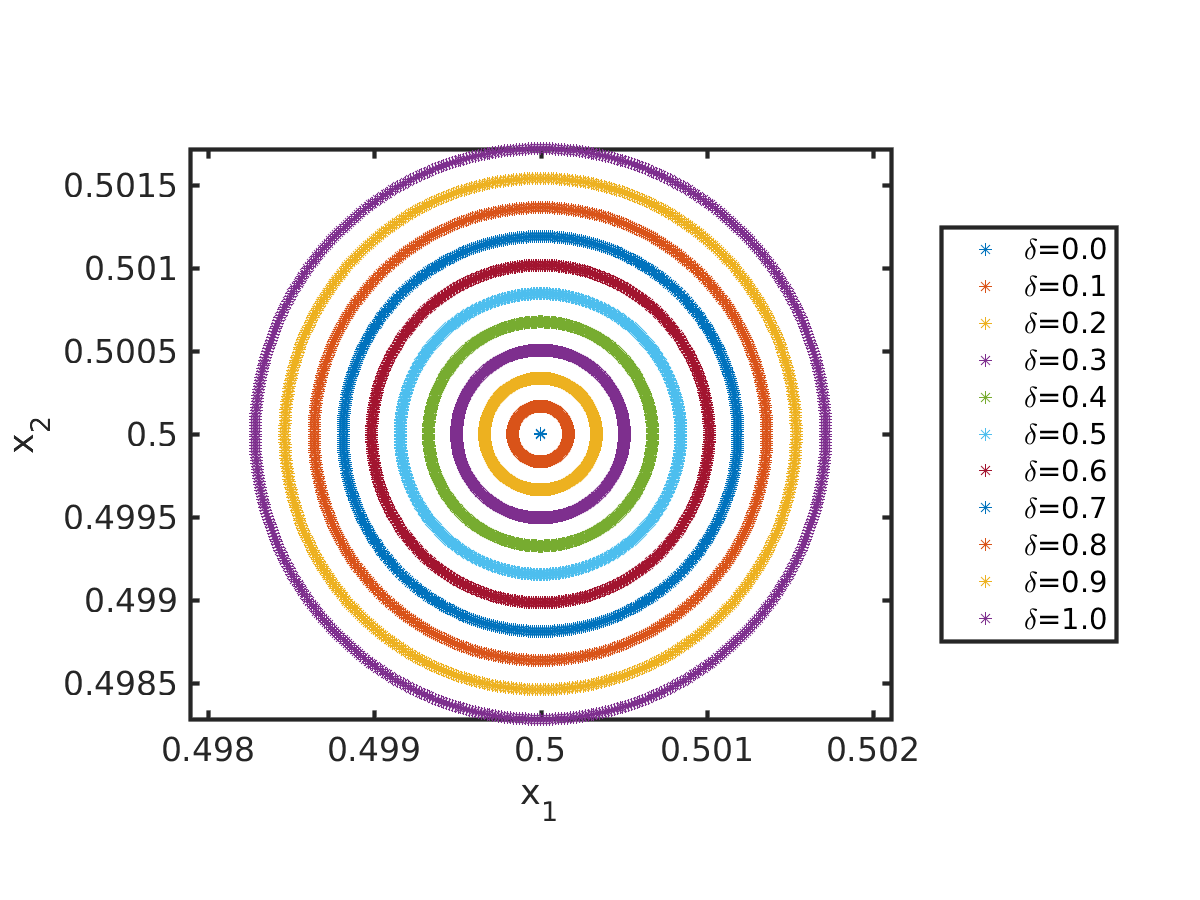}
\caption{Stationary solution to the K\"{u}cken-Champod model \eqref{eq:particlemodel} for $\chi=1$ and force $F(d,T)= F_A(d,T)+\delta F_R(d)$  for different values of $\delta$ (i.e. different sizes of the repulsion force $F_R$)  where $N=600$ and the initial data is equiangular distributed on a circle with center $(0.5,0.5)$ and radius $0.005$}\label{fig:numericalsol_dependence_repulsion_chi1}
\end{figure}

\subsubsection{Dependence on the tensor field}

In Figures \ref{fig:numericalsol_nonhomtensor} and \ref{fig:numericalsol_nonhomtensor_longtime} the numerical solution to the K\"{u}cken-Champod model \eqref{eq:particlemodel} for $N=600$, $\chi=0.2$ and  randomly uniformly distributed data  is shown for different non-homogeneous tensor fields $T=T(x)$ and different times $t$. Since $s=s(x)$ and $l=l(x)$ are assumed to be orthonormal vectors, the vector field $s=s(x)$ and the parameter $\chi$ determine the tensor field $T=T(x)$. One can clearly see in Figure \ref{fig:numericalsol_nonhomtensor} that the particles are aligned along the lines of smallest stress $s=s(x)$. However, these patterns are no equilibria. The evolution of the numerical solution for different tensor fields is illustrated in Figure \ref{fig:numericalsol_nonhomtensor_longtime}.
\begin{figure}[htbp]
    \centering
    \subfloat[Example 1]{
        \includegraphics[width=0.235\textwidth]{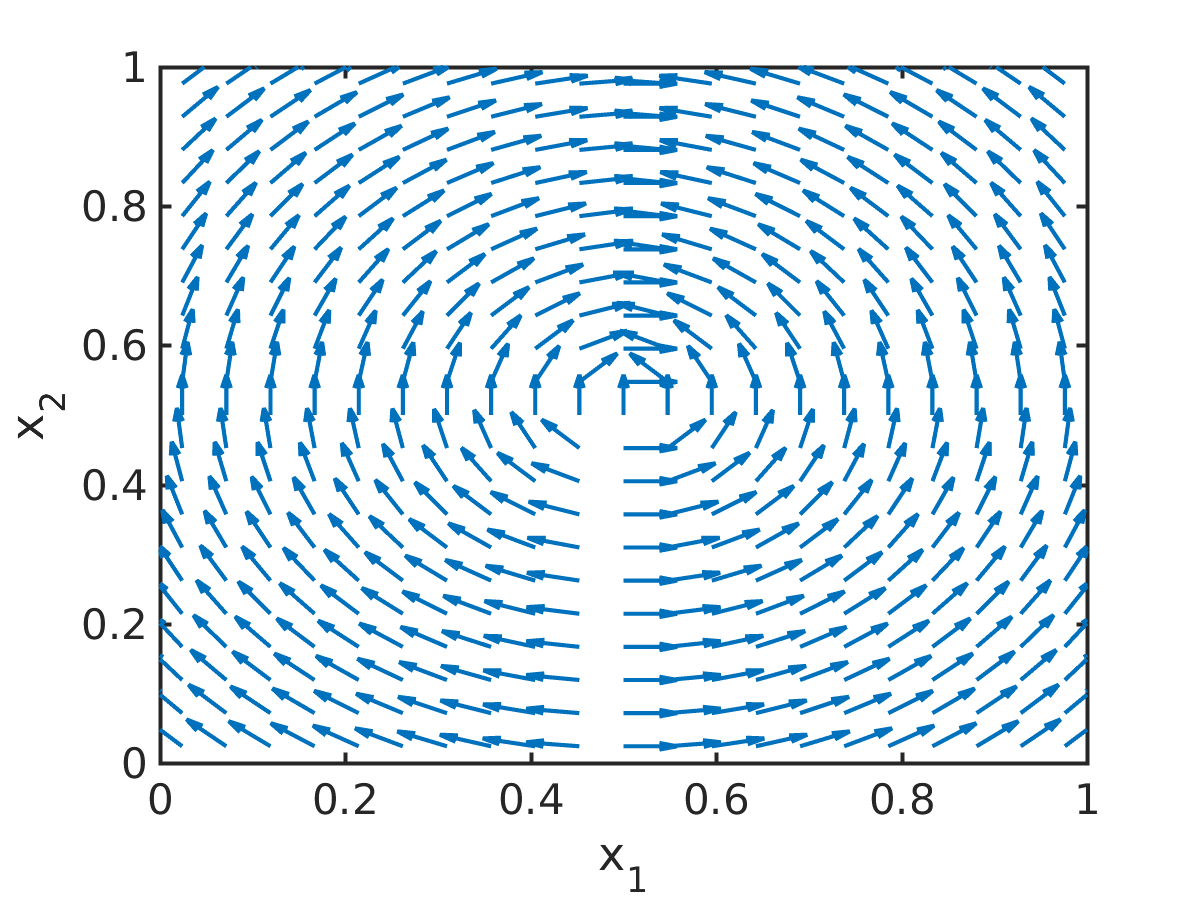}
        \includegraphics[width=0.235\textwidth]{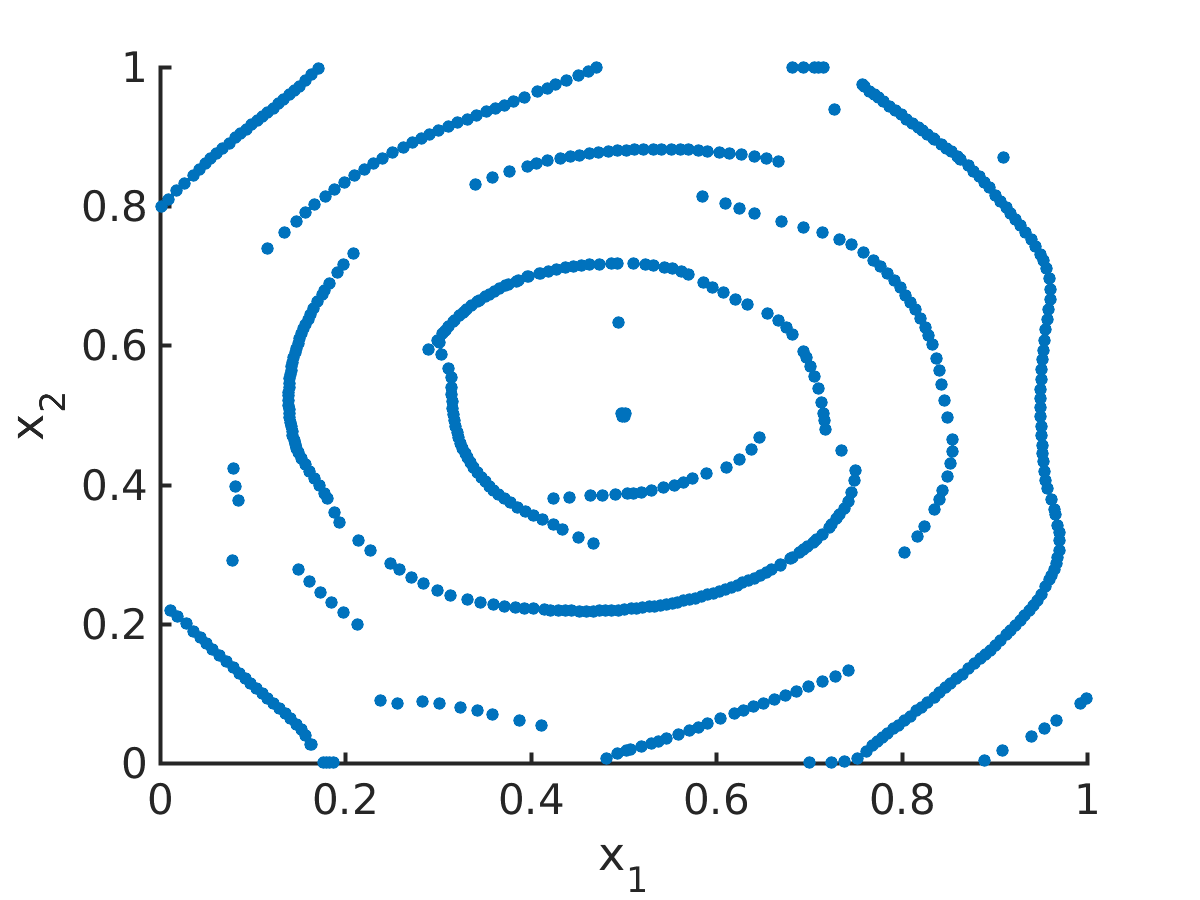}}
\subfloat[Example 2]{    \includegraphics[width=0.235\textwidth]{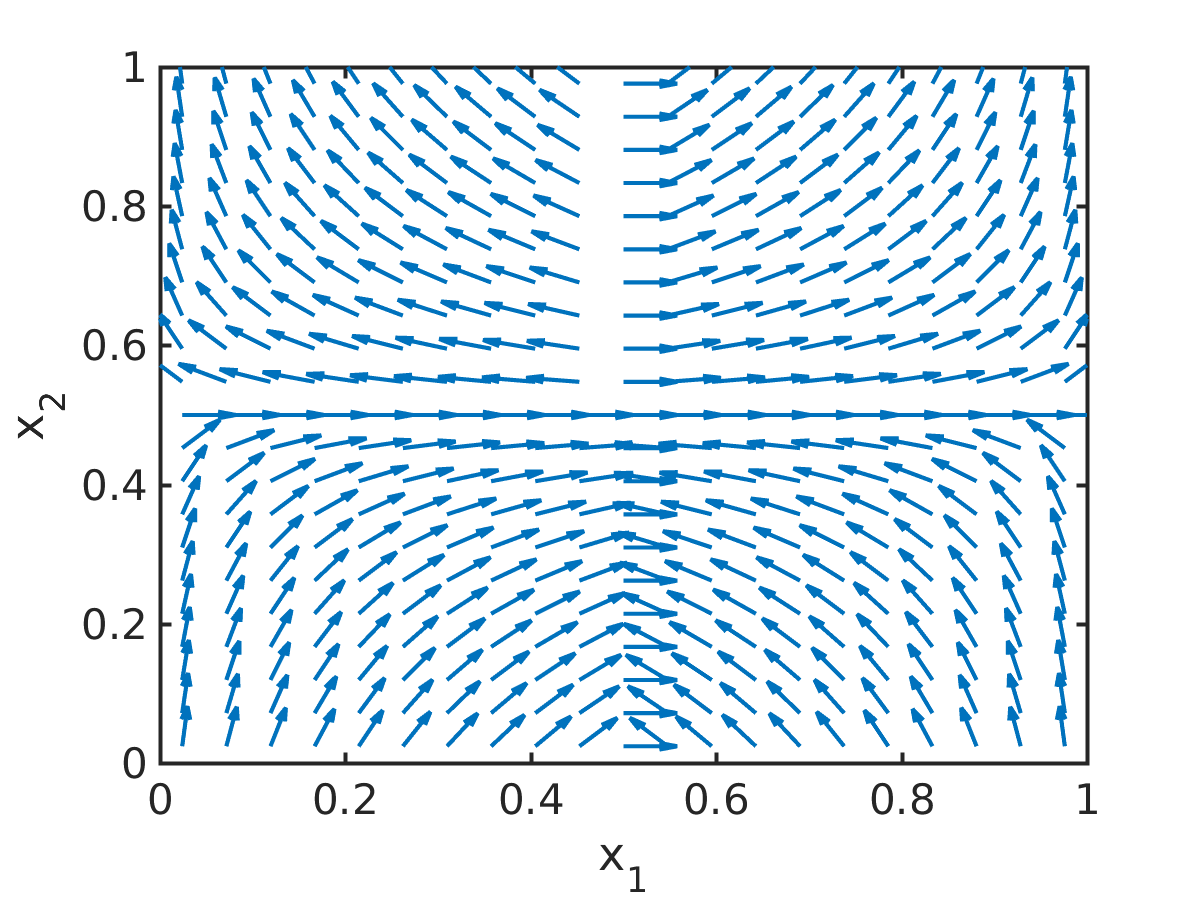}
        \includegraphics[width=0.235\textwidth]{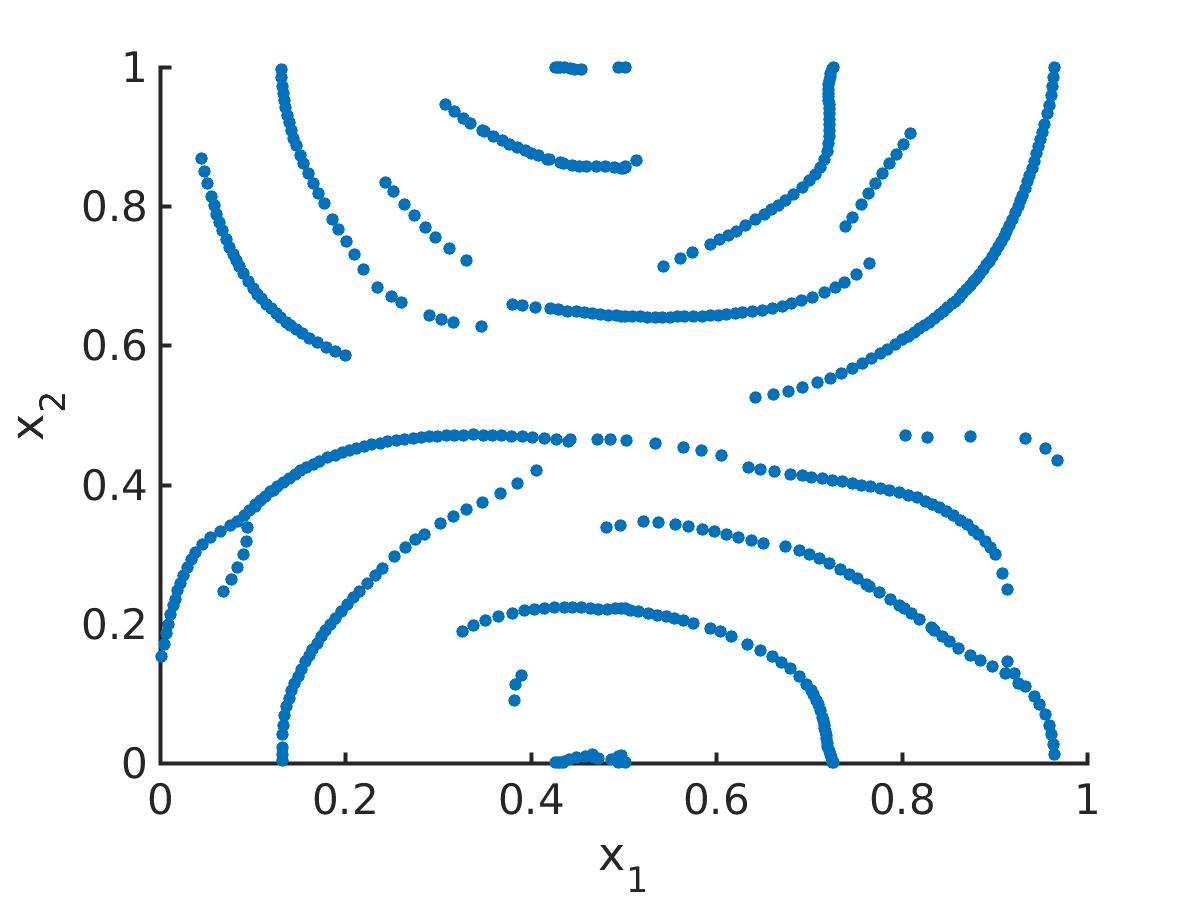} }
        
  \centering      
\subfloat[Example 3]{     \includegraphics[width=0.235\textwidth]{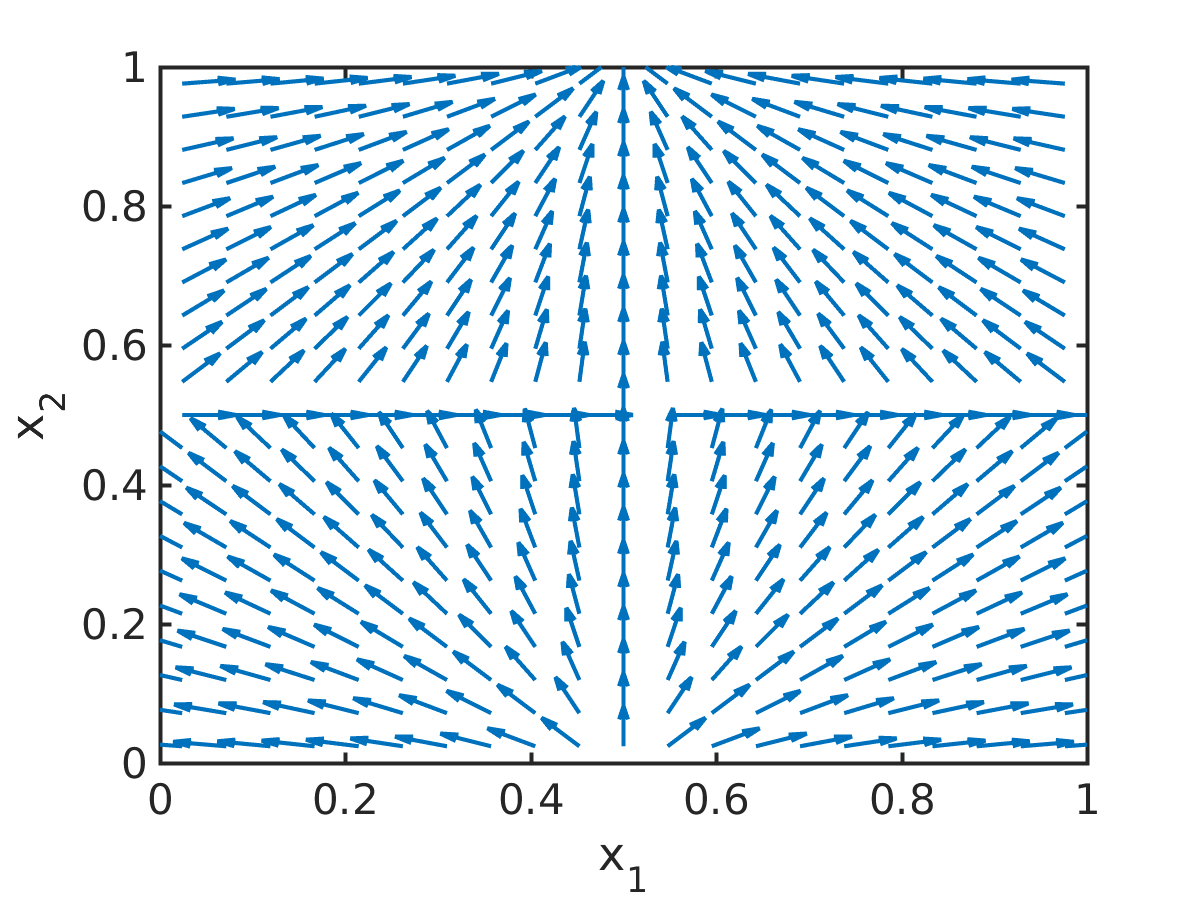}
        \includegraphics[width=0.235\textwidth]{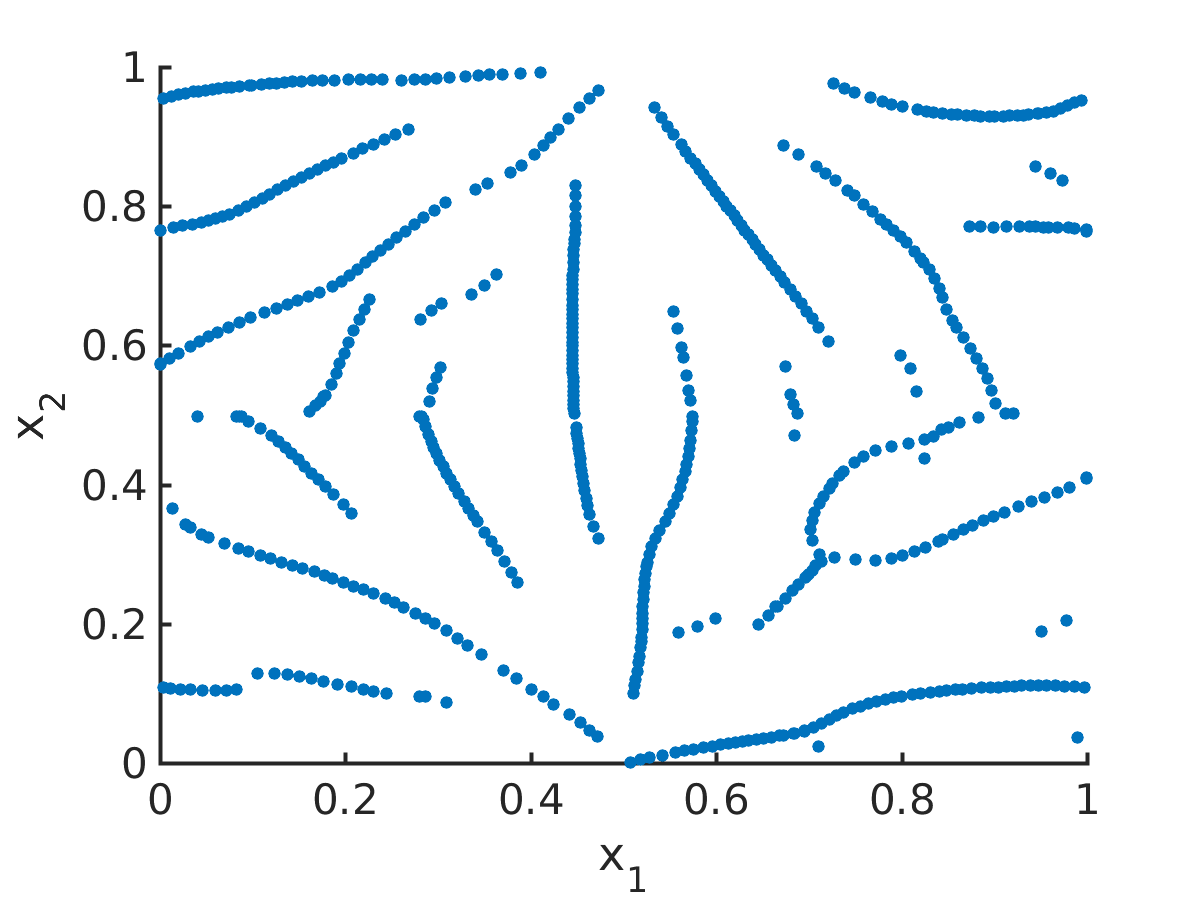}}
    \subfloat[Example 4]{     \includegraphics[width=0.235\textwidth]{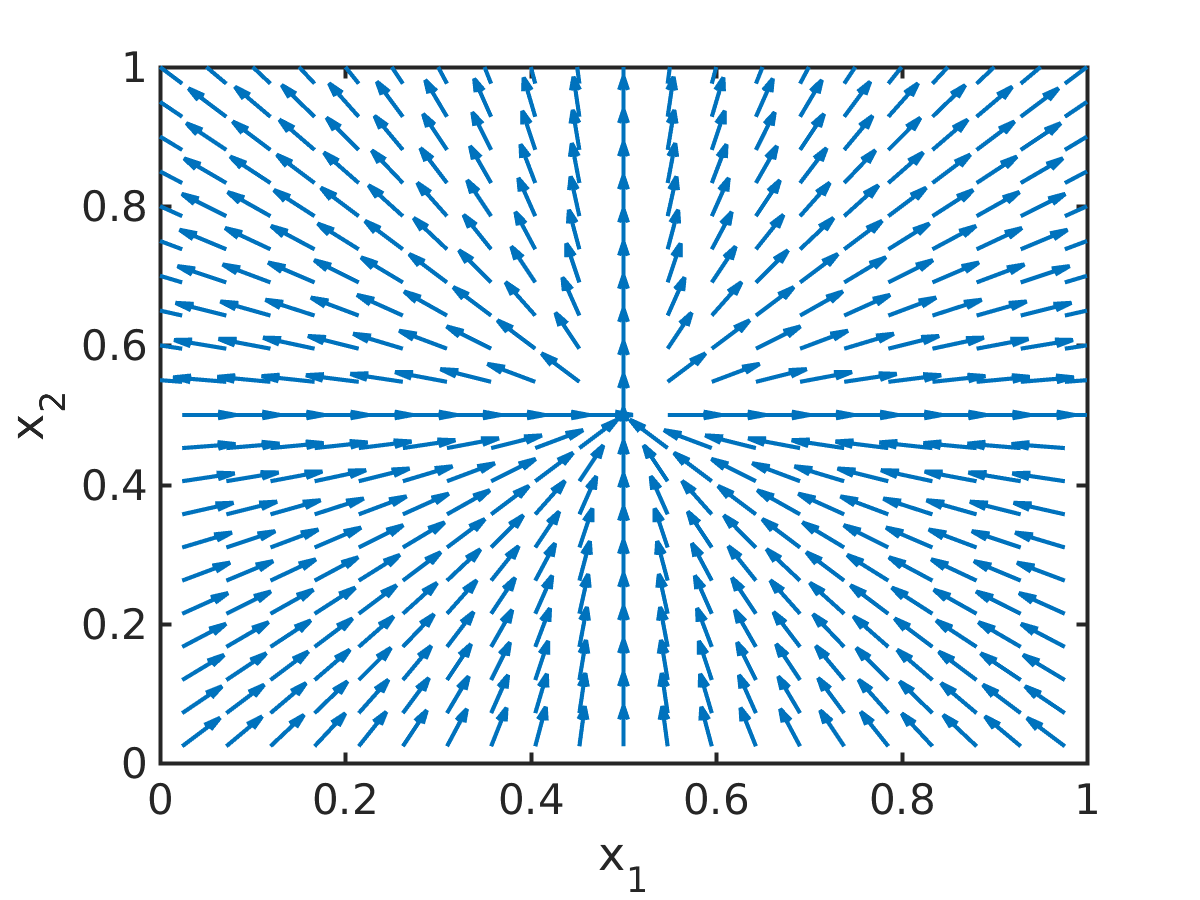}
        \includegraphics[width=0.235\textwidth]{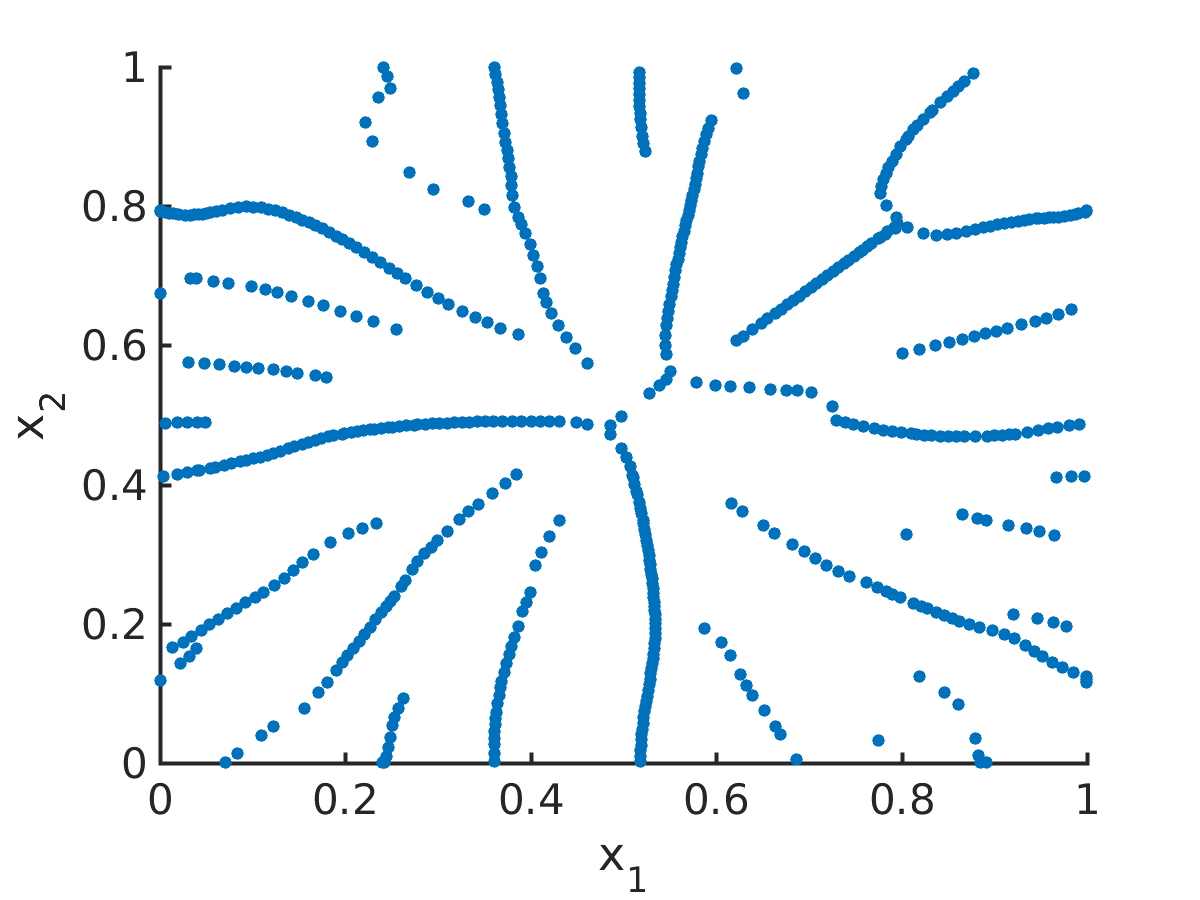}}
    \caption{Different non-homogeneous tensor fields $T=T(x)$ given by $s=s(x)$ (left) and the numerical solution to the K\"{u}cken-Champod model \eqref{eq:particlemodel}  at time $t=40000$  for  $\chi=0.2$,    $T=T(x)$ and randomly uniformly distributed initial data (right)}\label{fig:numericalsol_nonhomtensor}
\end{figure} 

\begin{figure}[htbp]
    \centering
\begin{minipage}{0.99\textwidth}
\centering
\subfloat[$s=s(x)$]{        \includegraphics[width=0.24\textwidth]{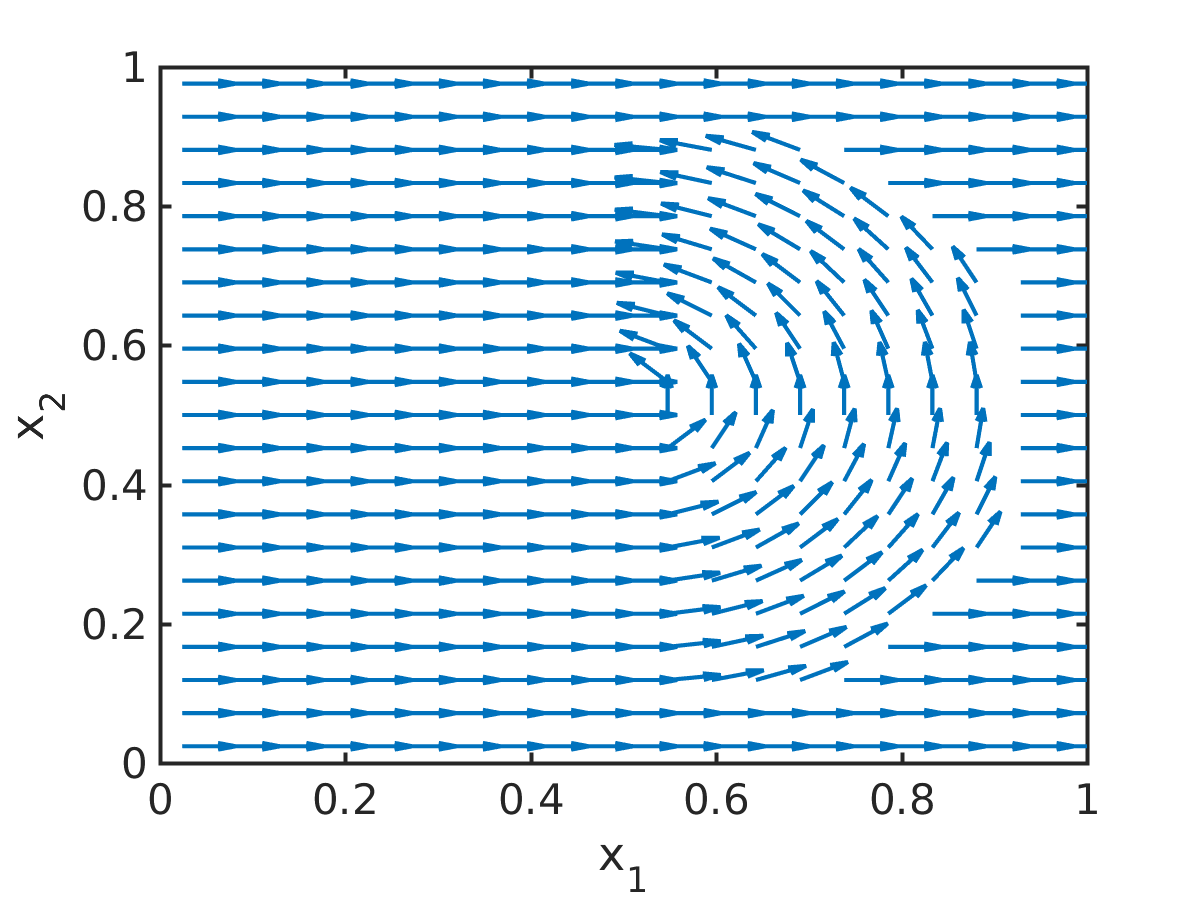}}
\subfloat[$t=40000$]{        \includegraphics[width=0.24\textwidth]{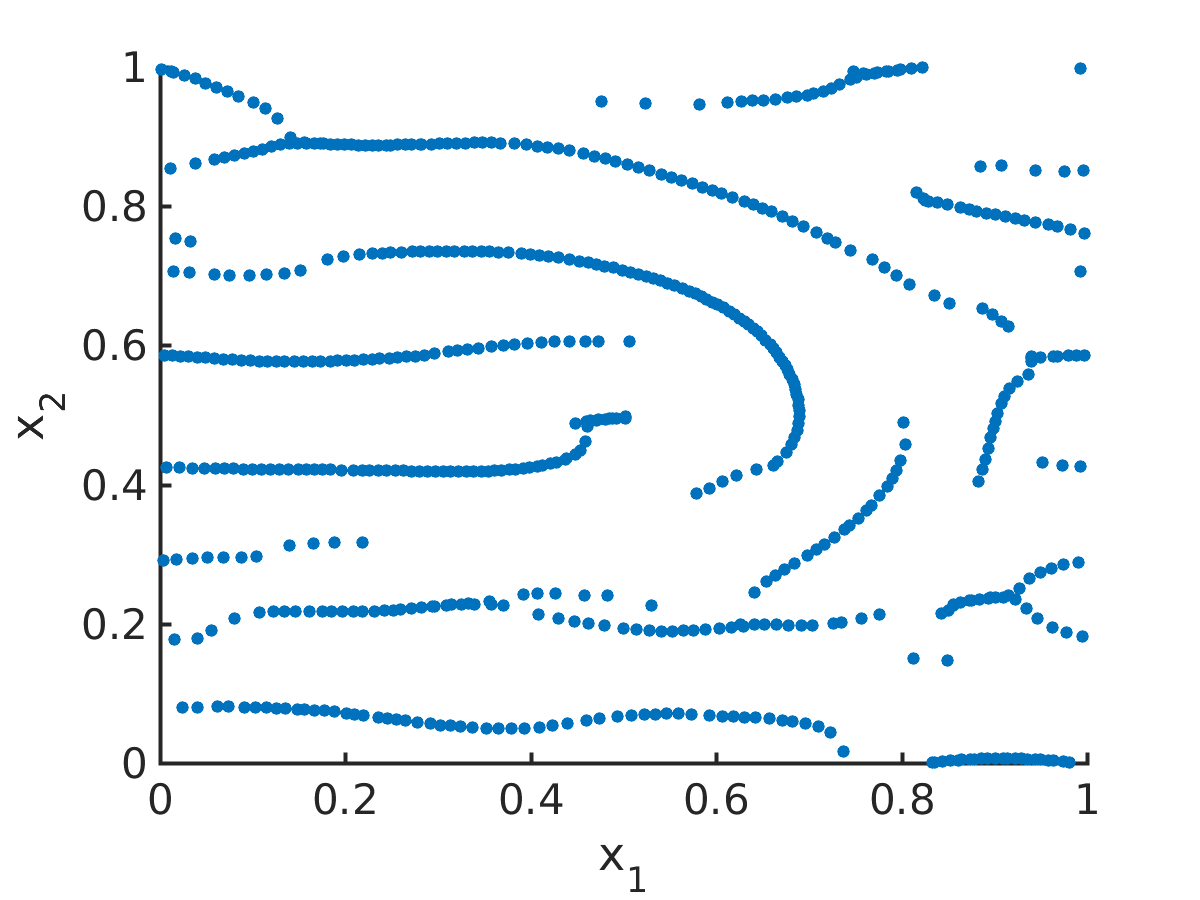}}
\subfloat[$t=200000$]{    \includegraphics[width=0.24\textwidth]{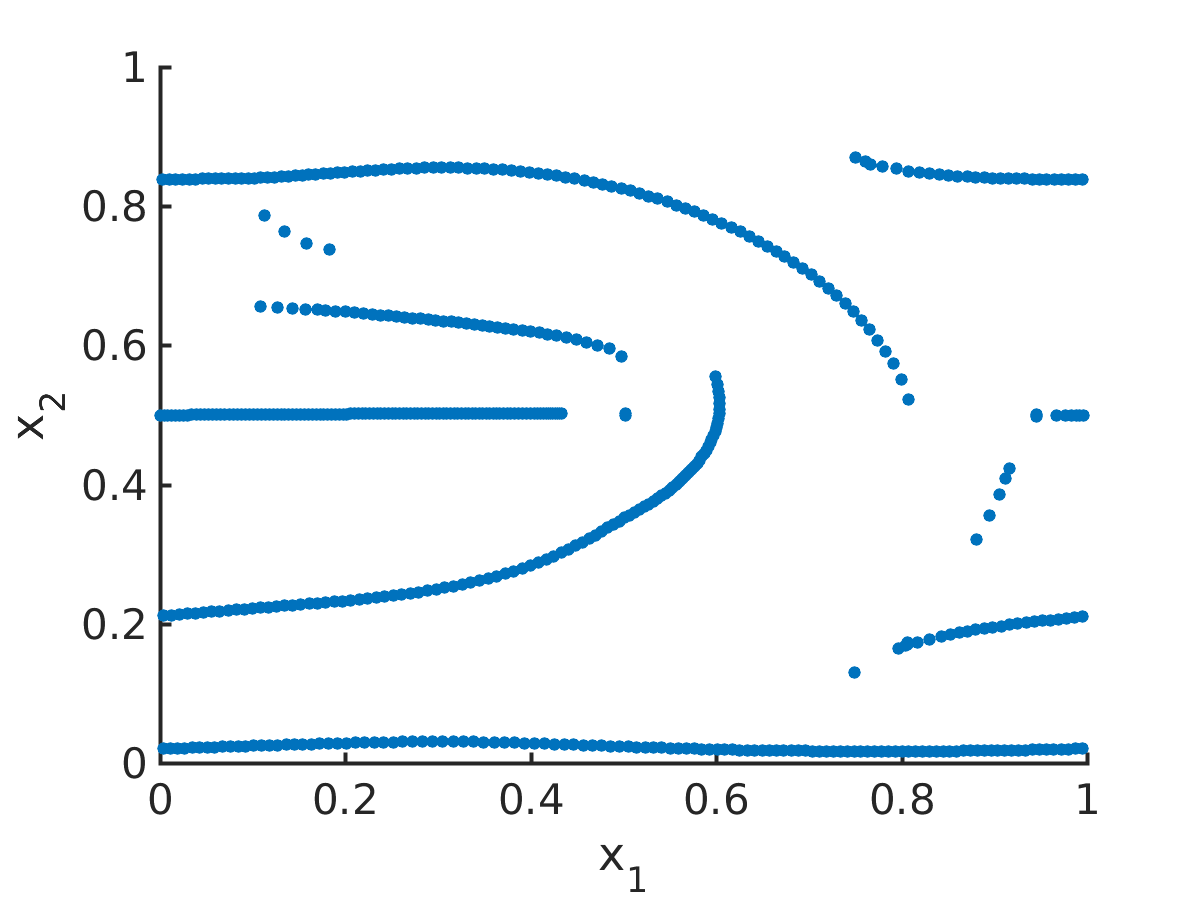}}
\subfloat[$t=400000$]{    \includegraphics[width=0.24\textwidth]{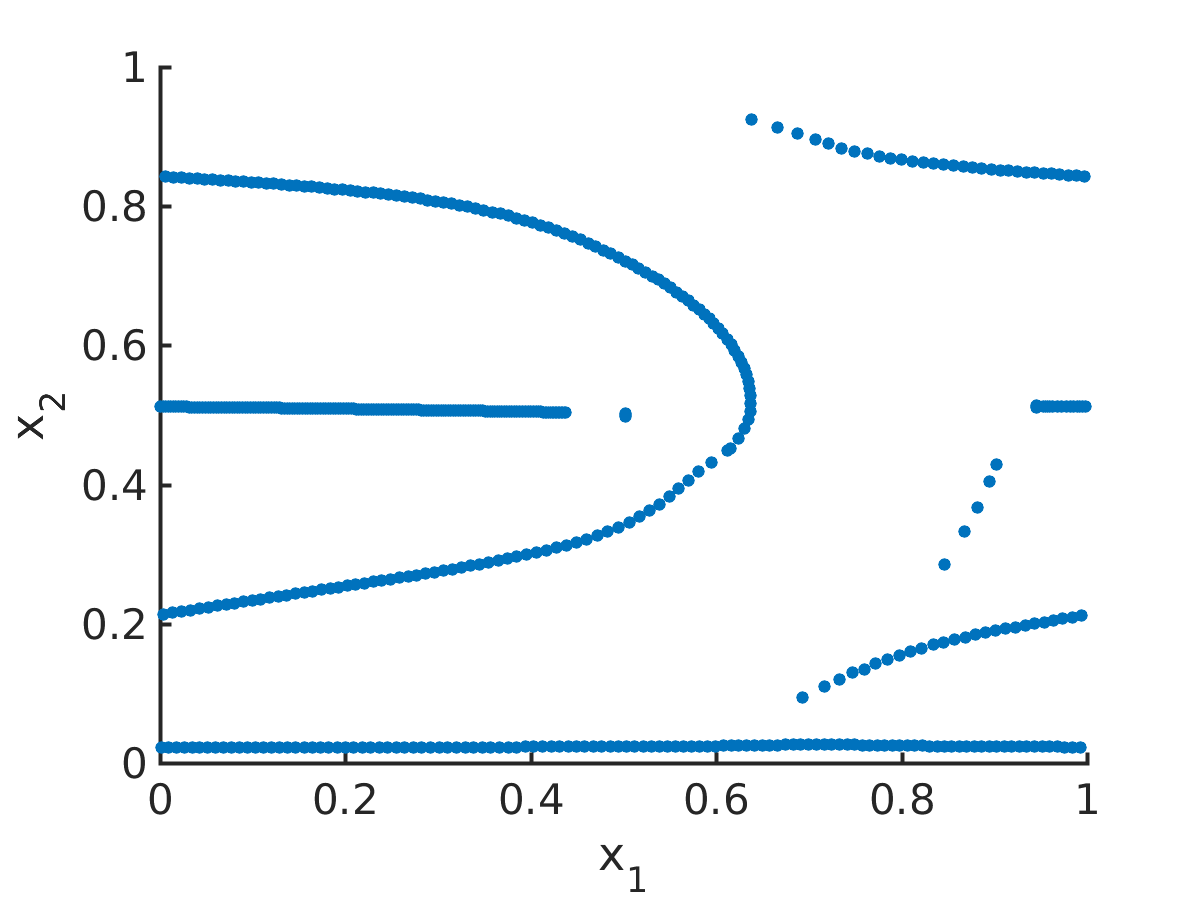}}
        \vspace{3mm}
        Example 5
        \end{minipage}
   \centering     
        \begin{minipage}{0.99\textwidth}
        \centering
    \subfloat[$s=s(x)$]{        \includegraphics[width=0.24\textwidth]{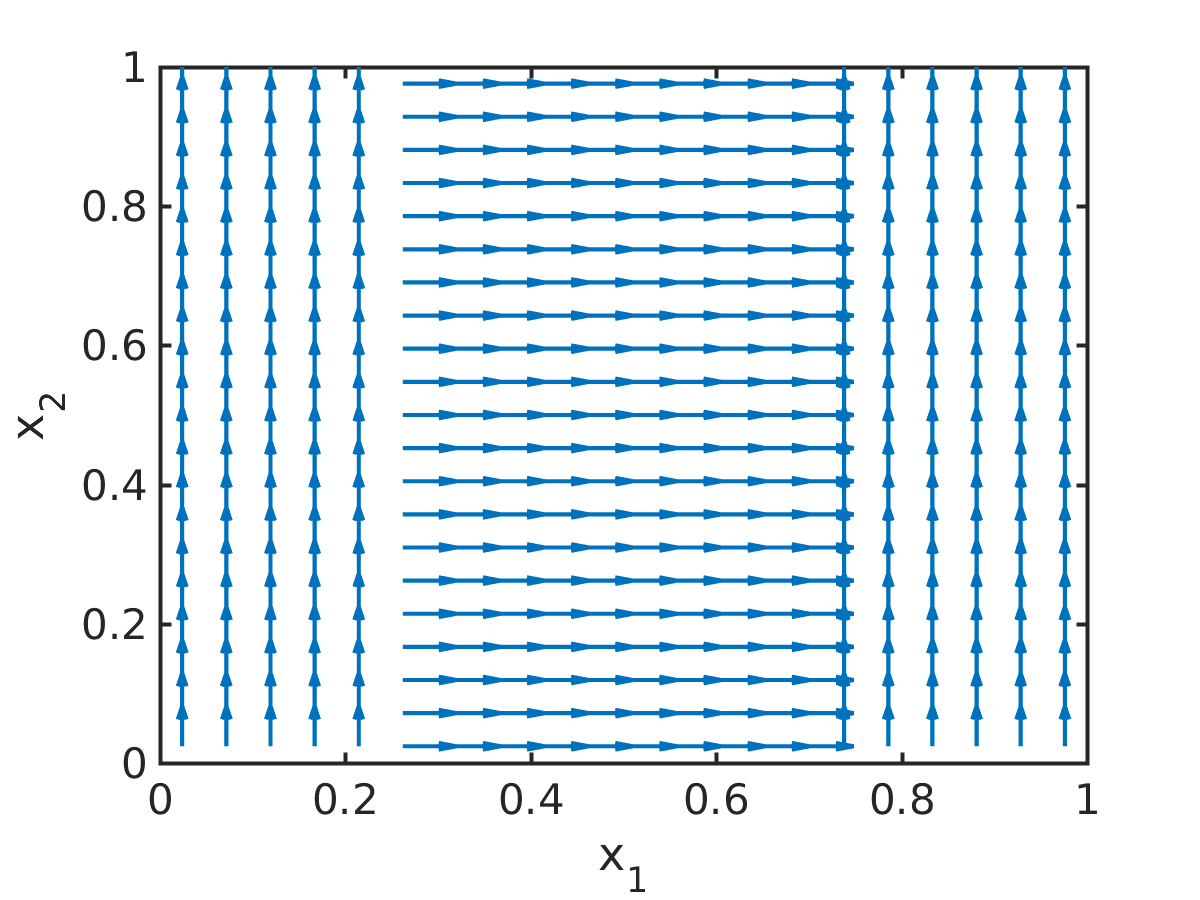}}
        \subfloat[$t=40000$]{       \includegraphics[width=0.24\textwidth]{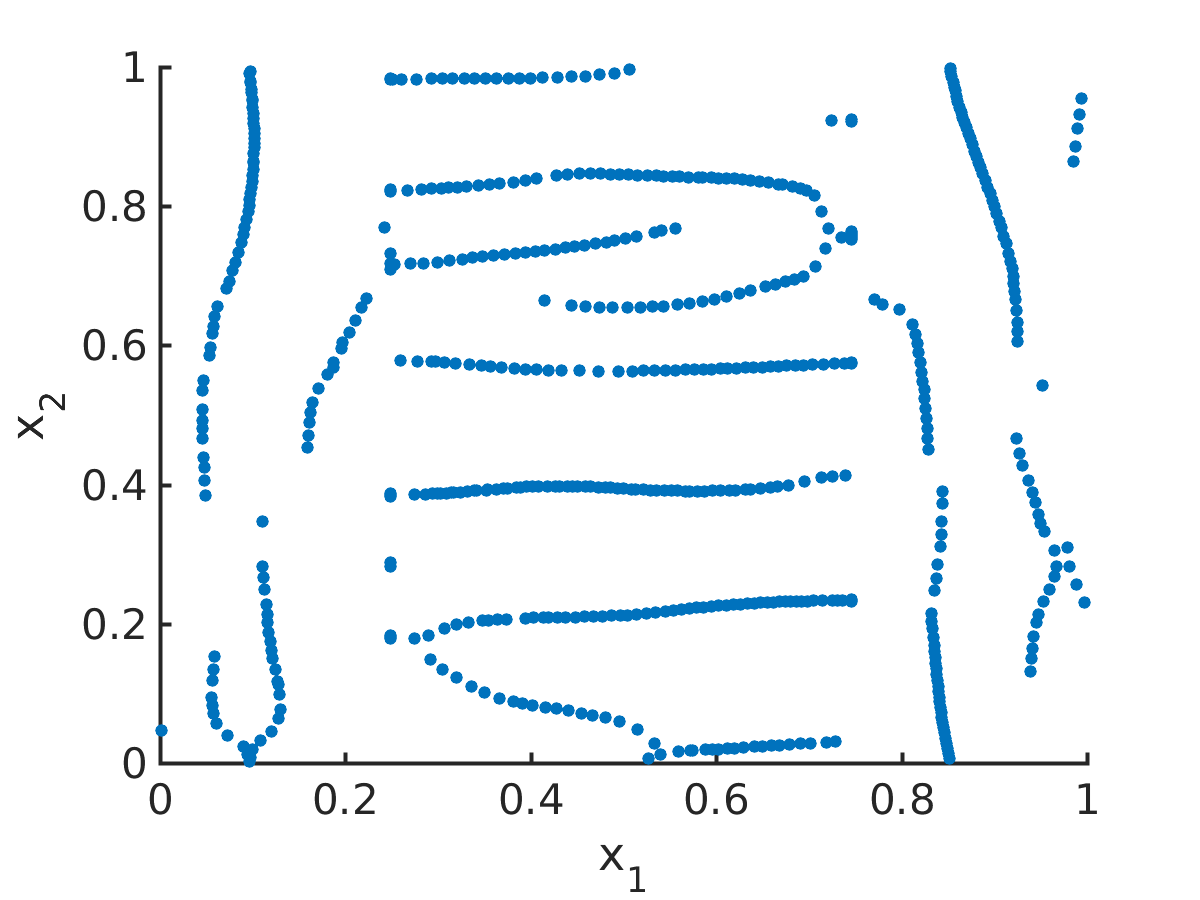}}
        \subfloat[$t=200000$]{        \includegraphics[width=0.24\textwidth]{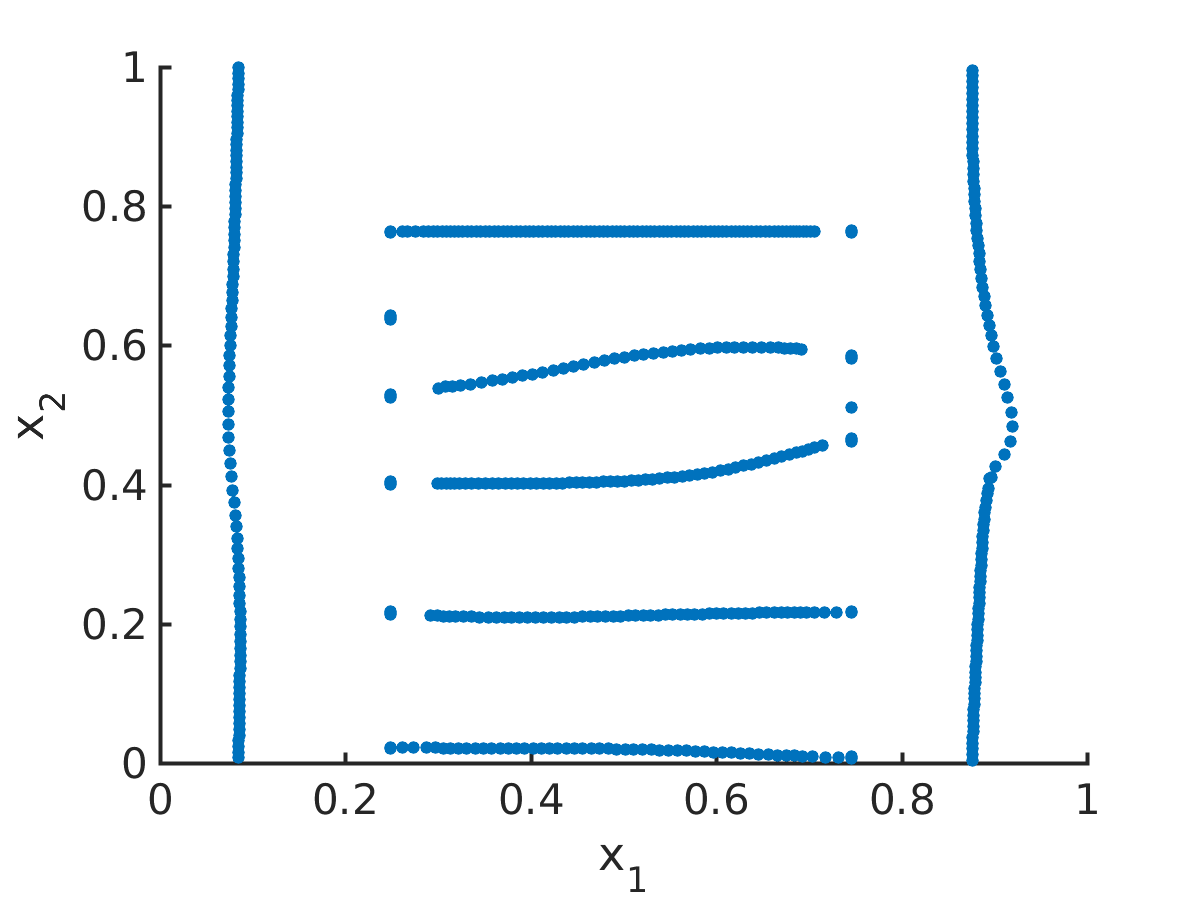}}
        \subfloat[$t=400000$]{     \includegraphics[width=0.24\textwidth]{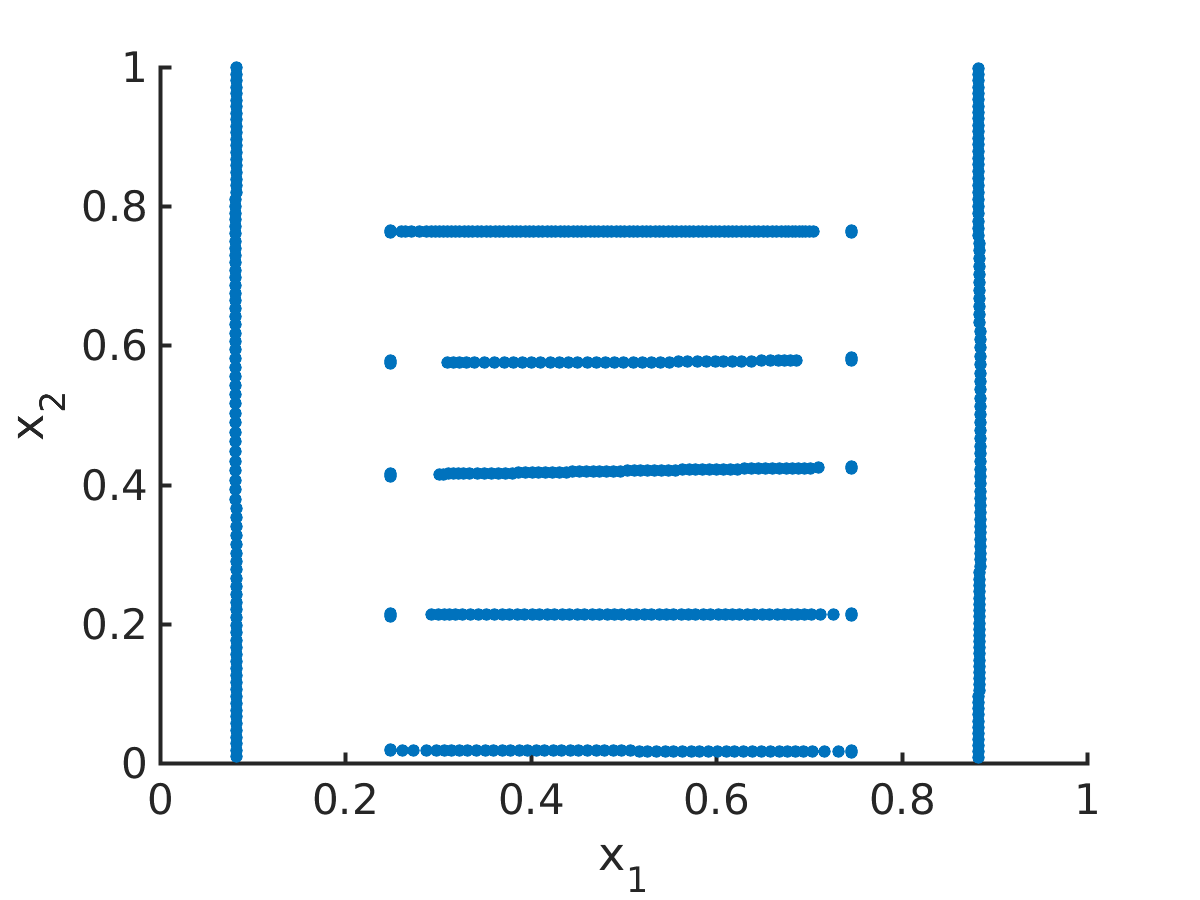}}
\vspace{3mm}       
        Example 6  
        \end{minipage}

    \caption{Different non-homogeneous tensor fields $T=T(x)$ given by $s=s(x)$ (left) and the numerical solution to the K\"{u}cken-Champod model \eqref{eq:particlemodel}  at different times $t$ for $\chi=0.2$,    $T=T(x)$, $N=600$ and randomly uniformly distributed initial data (right)}\label{fig:numericalsol_nonhomtensor_longtime}
\end{figure}

\subsection{Discussion of the numerical results}
In this section, we study the existence of equilibria and their stability of the K\"{u}cken-Champod model \eqref{eq:particlemodel} for  the spatially homogeneous tensor field $T=\chi s \otimes s+l\otimes l$ with $l=(1,0)$ and $s=(0,1)$  and compare them with the numerical results.

\subsubsection{Ellipse}
As outlined  in Section \ref{sec:interpretationforce} the anisotropic forces for $\chi\in [0,1)$ lead to  an additional advection along the vertical axis compared to the horizontal axis  for the given tensor field $T$. Hence, possible stationary ellipse patterns are stretched  along the vertical axis for $\chi\in [0,1)$. Besides, this advection  leads to accumulations within the ellipse pattern, i.e. the distances of the particles are much longer along the vertical lines (e.g. at the left or right side of the ellipse) than along the horizontal lines (e.g. at the top or bottom of the ellipse). As in Section \ref{sec:interpretationsteadystates} we denote the length of the minor and major axis of the ellipse state by $R$ and $R+r$, respectively. 

First, we consider ring patterns of radius $R>0$. We identify $\R^2$ with $\C$ and consider the ansatz 
\begin{align}\label{eq:ringparticleansatz}
\bar{x}_k=\bar{x}_k(R)=x_c+R\exp\bl \frac{2\pi i k}{N}\br,\quad k=0,\ldots,N-1
\end{align}
with center of mass $x_c$, i.e. the particles are uniformly distributed on a ring of radius $R$ with center $x_c$. The radius $R>0$ has to be determined such that the ansatz functions $\bar{x}_j=\bar{x}_j(R)$ satisfy
\begin{align}\label{eq:ringstabilityequality}
\sum_{\substack{k=1\\k\neq j}}^N F(\bar{x}_k(R)-\bar{x}_j(R),T)=0
\end{align}
for all $j=0,\ldots,N-1$. Denoting the  left-hand side of \eqref{eq:ringstabilityequality}   by $G_j(R)$, then $G_j(R)$ is highly nonlinear and zeros of $G_j$ can only be determined  numerically.   By symmetry it is sufficient to determine the zeros of $G_0$ for $\chi=1$. Since  $\Im G_0(R)=0$ for all $R>0$  by the definition of $F$ the condition simplifies to finding $R>0$  such that $\Re G_0(R)=0$.
Using Newton's algorithm the unique nontrivial zero of $\Re G_0$ can be computed as $\bar{R}\approx 0.0017$  for the forces \eqref{eq:repulsionforcemodel} and \eqref{eq:attractionforcemodel}  in the K\"{u}cken-Champod model \eqref{eq:particlemodel} with parameter values from \eqref{eq:parametervaluesRepulsionAttraction}, $N=600$ and a fixed center of mass $x_c$. Hence, given $x_c$ \eqref{eq:ringparticleansatz} with radius   $\bar{R}$ is the unique ring equilibrium for $\chi=1$ and $\bar{R}$ coincides with the radius of the numerically obtained ring equilibrium in Section \ref{sec:numericsevolution}. Based on a linearized stability analysis \cite{Turing} one can show numerically   that the ring pattern is stable for $\chi=1$ for the forces in  the K\"{u}cken-Champod model for parameters in \eqref{eq:parametervaluesRepulsionAttraction} and $N=1200$. Since $\Re G_j$ is independent of $\chi$ with unique zero $\bar{R}$ and $\chi f_A\leq 0$, this implies that there exists no $R>0$ such that $\Im G_j(\bar{R})= 0$ for all $j=0,\ldots,N$ for any $\chi\in[0,1)$, i.e. the ring solution \eqref{eq:ringparticleansatz} is no equilibrium for $\chi\in[0,1)$ and any $R>0$. This  is consistent with the analysis of the mean-field PDE \eqref{eq:macroscopiceq} in Section \ref{sec:analysis} and with the numerical results in Section \ref{sec:numericalresults}.

For the  general case of an ellipse where $r\geq 0$ we identify $\R^2$ with $\C$ and regard the equiangular ansatz 
\begin{align}\label{eq:ellipseparticleansatz}
\bar{x}_k=\bar{x}_k(r,R)=x_c+R\cos\bl \frac{2\pi k}{N}\br+i(R+r)\sin\bl \frac{2\pi k}{N}\br,\quad k=0,\ldots,N-1, 
\end{align}
where the distances of the particles are longer along vertical than along horizontal lines. 
An ellipse equilibrium has to satisfy \begin{align}\label{eq:ellipsestabilityequality}
\sum_{\substack{k=1\\k\neq j}}^N F(\bar{x}_k(R,r)-\bar{x}_j(R,r),T)=0
\end{align}
for all $j=0,\ldots,N-1$. Tuples $(R,r)$ such that \eqref{eq:ellipseparticleansatz} is a possible equilibria to \eqref{eq:particlemodel} can be determined numerically from $\Re G_0(R,r)=0$, where $G_j((R,r))$ for $j\in\{0,\ldots,N-1\}$ denotes the left-hand side of \eqref{eq:ellipsestabilityequality}. For the force coefficients \eqref{eq:repulsionforcemodel} and \eqref{eq:attractionforcemodel} in the K\"{u}cken-Champod model for parameter values \eqref{eq:parametervaluesRepulsionAttraction} and $N=600$, the condition in \eqref{eq:ellipsestabilityequality} implies that the larger $r$ the smaller $R$, i.e. as $r$ increases the ring of radius $R$ evolves into an ellipse whose major axis of length $2(R+r)$ gets longer and whose minor axis of length $2R$ gets shorter as $r$ increases. The numerically obtained tuples $(R,r)$ are shown  in Figure \ref{fig:rvsR}. Besides, it follows from plugging the definition of the total force for spatially homogeneous tensor fields  into \eqref{eq:ellipsestabilityequality} that each tuple $(R,r)$ can be associated to an equilibrium for at most one value of $\chi$. Further note that by Section \ref{sec:interpretationforce} the additional advection along the vertical axis is the stronger the smaller the value of $\chi$, implying that $r$ increases as $\chi$ decreases. Hence, we can conclude that for a given value of $\chi$ there exists at most one tuple $(R,r)$ such that the ansatz \eqref{eq:ellipseparticleansatz} is an equilibrium to \eqref{eq:particlemodel}. This can also be justified by evaluating $\Im G_{N/4}(R,r)$ as a function of radius pairs $(R,r)$ for fixed values of $\chi$ for $N=600$ particles. The eccentricity $e=\sqrt{1-(R/(R+r))^2}$ of the stationary ellipse pattern as a function of the parameter $\chi$ is shown in Figure \ref{fig:eccentricity}. Note that these observations are consistent with the numerical results in Section \ref{sec:numericalresults}. Further note that the shape of the relation between $R$ and $r$ as well as the eccentricity curve in Figures \ref{fig:rvsR} and \ref{fig:eccentricity} is similar to the ones in the continuous case, shown in Figures \ref{fig:rvsRcontinuous} and \ref{fig:eccentricitycontinuous}. However, there are small differences between the radius pairs for the discrete and the continuous case which is due to the additional functional determinant that has to be considered if the corresponding integrals in \eqref{eq:condEllipsesteadystate1} and \eqref{eq:condEllipsesteadystate2} are discretized.
\begin{figure}[htbp]
    \centering
    \subfloat[Tuples $(R,r)$ ]{\includegraphics[width=0.45\textwidth]{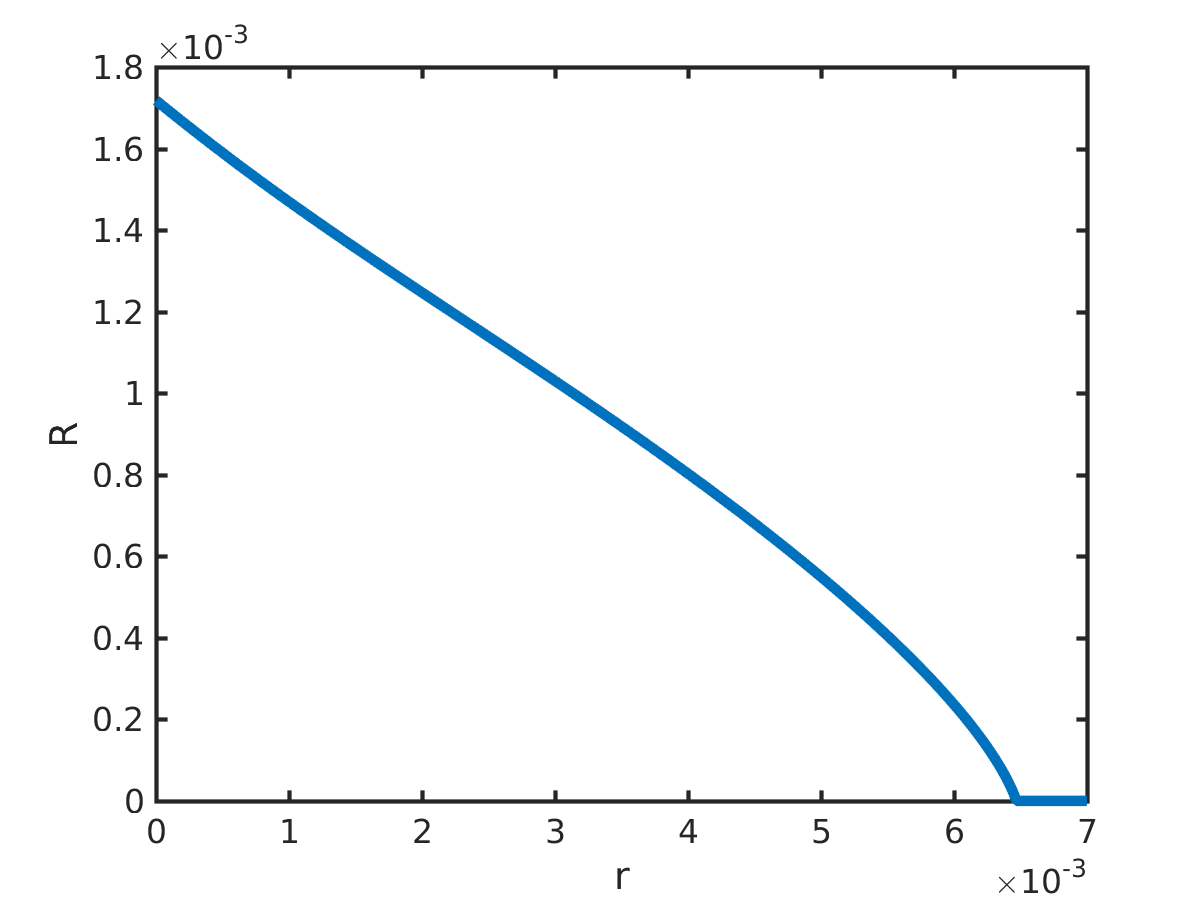}\label{fig:rvsR}}
\subfloat[$e=e(\chi)$] {\includegraphics[width=0.45\textwidth]{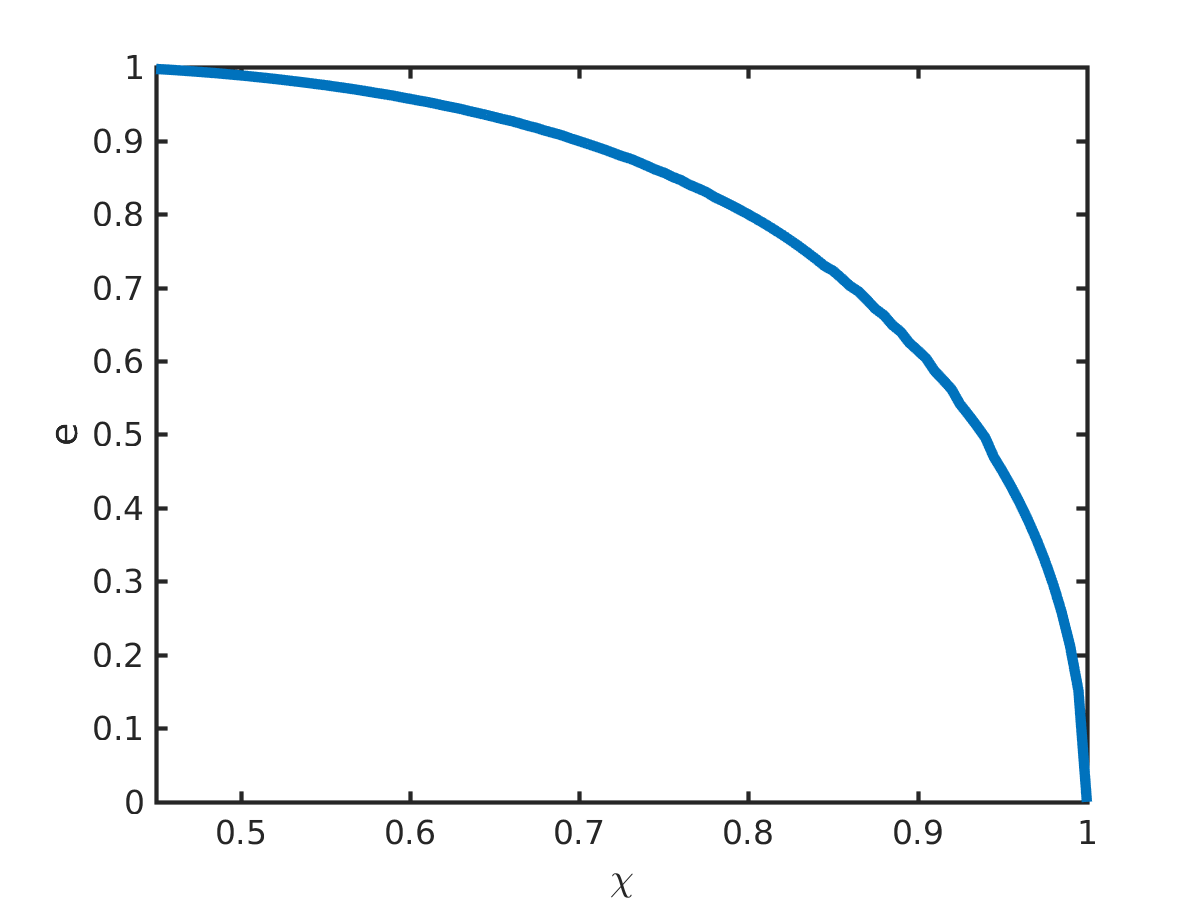}\label{fig:eccentricity}}
    \caption{Tuples $(R,r)$ for stationary ellipse patterns to \eqref{eq:particlemodel} with ansatz \eqref{eq:ellipseparticleansatz} and eccentricity $e$ as a function of $\chi$ for $N=600$ and the forces in the K\"ucken-Champod model for parameter values in \eqref{eq:parametervaluesRepulsionAttraction} }
\end{figure}

\subsubsection{Single straight vertical line}\label{sec:singlestraightvericalstationary}
Because of  the observations in Section \ref{sec:interpretationforce} a natural choice for line patterns are  vertical lines. Identifying $\R^2$ with $\C$ results in the ansatz
\begin{align}\label{eq:straightlineansatz}
\bar{x}_k=x_c+i\frac{2k-1}{2N},\quad k=0,\ldots,N-1, 
\end{align}
for a single straight vertical line with the center of mass $x_c$. One can easily see that ansatz \eqref{eq:straightlineansatz} defines an equilibrium of \eqref{eq:particlemodel} for all values $\chi\in [0,1]$ where the minimum image criterion  is crucial to guarantee that \eqref{eq:straightlineansatz} is an equilibrium for even values of $N$. Based on a linearized stability analysis \cite{Turing} one can show that \eqref{eq:straightlineansatz}  is a stable equilibrium of \eqref{eq:particlemodel} for $N=1200$ for $\chi\in[0,0.27]$ which is consistent with the numerical results in Section \ref{sec:numericalresults}.

\section*{Acknowledgments}
MB acknowledges support by ERC via Grant EU FP 7 - ERC Consolidator Grant 615216 LifeInverse and by the German Science Foundation DFG via EXC 1003 Cells in Motion Cluster of Excellence, M\"{u}nster, Germany. BD has been supported by the Leverhulme Trust research project grant `Novel discretizations for higher-order nonlinear PDE' (RPG-2015-69). LMK was supported by the UK Engineering and Physical Sciences Research Council (EPSRC) grant
EP/L016516/1. CBS acknowledges support from Leverhulme Trust project on Breaking the non-convexity barrier, EPSRC grant Nr. EP/M00483X/1, the EPSRC Centre Nr. EP/N014588/1 and the Cantab Capital Institute for the Mathematics of Information. 
The authors would like to thank Carsten Gottschlich and Stephan Huckemann for introducing them to the K\"{u}cken-Champod model and for very useful discussions on the dynamics required for simulating fingerprints.

\appendix

\section{Detailed computations of Section \ref{sec:interpretationvectorfield}}\label{sec:appendix_computation}

Let $\tilde{T}=\chi \tilde{s} \otimes \tilde{s}+\tilde{l}\otimes \tilde{l}$ denote a spatially homogeneous tensor field for  orthonormal vectors $\tilde{l}, \tilde{s}\in \R^2$. Given  $l=(1,0)$, $s=(0,1)$ and angle of rotation $\theta$ in \eqref{eq:anglerotationproof}, then $\tilde{T}=R_{\theta}T R_{\theta}^T$ with  $T=\chi s \otimes s+l\otimes l$ and rotation matrix $R_{\theta}$ in \eqref{eq:anglerotationproof}. 

Let $x_j=x_j(t),~j=1,\ldots,N$, denote the solution to the microscopic model \eqref{eq:particlemodel} on $\R^2$ for the tensor field $T$ and define
\begin{align*}
\tilde{x}_j(t)=x_c+R_{\theta}(x_j(t)-x_c),\quad j=1,\ldots,N
\end{align*}
where $x_c$ denotes the center of mass. Then, $\tilde{x}_j=\tilde{x}_j(t),j=1,\ldots,N$, is a solution to the microscopic model \eqref{eq:particlemodel} on $\R^2$ for the tensor field $\tilde{T}$. Besides, given   an equilibrium $\bar{x}_j,~j=1,\ldots,N$, to \eqref{eq:particlemodel} on $\R^2$ for the tensor field $T$, then
\begin{align*}
\bar{\tilde{x}}_j=x_c+R_{\theta}(\bar{x}_j-x_c),\quad j=1,\ldots,N,
\end{align*}
is an equilibrium  to \eqref{eq:particlemodel} on $\R^2$ for the tensor field $\tilde{T}$.

We  show that $\tilde{x}_j,~j=1,\ldots,N$, solves \eqref{eq:particlemodel} for the tensor field $\tilde{T}$. 
Since $x_j,~j=1,\ldots,N,$ solves \eqref{eq:particlemodel} for the tensor field $T$, we have 
\begin{align*}
\frac{\di x_j}{\di t}=\sum_{k\neq j} f_A(|d|)\left[\chi \bl s\cdot d\br s+\bl l\cdot d\br l\right]+f_R(|d|)d
\end{align*}
for all $j=1,\ldots,N$  where $d(x_j,x_k)=x_j-x_k$.
 Note that $\tilde{x}_j-\tilde{x}_k=R_{\theta}(x_j-x_k)$ and $|\tilde{x}_j-\tilde{x}_k|=|x_j-x_k|$. Using \eqref{eq:proofvectortransform} as well as the fact that $R_{\theta}$ is an orthogonal matrix we get
 \begin{align*}
 \chi \bl \tilde{s}\cdot (\tilde{x}_j-\tilde{x}_k)\br \tilde{s}+\bl \tilde{l}\cdot (\tilde{x}_j-\tilde{x}_k)\br \tilde{l}= R_{\theta}\left[\chi \bl s\cdot (x_j-x_k)\br s+\bl l\cdot (x_j-x_k)\br l\right].
 \end{align*}
Setting $\tilde{d}(\tilde{x}_j,\tilde{x}_k)=\tilde{x}_j-\tilde{x}_k$ this implies
\begin{align*}
\frac{\di \tilde{x}_j}{\di t}=\sum_{k\neq j} f_A\bl\left|\tilde{d}\right|\br\left[ \chi \bl \tilde{s}\cdot \tilde{d}\br \tilde{s}+\bl \tilde{l}\cdot \tilde{d}\br \tilde{l}\right]+f_R\bl\left|\tilde{d}\right|\br\tilde{d}
\end{align*}
 for all $j=1,\ldots,N$, i.e. $\tilde{x}_j,~j=1,\ldots,N$, solves \eqref{eq:particlemodel} for the tensor field $\tilde{T}$. Similarly, one can show that $\bar{\tilde{x}}_j$ is an equilibrium  to \eqref{eq:particlemodel}  for the tensor field $\tilde{T}$, given that $\bar{x}_j,~j=1,\ldots,N$, is an equilibrium  to \eqref{eq:particlemodel} for the tensor field $T$.

We turn to equilibria of the mean-field equation \eqref{eq:macroscopiceq} for spatially homogeneous tensor fields now. Let $\rho=\rho(\di x)$ denote an equilibrium  state to the mean-field PDE \eqref{eq:macroscopiceq} on $\R^2$ for the tensor field $T$ and define
\begin{align}\label{eq:rhovectortransformproof}
\tilde{\rho}(x)=\rho\bl x_c+R_{\theta}^{-1}(x-x_c)\br \text{a.e.}
\end{align}
where $x_c$ denotes the center of mass. Then, $\tilde{\rho}$ is an equilibrium state to \eqref{eq:macroscopiceq} for the tensor field $\tilde{T}$.

To show this result note that for $x\in\R^2$ we have 
\begin{align*}
&\bl F(\cdot,T)\ast \tilde{\rho}\br \bl x_c+R_{\theta}(x-x_c)\br\\
&=\int_{\R^2} F( x_c+R_{\theta}(x-x_c)-\bl x_c+R_{\theta}(x-x'_c)\br,{\tilde{T}}) \rho(\di x') \\
&=\int_{\R^2}\left[\vphantom{f_A(|x-x'|)\left[\chi \bl \tilde{s}\cdot (R_{\theta}(x-x'))\br \tilde{s}+\bl \tilde{l}\cdot (R_{\theta}(x-x'))\br \tilde{l}\right]}
f_R(|x-x'|)R_{\theta}(x-x')\right.\\&\qquad\qquad\qquad+\left.f_A(|x-x'|)\left[\chi \bl \tilde{s}\cdot (R_{\theta}(x-x'))\br \tilde{s}+\bl \tilde{l}\cdot (R_{\theta}(x-x'))\br \tilde{l}\right] \right]\rho(\di x')\\
&=R_{\theta}\bl F(\cdot,T)\ast \rho\br  \bl x\br 
\end{align*}
where the first equality follows from \eqref{eq:rhovectortransformproof} and the  substitution rule. The definitions of the repulsion and attraction forces in \eqref{eq:repulsionforce} and \eqref{eq:attractionforce} are used in the  second equality and \eqref{eq:proofvectortransform} is inserted in the third equality. Since  $x\in \text{supp}(\tilde{\rho})$ implies $x\in \text{supp}(\rho(x_c+R_{\theta}^{-1}(\cdot-x_c)))$ and $\bl F(\cdot,T)\ast \tilde{\rho}\br \bl x\br=R_{\theta}\bl F(\cdot,T)\ast \rho\br  \bl x_c+R_{\theta}^{-1}(x-x_c)\br $,  $\tilde{\rho}$ is an equilibrium state to \eqref{eq:macroscopiceq} for the tensor field $\tilde{T}$ provided that $\rho$ is an equilibrium state to \eqref{eq:macroscopiceq} for the tensor field $T$.

\bibliographystyle{plain}
\bibliography{references}
\end{document}